\documentclass[11pt]{amsart}
\usepackage{amscd,amsmath,amssymb,amsfonts,verbatim}
\usepackage[cmtip, all]{xy}
\usepackage[a4paper, left=2.55cm, right=2.55cm, top = 2.2cm, bottom = 2cm ]{geometry}
\usepackage{comment}
\newtheorem{thm}{Theorem}[section]
\newtheorem{prop}[thm]{Proposition}
\newtheorem{lem}[thm]{Lemma}
\newtheorem{cor}[thm]{Corollary}

\theoremstyle{definition}

\newtheorem{defn}[thm]{Definition}
\theoremstyle{remark}
\newtheorem{remk}[thm]{Remark}
\newtheorem{remks}[thm]{Remarks}

\newtheorem{exm}[thm]{Example}
\newtheorem{exms}[thm]{Examples}
\newtheorem{notat}[thm]{Notation}
\numberwithin{equation}{section}

{\hfill$\square$\end{defn}}
{\hfill$\square$\end{remk}}
{\hfill$\square$\end{remks}}
{\hfill$\square$\end{exm}}
{\hfill$\square$\end{exms}}
{\hfill$\square$\end{notat}}

\newcommand{\sB}{{\mathcal B}}

\newcommand{\sZ}{{\mathcal Z}}
\newcommand{\A}{{\mathbb A}}

\newcommand{\G}{{\mathbb G}}

\renewcommand{\P}{{\mathbb P}}

\newcommand{\W}{{\mathbb W}}

\newcommand{\Z}{{\mathbb Z}}

\newcommand{\fm}{{\mathfrak m}}

\newcommand{\CH}{{\rm CH}}

\newcommand{\surj}{\twoheadrightarrow}
\newcommand{\inj}{\hookrightarrow}

\newcommand{\codim}{{\rm codim}}

\newcommand{\Spec}{{\rm Spec \,}}

\newcommand{\id}{{\operatorname{id}}}

\newcommand{\Sch}{{\operatorname{\mathbf{Sch}}}}

\renewcommand{\max}{{\operatorname{\rm max}}}

\newcommand{\Sm}{{\mathbf{Sm}}}

\newcommand{\ds}{{/\kern-3pt/}}

\newcommand{\ess}{\text{\rm{ess}}}

\newcommand{\Gr}{{\text{\rm Gr}}}

\newcommand{\Proj}{{\operatorname{Proj}}}

\newcommand{\tr}{{\operatorname{tr}}}

\newcommand{\TZ}{{\operatorname{Tz}}}
\renewcommand{\TH}{{\operatorname{TCH}}}
\newcommand{\un}{\underline}
\newcommand{\ov}{\overline}

\newcommand{\dgn}{{\operatorname{degn}}}
\renewcommand{\dim}{\text{\rm dim}}

\newcommand{\tuborg}{\left\{\begin{array}{ll}}
\newcommand{\sluttuborg}{\end{array}\right.}

\newcommand{\sfs}{{\rm sfs}}
\newcommand{\fs}{{\rm fs}}

\newcounter{elno}

\newcounter{elno-abc}   
\newenvironment{listabc}{
                         \begin{list}{\alph{elno-abc})
                                     }{\usecounter{elno-abc}}
                      }{
                         \end{list}}
\newcounter{elno-abc-prime}

\begin{document}
\title{A moving lemma for relative $0$-cycles}
\author{Amalendu Krishna, Jinhyun Park}
\address{School of Mathematics, Tata Institute of Fundamental Research, 1 Homi Bhabha Road, Colaba, Mumbai, India}
\email{amal@math.tifr.res.in}
\address{Department of Mathematical Sciences, KAIST, 291 Daehak-ro, Yuseong-gu, Daejeon, 34141, Republic of Korea (South)}
\email{jinhyun@mathsci.kaist.ac.kr; jinhyun@kaist.edu}
  
\keywords{algebraic cycles, moving lemma, higher Chow group, additive higher Chow group, linear projection, Grassmannian}        

\subjclass[2010]{Primary 14C25; Secondary 19E15, 14F42}

\begin{abstract}
We prove a moving lemma for the additive and ordinary higher Chow groups of relative $0$-cycles of regular semi-local $k$-schemes essentially of finite type over an infinite perfect field. From this, we show that the cycle classes can be represented by cycles that possess certain finiteness, surjectivity, and smoothness properties. It plays a key role in showing that the crystalline cohomology of smooth varieties can be expressed in terms of algebraic cycles.
\end{abstract}
\maketitle

\setcounter{tocdepth}{1}

\tableofcontents

\section{Introduction}\label{sec:Intro}
Just as the classical Chow moving lemma played a fundamental role in studies of Chow groups of smooth algebraic varieties over a field, the moving lemma of Bloch \cite{Bl1, Bl2} played a significant role in studies of higher Chow groups of smooth algebraic varieties, i.e., the motivic cohomology. One limitation of those moving lemmas however is that they focus only on the proper intersection properties of the given cycles. Occasionally, the given circumstances require us to know more about the cycles beyond such proper intersection properties. For instance, we often need to know whether the given cycles are finite over the base scheme, and smooth, or, if not, whether they can be moved to such cycles. Such questions require more subtle treatments and may hold under special circumstances only.

The goal of this article is to prove a moving lemma of this sort for higher relative 0-cycles of a regular semi-local scheme essentially of finite type over an infinite perfect field $k$. Here, `essentially of finite type' means it is obtained by localizing a quasi-projective $k$-scheme at a finite set $\Sigma$ of points. Achieving suitable finiteness and regularity of the cycles is the main characteristic of the moving lemma we seek.

In the Introduction, we state the main results, explain the motivation, and give an outline of the article.

\subsection{The \sfs-moving lemma}\label{sec:sfs**}
Let $k$ be an infinite perfect field. Let $R$ be a regular semi-local $k$-algebra essentially of finite type. Let $V = \Spec(R)$ and let $\Sigma$  denote the set of closed points of $V$. Let $\TZ^q(V, \bullet ; m)$ be the non-degenerate additive cycle complex of $V$ in codimension $q \ge 1$ and with modulus $m \ge 1$. Let $\TH^q(V, n; m)$ denote the associated homology groups, called the additive higher Chow groups of $V$ (see \S \ref{sec:additive complex}).

For $n \ge 1$, let $\TZ^n_{\sfs}(V,n;m)$ denote the subgroup of $\sfs$-cycles in $\TZ^n(V,n;m)$ (see \S \ref{sec:SFS*}). Roughly speaking, an \sfs-cycle is an element $\alpha \in \TZ^n(V,n;m)$ such that every irreducible component of $\alpha$ intersects $\Sigma \times F$ properly for every face $F \subset \square^{n-1}_k$, is finite and surjective over an irreducible component of $V$, and the image under every projection $V \times \square^{n-1}_k \to V \times  \square^j_k$ ($0 \le j \le n-1$) is a regular scheme. Those cycles have the trivial boundaries (see Lemma \ref{lem:No-intersection}). Let $\TH^n_{\sfs}(V, n; m)$ denote the image of the canonical map $\TZ^n_{\sfs} (V,n;m) \to \TH^n(V, n; m)$ (see \S \ref{section:ACH-sfs}). The goal of this article is to prove the following result.

\begin{thm}[The \sfs-moving lemma]\label{thm:Main-1}
Let $k$ be an infinite perfect field. Let $m, n\ge 1$ be integers. Let $V$ be a smooth semi-local $k$-scheme essentially of finite type. Then the canonical map $\TH^n_{\sfs}(V, n; m) \to \TH^n(V, n; m)$ is an isomorphism.
\end{thm}

For the same $V$ as above, let $z^q(V, \bullet)$ denote the cubical version of Bloch's cycle complex (see \cite[\S 1]{KL}) and let $\CH^q(V, n)$ denote the associated higher Chow groups. We can define the subgroup $z^n_{\sfs}(V, n)$ of \sfs-cycles and the higher Chow group $\CH^n_{\sfs}(V,n)$ of \sfs-cycles analogous to the additive higher Chow group of \sfs-cycles. There is a canonical map $\CH^n_\sfs(V,n) \to \CH^n(V,n)$. As a byproduct of the discussions toward the proof of Theorem \ref{thm:Main-1}, we can recover the following result, 
stated in \cite{EVMS}{\footnote{In \cite[Lemma 3.11]{EVMS}, Theorem \ref{thm:Main-3} is claimed for arbitrary fields, but we do not know if this can be
achieved using the techniques of linear projections.}}.

\begin{thm}\label{thm:Main-3}
Let $k$ be an infinite perfect field. Let $V = \Spec(R)$ be a smooth semi-local $k$-scheme essentially of finite type. Let $n \ge 1$ be an integer. Then the canonical map $\CH^n_\sfs(V,n) \to \CH^n(V,n)$ is an isomorphism.
\end{thm}

Theorem \ref{thm:Main-1} provides the main geometric ground for the proof of the following result and a few of its consequences in the paper \cite{KP-crys}, discussed separately due to the huge size and complexities of the proofs
of the current article. In particular, it allows one to describe the crystalline cohomology of a smooth scheme in positive characteristic in terms of algebraic cycles. 
  
\begin{thm}[{\cite{KP-crys}}]\label{thm:Main-4}
Let $k$ be any field and let $R$ be a smooth semi-local $k$-algebra essentially of finite type. Let $m, n \ge 1$ be integers. Then there is a natural isomorphism
\[
\tau_R: \W_m \Omega^{n-1}_R \xrightarrow{\cong} 
\TH^n(R, n;m),
\]
where $\W_m\Omega^{\bullet}_R$ is the big de Rham-Witt complex of Hesselholt and Madsen.
\end{thm}

\subsection{The presentation lemma}\label{sec:PL}
We deduce Theorem \ref{thm:Main-1} from the following general presentation lemma for residual cycles of linear projections. This has the flavor (hence the name) of Gabber's geometric presentation lemma (see \cite{CTHK}). Of course, our assertions are different and intricate.

Let $k$ be an infinite perfect field. Given a finite map $h: Y' \to Y$ of $k$-schemes and a reduced closed subscheme $Z \subset Y'$, let $h^+(Z)$ be the closure of $h^{-1}(h(Z)) \setminus Z$ in $Y'$ with the reduced induced closed subscheme structure. We call this the `residual scheme of $Z$' with respect to $h$.

Let $n \ge 1$ and let $\widehat{A}_0, \ldots, \widehat{A}_{n-1}$ be smooth projective and geometrically integral $k$-schemes of positive dimensions. For $0 \le j \le n-1$, let $A_j \subset \widehat{A}_j$ be a nonempty affine open subset. Set $C_0: = \Spec(k)$ and ${C}_j :=  \prod_{i=0} ^{j-1} A_i$ for $j \ge 1$. Let $\pi_j: C_{n} \to C_j$ be the obvious projection. For any map $f: Y' \to Y$, let $f_j: Y' \times C_j \to Y \times C_j$ be the map $f \times {\rm id}_{C_j}$.

Let $\overline{X} \subset \P^m_k$ be a reduced closed subscheme of pure dimension $r \ge 1$ and let $X \subset \overline{X}$ be the complement of a hyperplane in $\P^m_k$ such that $X$ is regular and integral. Let $\Sigma \subset X$ be a finite set of closed points. Let $Z \subset X \times C_{n}$ be an integral closed subscheme of dimension $r$ such that the projection $Z \to C_{n}$ is not constant, and the projection $Z \to X$ is finite and surjective. 

The presentation lemma for the residual schemes that we prove is the following.

\begin{thm}\label{thm:Main-2}
Let $k$ be an infinite perfect field. There exist a closed embedding $\overline{X} \inj \P^N_k$, a hyperplane $H \subset \P^N_k$ with $X = \overline{X} \setminus H$, and a dense open subset $\mathcal{U} \subset \Gr(N-r-1, H)$ of the Grassmannian variety such that for each $L \in \mathcal{U}(k)$, the linear projection $\phi_L: \P^N_k \setminus L \to \P^r_k$ away from $L$ defines a finite surjective morphism $\phi: \overline{X} \to \P^r_k$ such that the following hold.
\begin{enumerate}
\item
There exists a Cartesian square
\[
\xymatrix@C.8pc{
X \ar@{^{(}->}[r] \ar[d]_{\phi} & \overline{X} \ar[d]^{\phi} \\
\A^r_k \ar@{^{(}->}[r] & \P^r_k.}
\]
\item $\phi$ is {\'e}tale over an affine open neighborhood of $\phi(\Sigma)$.
\item $\phi(x) \neq \phi(x')$ for each pair $x \neq x'$ of points in $\Sigma$.
\item The map $k(\phi(x)) \to k(x)$ is an isomorphism for each $x \in \Sigma$.
\item The induced map $Z \to \phi_{n}(Z)$ is birational.
\item The map $\phi^+_n(Z) \to X$ is finite and surjective. 
\item $\pi_j(\phi^+_{n}(Z))$ is regular at all points lying over $\Sigma$ for each $0 \le j \le n$.
\end{enumerate}
\end{thm}

\subsection{Outline of proofs and remarks}\label{sec:Outline}
We first remark that although $V$ may be in general obtained by localizing a quasi-projective $k$-scheme at a finite set $\Sigma$ of not necessarily closed points, for the proof of the \sfs-moving lemma, we can easily reduce to the case of closed points. See Proposition \ref{prop:no_closed}. Then the proof the \sfs-moving lemma can be broadly divided into two parts. 

In the first part, we prove it when the underlying semi-local ring is the localization of an affine space $\A^r_k$ at a finite set of closed points. To solve this case, we rely on two key ingredients: the lemma of Bloch \cite[Lemma 1.2]{Bl1} and the moving lemma for cycles with modulus on affine spaces by Kai \cite{Kai}. (N.B. Part of what we need in this article from \emph{ibid.} is also available in \cite{KP3}.) The moving lemma of Kai allows us to ensure that our cycles can be made to intersect the closed points of the semi-local scheme $V$ properly. After this, we apply an ``spread out and specialize" type of argument using \cite[Lemma 1.2]{Bl1} to achieve our goal. 

Roughly speaking, we argue that we can equip the \sfs-property to cycles after moving them via a certain kind of twisted translations by a general set of $k$-rational points of $\A^r_k$. This requires us to use that the ground field $k$ is infinite. The rest of the argument is to construct a homotopy between the new and the original cycle. The plain translations by the rational points do not work and the twisted translations make the argument more involved than the classical case. This is done in \S \ref{sec:sfs-AS}.

In the second part, we prove the general case of the \sfs-moving lemma by combining the affine space case and the presentation lemma (Theorem \ref{thm:Main-2}). The proof of the presentation lemma is an intricate application of the method of linear projections and moduli in algebraic geometry. 

The reason for this intricacy lies in the fact that it is not sufficient for us to find enough linear projections which give finite and flat morphisms from a projective variety $X$ to projective spaces. We need to invoke a more delicate linear projection in such a way that if we project a subvariety in some smooth family over $X$ to a similar family over the projective space, the resulting residual scheme has certain desired geometric properties, e.g., regularity along a given set of fibers in the family. Even more, we need to ensure that if we project this smooth family over $X$ to a smaller dimensional family via proper maps, then the images of the residual scheme continue to enjoy the good properties.

Showing that one can find enough such linear projections that do the above jobs lies at the heart of the argument. We see that the moduli spaces of linear subspaces that we encounter in the process are all rational, and we find enough rational lines in them. We then reduce the argument to studies of a family of linear subspaces parameterized by a rational line (pencil of linear subspaces). This simplifies the problem. 

Along the proofs, we need to separate the cases of algebraically closed and general infinite perfect fields. We first prove the results over algebraically closed fields. Over a general infinite perfect field $k$, we argue that we can find enough linear subspaces after going to an algebraic closure $\bar{k}$ so that all desired properties are achieved (over $\bar{k}$) in such a generality that they remain to be satisfied for the original cycle over $k$ after descent. One of these generalities we ensure over $\bar{k}$ is that the whole residual \emph{scheme} is regular, and not just its irreducible components (even if the latter case suffices for the sfs-moving lemma). 
We then show that there are enough such linear subspaces defined over $k$. This is achieved using a Galois descent.  

Carrying out this program rigorously takes up from \S \ref{sec:LP} to \S \ref{sec:Pres-lem}. We combine them to prove the main results in \S \ref{sec:MR}.

We now make some remarks on our assumption on the ground field. We need $k$ to be infinite to ensure that our moduli spaces have enough $k$-rational points. We need it to be perfect to achieve the regularity of various residual subvarieties. Although we only need the regularity of cycles, our argument at some stage uses the condition that some regular schemes that we encounter in the middle are actually smooth over $k$ (e.g., see the last part of the proof of Proposition~\ref{prop:Main-II}). Perfectness requirement is evident even in the proof of the \sfs-moving lemma in affine space, where we need to use a specialization argument. To make sure that we do not destroy the regularity during the specialization, 
we need our over-field to be separably generated over $k$ (e.g., see the proof of Lemma \ref{lem:sfs-Kai}). This requires $k$ to be perfect.

Recall that the moving lemma of Bloch and Chow hold over all fields. One proves this for infinite perfect fields first. The case of finite field reduces to the case of infinite perfect fields using the techniques of pro-$\ell$-extensions and the push-pull operators on the Chow groups. However, we cannot use this technique in our case because the smoothness property of the \sfs-cycles are not well-behaved under the push-forward operators. However, based on Theorem \ref{thm:Main-1}, we prove Theorem \ref{thm:Main-4} in \cite{KP-crys} over all base fields 
with different methods. 

Finally, the reader may notice that our sfs-moving lemma is stated and proven in this paper for $\TH^n(V,n;m)$ for $m \ge 1$. However, we remark that one does not miss out on anything by this assumption because it is shown in \cite[Theorem~1.5]{KP-6} that $\TH^n(V,n;0) = 0$. In particular, $\TH^n_{\sfs}(V, n; 0) = 0$.

The main result of this article plays essential roles in \cite{GK}, \cite{GK-1} and \cite{KP-crys}. Apart from these applications, we hope that 
our presentation lemma through linear projection techniques as well as various results and ideas of manipulating locally closed subsets of the Grassmannian will be useful in the future to anyone in the mathematics community (in particular, those working with algebraic cycles) who uses the linear projection machines in the tool box.

\subsection{Conventions}\label{sec:Conven}
Unless we specify otherwise, $k$ is a fixed field. We shall
assume later that $k$ is infinite and perfect for our main results.
A $k$-scheme is a separated scheme of finite type over $k$. An affine $k$-scheme is a $k$-scheme which is affine. A $k$-variety is an equidimensional reduced $k$-scheme. The product $X \times Y$ means $X \times _k Y$, unless we specify otherwise. We let $\Sch_k$ be the category of $k$-schemes and $\Sm_k$ of smooth $k$-schemes. A scheme essentially of finite type is a scheme obtained by localizing at a finite subset of (not necessarily closed) points of a quasi-projective subscheme of a finite type $k$-scheme. We include the case of not localizing at all. For $\mathcal{C} =  \Sch_k, \Sm_k$, we let $\mathcal{C}^{\ess}$ be the extension of the category $\mathcal{C}$, whose objects are either those in $\mathcal{C}$ or those obtained by localizing an object of $\mathcal{C}$ at a finite subset. 

Given $X \in \mathcal{C}^{\ess}$ and a finite set of points $\Sigma \subset X$, we write $X_{\Sigma}$ for the localization of $X$ along $\Sigma$. If $Y \subset X$ is an inclusion of a reduced locally closed subscheme, then the closure of $Y$ is considered a closed subscheme of $X$ with the reduced induced structure. The image of a reduced closed subset under a proper map is considered a closed subscheme of the target scheme with the reduced induced structure.

\section{The \fs \ and \sfs-cycles}\label{sec:SFS}
After recalling the definition of higher Chow groups and additive higher Chow groups, we define our main objects of study: the \fs \ and \sfs-cycles. We prove some preliminary results about these cycles. 

\subsection{Higher Chow groups and additive higher Chow groups}\label{sec:additive complex}
Let $k$ be a field. First recall (\emph{cf.} \cite{Bl1}) the definition of higher Chow groups. Let $X \in \Sch_k ^{\ess}$ be equidimensional. Let $\P_k^1=\Proj\, k[Y_0,Y_1]$, and $\square^n=(\P_k^1\setminus\{1\})^n$. Let $(y_1, \cdots, y_n) \in \square^n$ be the coordinates. A \emph{face} of $\square^n$ is a closed subscheme defined by a set of equations $\{ y_{i_1} = \epsilon_1, \cdots, y_{i_s} = \epsilon_s\}$, where $\epsilon_j \in \{ 0, \infty\}$. For $1 \leq i \leq n$ and $\epsilon =0, \infty$, let $\iota_i ^{\epsilon}: \square^{n-1} \to \square^n$ be the inclusion given by $(y_1, \cdots, y_{n-1}) \mapsto (y_1, \cdots, y_{i-1}, \epsilon, y_{i}, \cdots, y_{n-1})$. Its image gives a codimension $1$ face.

Let $q, n \geq 0$. When $X$ is obtained by localizing at a non-closed point, for closed subschemes in $X \times \square^n$, the notion of dimensions could be ambiguous but the codimensions are well-defined. So, we use dimensions only when there is no ambiguity. 

Let $\un{z}^q (X, n)$ be the free abelian group on the set of integral closed subschemes of $X \times \square^n$ of codimension $q$, that intersect properly with $X \times F$ for each face $F$ of $\square^n$. We define the boundary map $\partial_i ^{\epsilon} (Z):= [ ({\rm Id}_X \times \iota_i ^{\epsilon})^* (Z)]$. This collection of data gives a cubical abelian group $(\un{n} \mapsto \un{z} ^q (X, n))$ in the sense of \cite[\S 1.1]{KL}, and the groups $z^q (X, n) := \un{z} ^q (X, n) / \un{z} ^q (X, n)_{\rm degn}$ (in the notations of \emph{loc.cit.}) give a complex of abelian groups, whose boundary map at level $n$ is given by $\partial:= \sum_{i=1} ^n (-1)^i (\partial_i ^{\infty} - \partial_i ^0)$. The homology ${\rm CH} ^q (X, n):= {\rm H}_n (z^q (X, \bullet), \partial)$ is called the higher Chow group of $X$.

\bigskip

We recall the definition of additive higher Chow groups from \cite[\S 2]{KP2} (see also \cite{Park}). Let $X\in \Sch_k^{\ess}$ be equidimensional. Let $\A^1_k=\Spec k[t]$,  $\G_m=\Spec k[t,t^{-1}]$, and $\overline{\square} = \mathbb{P}^1_k$. For $n \ge 1$, let $B_n = \mathbb{A}^1_k \times \square^{n-1}$, $\overline{B}_n = \A^1_k \times \overline{\square}^{n-1}$ and $\widehat{B}_n = {\P}^{1}_k \times \overline{\square}^{n-1} \supset \overline{B}_n$. Let $(t, y_1, \ldots , y_{n-1})\in \overline{B}_n$ be the coordinates.

On $\overline{B}_n$, define the Cartier divisors $F_{n,i} ^1 := \{ y_i = 1 \}$ for $1 \leq i \leq n-1$, $F_{n,0} := \{ t = 0 \}$, and let $F_n ^1 := \sum_{i=1} ^{n-1} F_{n, i} ^1$. A {\em face} of $B_n$ is a closed subscheme defined by a set of equations of the form $y_{i_1}=\epsilon_1, \ldots,  y_{i_s}=\epsilon_s$, where $\epsilon_j\in\{0,\infty\}.$ For $1 \leq i \leq n-1$ and $\epsilon=0, \infty$, let $\iota_{n,i} ^{\epsilon}\colon B_{n-1}\to B_n$ be the inclusion $(t,y_1,\ldots, y_{n-2})\mapsto (t,y_1,\ldots, y_{i-1}, \epsilon,y_i,\ldots, y_{n-2})$. Its image is a codimension $1$ face.

The additive higher Chow complex is defined similarly using the spaces $B_n$ instead of $\square^n$, but together with proper intersections with all faces, we impose additional conditions called the \emph{modulus conditions}, that control how the cycles should behave at ``infinity'': (see \cite[Definition 2.1]{KP2}) let $X$ be a $k$-scheme, and let $V$ be an integral closed subscheme of $X \times B_n$. Let $\ov V$ denote the Zariski closure of $V$ in $X \times \overline{B}_n$ and let $\nu \colon {\ov V}^N \to {\ov V} \subset X \times \overline{B}_n$ be the normalization of $\ov V$. Let $m,n \geq 1$ be integers. We say that $V$ satisfies the \emph{modulus $m$ condition} on $X \times B_n$, if as Weil divisors on ${\ov V}^N$ we have $(m+1)[{\nu}^*(F_{n,0})] \le [\nu^*(F^1_n)].$  When $n=1$, we have $F_1 ^1 = \emptyset$, so it means $\nu^* (F_{1,0}) = 0$, or $\{ t = 0 \} \cap \overline{V} = \emptyset$. If $V$ is a cycle on $X \times B_n$, we say that $V$ satisfies the modulus $m$ condition if each of its irreducible components satisfies the modulus $m$ condition. When $m$ is understood, often we just say that $V$ satisfies the modulus condition. Note that since $F_{n,0} = \{ t = 0 \} \subset \overline{B}_n$, replacing $\overline{B}_n$ by $\widehat{B}_n$ in the definition does not change the nature of the modulus condition on $V$.

% ANT number 1
For an equidimensional $X  \in \Sch_k ^{\ess}$, and integers $m, n, q \geq 1$, we first define $\un{\TZ}^q (X, 1; m)$ to be the free abelian group on integral closed subschemes $Z$ of $X \times \mathbb{A}^1$ of codimension $q$, satisfying the modulus condition (see \cite[Definition 2.5]{KP2}).
For $n>1$,  $\un{\TZ}^q (X, n; m)$ is the free abelian group on integral closed subschemes $Z$ of $X \times B_n$ of codimension $q$ such that for each face $F$ of $B_n$, $Z$ intersects $X \times F$ properly on $X \times B_n$, and $Z$ satisfies the modulus $m$ condition on $X \times B_n$. For each $ 1 \leq i \leq n-1$ and $\epsilon= 0, \infty$, let $\partial_i ^\epsilon (Z) :=[ ({\rm Id_X}  \times\iota_{n, i}^{\epsilon})^*(Z)]$. The proper intersection with faces ensures that $\partial_i ^{\epsilon} (Z)$ are well-defined. The cycles in $\un{\TZ}^q (X, n;m)$ are called the \emph{admissible cycles} (or, often as \emph{additive higher Chow cycles, or additive cycles}). 

This gives the cubical abelian group $(\un{n}\mapsto \un{\TZ}^{q}(X, n+1; m))$ in the sense of \cite[\S 1.1]{KL}. Using the containment lemma \cite[Proposition 2.4]{KP}, that each face $\partial_i ^{\epsilon} (Z)$ lies in $\un{\TZ} ^q (X, n-1; m)$ is implied from the defining conditions.

For a cycle $ \sum_{i=1} ^s n_i Z_i$,we let $|\alpha|$ be the closed subscheme $ \bigcup_{i=1} ^s Z_i$ with its reduced structure. This is called the support of $\alpha$. If $f:Y \to X$ is flat and $\alpha \in \un{\TZ} ^q (X, n; m)$, we write $f^*(\alpha)$ often as $\alpha_Y$. This shorthand is more evident when $f$ is a localization morphism.

\begin{defn}[{\cite[Definition 2.6]{KP2}}]
Let $X \in \Sch_k ^{\ess}$ be equidimensional. The {\em additive higher Chow complex}, or just the \emph{additive cycle complex}, $\TZ^q(X, \bullet ; m)$ of $X$ in codimension $q$ with modulus $m$ is the non-degenerate complex associated to the cubical abelian group $(\un{n}\mapsto \un{\TZ}^{q}(X, n+1; m))$, {i.e.}, $\TZ^q(X, n; m)$ is the quotient $\un{\TZ}^q(X, n; m)/{\un{\TZ}^q(X, n; m)_{\dgn}}.$

The boundary map of this complex at level $n$ is given by $\partial := \sum_{i=1} ^{n-1} (-1)^i (\partial^{\infty}_i - \partial^0_i)$, and it satisfies $\partial ^2 = 0$. The homology $\TH^q(X, n; m): = {\rm H}_n (\TZ^q(X, \bullet ; m))$ for $n \ge 1$ is the {\it additive higher Chow group} of $X$ with modulus $m$. 
\end{defn}

\subsection{Subcomplexes associated to some algebraic subsets}\label{sec:subcx ass}

Let $X \in \Sch_k ^{\ess}$ be a variety. Here are some subgroups of $\TZ^q (X, n;m)$ with a finer intersection property with a given finite set $\mathcal{W}$ of locally closed algebraic subsets of $X$:

\begin{defn}[{\cite[Definition 4.2]{KP}}]\label{defn:subcx ass}
Define $\un\TZ^q _{\mathcal{W}} (X, n;m)$ to be the subgroup of $\un\TZ^q (X, n;m)$ generated by integral closed subschemes $Z \subset X \times B_n$ that additionally satisfy 
\begin{equation}\label{eqn:add Chow W}
{\rm codim}_{W \times F} (Z \cap (W \times F)) \geq q  \mbox{ for all } W \in \mathcal{W} \mbox{ and all faces } F \subset B_n.
\end{equation} 
The groups $ \un\TZ^q _{\mathcal{W}} (X, n+1;m)$ for $n \geq 0$ form a cubical subgroup of $(\un{n} \mapsto \un\TZ^q (X, n+1;m))$ and they give the subcomplex $\TZ^q _{\mathcal{W}} (X, \bullet;m) \subset \TZ^q (X, \bullet;m)$ by modding out by the degenerate cycles. The homology groups are denoted by $\TH^q_{\mathcal{W}} (X, n;m)$.
\end{defn}

\subsection{Schemes with finite closed points}

Recall that (\emph{cf.} \cite[\S 2.2]{GLL}) we say a scheme $X$ is an FA-scheme if for any finite subset $\Sigma \subset X$, there exists an \emph{affine} open neighborhood $U \subset X$ of $\Sigma $. We have the following (\emph{loc.cit.}):

\begin{lem}\label{lem:FA} Any quasi-projective $k$-scheme is FA. Any open subset of an FA-scheme is FA. Given any finite subset $\Sigma$ of a quasi-projective $k$-scheme, and an open subset $U \subset X$ containing $\Sigma$, there exists an affine open neighborhood $W \subset U$ of $\Sigma.$
\end{lem}

Recall (\S \ref{sec:Conven}) that a semi-local $k$-scheme $V$ is essentially of finite type if there is a quasi-projective $k$-scheme whose localization at a finite subset $\Sigma$ of points gives $V$. By Lemma \ref{lem:FA}, we may obtain it by localizing an \emph{affine} $k$-scheme of finite type.

\begin{defn}\label{defn:Geom-type}
For any semi-local $k$-scheme $V$ essentially of finite type, a pair $(X, \Sigma)$ consisting of an affine $k$-scheme $X$ of finite type and a finite set $\Sigma$ of points such that $V= \Spec (\mathcal{O}_{X, \Sigma})$, is called an \emph{atlas} for $V$. A smooth (resp. regular) atlas  $(X, \Sigma)$ is an atlas such that $X$ is smooth over $k$ (resp. regular).
\end{defn}

\begin{lem}\label{lem:Adm-spread}
Let $V= \Spec (R)$ be a semi-local $k$-scheme obtained by localizing at a finite set $\Sigma$ of points of a quasi-projective $k$-variety $X$. For a cycle $\alpha$ on $V \times B_n$, let $\overline{\alpha}$ be its Zariski closure in $X \times B_n$. 

Then $\alpha \in \TZ^q _{\Sigma} (V, n;m)$ if and only if there exists an affine open neighborhood $U \subset X$ of $\Sigma$ such that $\overline{\alpha}_U \in \TZ^q_{\Sigma}(U, n;m)$. 

Here, if $\partial (\alpha) = 0$, then we can assume that 
$\partial (\overline{\alpha}_U) = 0$. If $\alpha$ is a boundary, then we can assume $\overline{\alpha}_U$ is also a boundary. If $V$ is smooth over $k$, then we may take $(U, \Sigma)$ to be a smooth atlas. 
\end{lem}

\begin{proof}
The first three assertions were proven in \cite[Lemmas 4.13, 4.14]{KP3}. For the last one, choose any $X$ of finite type using the first assertion. Since $V$ is smooth, we have $X_{\rm sing} \cap V = \emptyset$ and $X_{\rm sm} = X \setminus X_{\rm sing} \supset \Sigma$. By Lemma \ref{lem:FA}, we can choose an affine open $U\subset X_{\rm sm} $ containing $\Sigma$.
\end{proof}

\subsection{The $\fs$-cycles}\label{sec:fs}

Recall that for higher Chow groups of a semi-local $k$-scheme $V$ in the Milnor range, \cite[Lemma 3.11]{EVMS} used the notions called \fs-cycles and \sfs-cycles. 
An \fs-cycle in \emph{loc.cit.} is a cycle $\alpha$ on $V \times \square_k ^n$ such that for each irreducible component $Z$, the morphism $Z \to V$ is finite and surjective. However, a moment's thought gives that it is not a good notion. For instance, if $V$ is reducible, then one can almost never achieve the surjection part.

Even if we modify the definition a bit by requiring instead that 
the support $|\alpha| \to V$ is finite and surjective, still there is a problem when $V$ is not irreducible: suppose $V = V_1 \cup V_2$ is a disjoint union of irreducible components. Suppose for $i=1,2$, we have an irreducible closed subscheme $Z_i$ on $V \times \square_k ^n$ such that $Z_i\to V_i$ is finite surjective. Then $W:= Z_1 + Z_2$ and $W' := Z_1 + 2 Z_2$ are both \fs-cycles in this updated sense. But, then $W' - W = Z_2$ is still finite over $V$, while it is no longer surjective over $V$. As a result the set of \fs-cycles in the above sense is not even closed under basic summation of cycles, thus they do not form a group. 

The natural notion to work with is the following.

\begin{defn}\label{defn:fs}
Let $X, Y \in \Sch_k ^{\ess}$. First suppose that $Y$ is irreducible. In this case, we say that a morphism $Y \to X$ of $k$-schemes is \emph{fs over $X$} (or an \emph{fs-morphism}, or simply \emph{fs} when $X$ is understood) if it is finite and it is surjective to an irreducible component of $X$. 

In case $Y$ is not necessarily irreducible, we say $Y \to X$ is \emph{fs over $X$} if for each irreducible component $Y_j \subset Y$, the induced map $Y_j \to X$ is fs over $X$. 

We generalize it further: let $f: Y \to X$ be a morphism in $\Sch_k ^{\ess}$ and let $U \to X$ be a flat morphism. We say that $Y \to X$ is \emph{fs over $U$}, if the fiber product $f' : Y \times_X U \to U$ is fs.
\end{defn}

This notion coincides with the na\"ive notion mentioned above when $X$ is irreducible. Unlike the na\"ive notion, this notion of \fs-morphisms behaves well under base changes:

\begin{lem}\label{lem:fs bc}
Let $f: Y \to X$ be an fs morphism in $\Sch_k ^{\ess}$. Let $U \to X$ be a flat morphism in $\Sch_k ^{\ess}$. Then the fiber product $f': Y \times_X U \to U$ is fs.
\end{lem}

\begin{proof}
That the base change of a finite morphism is again finite is apparent. The remaining part on surjectivity over an irreducible component follows by \cite[Proposition (2.3.7)-(ii), p.16]{EGA4-2}, where the dominance there is equivalent to surjectivity under finiteness.
\end{proof}

\begin{lem}\label{lem:fs fpf}
Let $Z$ be a cycle on $Y \times B$ such that $Z$ is fs over $Y$ in the sense that each irreducible component of $Z$ is fs over $Y$. 

Let $f: Y \to X$ be a finite surjective morphism in $\Sch_k ^{\ess}$ of irreducible schemes. Then the finite push-forward $f_* (Z)$ on $X \times B$ is fs over $X$. 
\end{lem}

\begin{proof}
We may assume $Z$ is irreducible. Since $Z \to Y$ is finite surjective and $Y \to X$ is finite surjective, the composite $Z \to Y \to X$ is finite surjective.
\end{proof}

Here is one simple criterion on finiteness

\begin{lem}[Finiteness criterion]\label{lem:finiteness}
Let $X$ be an equidimensional affine $k$-scheme essentially of finite type. Let $\widehat{B}$ be a smooth projective geometrically integral $k$-scheme of finite type of dimension $n>0$ and let $B \subset \widehat{B}$ be a nonempty affine open subset.

Let $Z \in z^n ( X \times B)$ be an irreducible cycle. Then $Z \to X$ is fs over $X$ if and only if $Z$ is closed in $X \times \widehat{B}$.
\end{lem}

\begin{proof}
Let $f: Z \hookrightarrow X \times \widehat{B} \to X$ be the composite map. Suppose $f$ is fs over $X$. Since the second map is projective, by \cite[Corollary II-4.8-(e), Theorem II-4.9, pp.102-103]{Hartshorne}, the first map is a closed immersion. This proves $(\Rightarrow)$. 

Conversely, suppose that $Z$ is closed in $X \times \widehat{B}$, i.e., the first map is a closed immersion (thus projective). Since the second map is projective, the composite $f$ is projective. Hence, $f$ is a projective morphism of affine schemes, so that it must be finite by \cite[Exercise II-4.6, p.106]{Hartshorne}. Moreover, $Z \to X_i$ being a finite map of irreducible affine schemes of the same dimension, where $X_i$ is the irreducible component that receives $Z$, this morphism must also be surjective. This proves $(\Leftarrow)$.
\end{proof}

\begin{lem}\label{lem:fs-smooth}
Let $V = \Spec (R)$ be a semi-local $k$-scheme essentially of finite type with the set of closed points $\Sigma$. Let $B \subset \widehat{B}$ be as in Lemma \ref{lem:finiteness}. Let $F:= \widehat{B} \setminus B$. Let $Z \in z^n (V \times B)$ be an irreducible cycle and let $\widehat{Z}$ be the Zariski closure of $Z$ in $V \times \widehat{B}$.

Suppose that $\widehat{Z} \cap (\Sigma \times F) = \emptyset$. Then given any affine atlas $(X,\Sigma)$ for $V$, there exists an affine open subatlas $(U, \Sigma)$ for $V$ such that for the Zariski closure $\bar{Z}$ of $Z$ in $X \times B$, the projection map $\bar{Z}_U \to U $ is fs over $U$.

If $V$ is smooth over $k$ from the first place, then we can choose $(U, \Sigma)$ such that $U$ is smooth over $k$ as well.
\end{lem}

\begin{proof}
Let $(X,\Sigma)$ be a given atlas.  Let $\widehat{\bar{Z}}$ be the Zariski closure of $\bar{Z}$ in $X \times \widehat{B}$ and let $\widehat{f}: \widehat{\bar{Z}} \hookrightarrow X \times \widehat{B} \to X$ be the composition with the projection. Let $Y:= \widehat{f} ( \widehat{\bar{Z}} \cap (X \times F))$. Since $\widehat{f}$ is projective and since $\widehat{\bar{Z}} \cap (\Sigma \times F) = \widehat{Z} \cap (\Sigma \times F) = \emptyset$, we see that $Y \subset X$ is a closed subset disjoint from $\Sigma$. Hence, $X \setminus Y$ is an open neighborhood of $\Sigma$ such that $\widehat{\bar{Z}} \cap ((X \setminus Y) \times F) = \emptyset$. By Lemma \ref{lem:FA}, we can find an affine open neighborhood $U$ of $\Sigma$ in $X \setminus Y$, so we have $\widehat{\bar{Z}} \cap (U \times F) = \emptyset$. In particular, $\widehat{\bar{Z}} \cap (U \times \widehat{B}) = \bar{Z} \cap (U \times \widehat{B})$. This means $\bar{Z}_U$ is closed in $U \times \widehat{B}$. Hence, by Lemma \ref{lem:finiteness}, the map $\bar{Z}_U \to U$ is fs over $U$.

In case $V$ is smooth, then by excising the singular locus of $X$, which is disjoint from $\Sigma$, we may assume that $X$ is smooth. Then the open subset $U \subset X$ is also smooth.
\end{proof}

Let $X$ be an equidimensional quasi-projective $k$-scheme and let $\Sigma \subset X$ be a finite set of points. By Lemma \ref{lem:FA}, we may replace $X$ be an affine $k$-scheme. We have the following two notions of \fs-cycles:

\begin{defn} Let $V= X_{\Sigma}$. Let $m, n \geq 1$ be integers.
\begin{enumerate} 
\item A cycle $\alpha \in \TZ^n _{\Sigma} (X, n;m)$ is said to be an \emph{\fs-cycle along $\Sigma$} if there is an affine open neighborhood $U \subset X$ of $\Sigma$ such that each irreducible component of $\alpha_U$ is fs over $U$. The group of \fs-cycles along $\Sigma$ is denoted by $\TZ^n_{\Sigma, \fs}(X, n;m)$. 
\item A cycle $\alpha \in \TZ^n _{\Sigma} (V, n;m)$ is said to be an \emph{\fs-cycle} if each irreducible component of $\alpha$ is fs over $V$. The group of \fs-cycles is denoted by $\TZ^n_{\fs} (V, n;m)$. 
\end{enumerate}
\end{defn}

These two notions are related as follows:

\begin{cor}\label{cor:fs-fine}
Let $X$ be an equidimensional affine $k$-scheme and let $\Sigma \subset X$ be a finite set of points. Let $V= X_{\Sigma}$. Let $m, n \geq 1$ be integers. Then a cycle $\alpha \in {\TZ}^n_{\Sigma} (X, n;m)$ is an fs-cycle along $\Sigma$ if and only if $\alpha_V \in {\TZ}^n_{\Sigma} (V, n;m)$ is an fs-cycle.
\end{cor}

\begin{proof}
$(\Rightarrow)$ Since the localization map $V \to X$ is flat and it factors through any open neighborhood $U \subset X$ of $\Sigma$, one can pull-back by Lemma \ref{lem:fs bc} to prove this direction. 

$(\Leftarrow)$ By Lemma \ref{lem:Adm-spread}, there exists an affine open subatlas $(U_1, \Sigma)$ of $(X, \Sigma)$ for $V$ such that the closure $\overline{\alpha}$ of $Z$ in $U_1 \times B_n$ is in $\TZ_{\Sigma} ^n (U_1, n;m)$. 

For each irreducible component $Z$ of $\alpha$, let $\widehat{Z}$ be its Zariski closure in $V \times \widehat{B}$. Since $Z$ is fs over $V$, by Lemma \ref{lem:finiteness} $Z$ is already closed in $V \times \widehat{B}_n$, thus $Z= \widehat{Z}$. In particular, $\widehat{Z} \cap (\Sigma \times F_n) = \emptyset$. Hence by Lemma \ref{lem:fs-smooth} there exists an affine open subatlas $(U_Z, \Sigma)$ for $V$ of $(U_1, \Sigma)$ such that for the Zariski closure $\overline{Z}$ of $Z$ in $U_1 \times B_n$, the base change $\overline{Z}_{U_Z} \to U_Z$ is fs. By taking $U:= \bigcap_Z U_Z$ where the intersection is taken over all (finitely many) irreducible components of $\alpha$, we deduce that $\overline{Z} _{U} \to U$ is fs. This proves the corollary.
\end{proof}

We have the following a bit different characterization of the cycles centered around $\TZ^n_{\fs} (V,n;m)$:

\begin{prop}\label{prop:fs-suff-condn}
Let $V = \Spec(R)$ be a semi-local $k$-scheme of geometric type with the set $\Sigma$ of closed points. Let $m, n \ge 1$. Let $Z \in \TZ^n_{\Sigma}(V, n;m)$ be an irreducible cycle. Then $Z$ is an $\fs$-cycle if and only if there is an atlas $(X,\Sigma)$ for $V$ such that for the closures $\bar{Z}$ in $X \times B_n$ and $\widehat{Z}$ in $V \times \widehat{B}_n$, we have $\bar{Z} \in  \TZ^n_\Sigma(X, n;m)$ and $\widehat{Z}\cap (\Sigma \times F_n) = \emptyset$.

Here, $V$ is smooth over $k$ if and only if we can choose $(X, \Sigma)$ in the above such that $X$ is smooth over $k$ as well.
\end{prop}

\begin{proof}
For the first assertion, suppose that $Z$ is an $\fs$-cycle. By Lemma \ref{lem:Adm-spread}, there is a affine atlas $(X, \Sigma)$ for $V$ such that $\bar{Z} \in \TZ^n_{\Sigma} (X, n;m)$. Since $Z \to V$ is fs over $V$, by Lemma \ref{lem:finiteness}, $\widehat{Z} \cap (\Sigma \times F_n) = \emptyset$. 

Conversely, suppose that for an atlas $(X,\Sigma)$ and the closure $\bar{Z}$ in $X \times B_n$, we have $\bar{Z} \in \TZ^n _\Sigma (X, n;m)$ and $\widehat{Z} \cap (\Sigma \times F_n) = \emptyset$. Then, by Lemma \ref{lem:fs-smooth}, we may shrink $(X,\Sigma)$ to an affine open atlas $(U,\Sigma)$ such that $\bar{Z}_U \to U$ fs over $U$. Hence $\bar{Z}_U \in \TZ^n _{\Sigma, \fs} (U, n;m)$. Now by Corollary \ref{cor:fs-fine}, we have $Z \in \TZ^n _{\fs} (V, n;m)$. 

For the second assertion, in case $V$ was smooth, then we could have take $X$ to be smooth here by the last assertion of Lemma \ref{lem:Adm-spread}. Conversely, a localization of a smooth scheme is smooth again, so that $V$ is smooth over $k$.
\end{proof}

\subsection{The \sfs-cycles}\label{sec:SFS*}
For $1 \le j \le n$, let $\pi_j: B_n \to B_{j}$ and ${\widehat{\pi}}_j : {\widehat{B}}_n \to {\widehat{B}}_{j}$ be the projection maps. Let $X \in \Sch^{\ess}_k$ equidimensional. We shall often denote the maps ${\rm id}_X \times \pi_j: X \times B_n \to X \times B_j$ and ${\rm id}_X \times {\widehat{\pi}}_j: X \times \widehat{B}_n \to X \times \widehat{B}_j$ simply by $\pi_j$ and $\widehat{\pi}_j$, respectively, if the scheme $X$ is fixed in a given context.

For any reduced closed subscheme $Z \subset X \times B_n$ and $1 \le j \le n$, let $Z^{(j)} =({\rm id}_X \times  \pi_j)(Z)$ be the scheme-theoretic image of $Z$. Let $Z^{(0)}$ be the scheme-theoretic image of $Z$ in $X$.Note that if the projection $Z \to X$ is proper, then $({\rm id}_X \times  \pi_j)(Z)$ is closed in $X \times B_j$ and, with its reduced induced closed subscheme structure, coincides with $Z^{(j)}$. The same holds for $Z^{(0)}$. We shall use $Z^{(j)}$ when $Z \to X$ is in fact finite.

\begin{defn}\label{defn:sfs-0} 
Let $X \in \Sch^{\ess}_k$ be smooth over $k$ and let $\Sigma \subset X$ be a finite set of points. Let $ m, n \geq 1$ be integers. An integral cycle $[Z] \in \un{\TZ}^n (X, n;m)$  is called an \emph{\sfs-cycle along $\Sigma$}, if $[Z] \in \un{\TZ}^n_{\Sigma}(X, n;m)$, and there exists an affine neighborhood $U\subset X $ of $\Sigma$ such that the following hold.
\begin{enumerate}
\item $Z_U$ is finite and surjective over an irreducible component of $U$, 
i.e., $Z_U \to U$ is an \fs-morphism.
\item The scheme $(Z^{(j)})_U$ is smooth over $k$ for every $0 \le j \le n$. 
\end{enumerate}
A cycle $\alpha \in \un{\TZ}^n (X, n;m)$ is called an \emph{sfs-cycle along $\Sigma$} if every irreducible component of $\alpha$ is an \sfs-cycle along $\Sigma$.
\end{defn}

\begin{lem}\label{lem:sfs-spread}
Let $X$ be an equidimensional smooth affine $k$-scheme and let $\Sigma \subset X$ be a finite set of points. Let $V= X_{\Sigma}$. Let $m, n \geq 1$ be integers. Then $\alpha \in \un{\TZ}^n_{\Sigma}(X, n;m)$ is an \sfs-cycle along $\Sigma$ if and only if $\alpha_V \in \un{\TZ}^n_{\Sigma}(V, n;m)$ is an fs-cycle such that $Z^{(j)}$ is smooth over $k$ for each $0 \leq j \leq n$ and for each irreducible component $Z$ of $\alpha_V$.
\end{lem}

\begin{proof}
Under Corollary \ref{cor:fs-fine}, the $(\Rightarrow)$ direction is obvious. We prove $(\Leftarrow)$. By Corollary \ref{cor:fs-fine}, together with Lemma \ref{lem:FA}, we can find an affine open neighborhood $U' \subset X$ of $\Sigma$ such that the closure $\alpha_{U'} \in \un{\TZ}^n_{\Sigma} (U', n;m)$ is an \fs-cycle along $\Sigma$. Now let $Y \subset U'$ be the union of the images of the finite maps $(Z_{U'}^{(j)})_{\rm sing} \to U'$, where $Z$ runs over all irreducible components of $\alpha$ and $0 \leq j \leq n$. Since $Z_{U'} \to U'$ is finite for each $Z$, this $Y \subset U'$ is a closed subset that does not meet $\Sigma$. By Lemma \ref{lem:FA}, we can choose an affine open neighborhood $U \subset U'\setminus Y$ of $\Sigma$. Then for each component $Z$ of $\alpha$ and each $0 \leq j \leq n$, the scheme $Z_U ^{(j)}$ is smooth over $k$. Note $(Z_U) ^{(j)} = (Z^{(j)})_U$ naturally. This shows that $\alpha_U$ is an \sfs-cycle along $\Sigma$. 
\end{proof}

Another property that \sfs-cycles enjoy is the following.

\begin{lem}\label{lem:pull-back-sfs}
Let $\phi: X \to Y$ be an \'etale morphism of smooth affine $k$-schemes. Let $\Sigma \subset Y$ be a finite set of points and let $\Sigma' =\phi^{-1}(\Sigma)$. Let $Z \in \un{\TZ}^n (Y, n;m)$ be an integral \sfs-cycle along $\Sigma$. Then the flat pull-back $\phi^*(Z) \in \un{\TZ}^n(X, n;m)$ is an \sfs-cycle along $\Sigma'$.
\end{lem}

\begin{proof}
It is easy to see that $\phi^*(Z) \in \un{\TZ}^n_{\Sigma'}(X, n;m)$. We now prove the other properties. We can shrink $Y$ and assume that $Z \to Y$ is finite and surjective, and $Z^{(j)}$ is smooth over $k$ for $0 \le j \le n$. Let $W := \phi^*(Z)$. It follows from Lemma \ref{lem:fs bc} that $W$ is an \fs-cycle along $\Sigma'$. To prove that each $W^{(j)}$ is smooth over $k$, let $W_j :=\phi^*(Z^{(j)})$ and consider the commutative diagram
\begin{equation}\label{eqn:pull-back-sfs-0}
\xymatrix@C.8pc{
W \ar@{->>}[r] \ar[d] & W^{(j)} \ar[r] & W_j \ar[r] \ar[d] & X 
\ar[d]^{\phi} \\
Z \ar@{->>}[rr] & & Z^{(j)} \ar[r] & Y.}
\end{equation} 
Here, the map $W^{(j)} \to W_j$ exists uniquely since the right square is Cartesian. The outer big square is also Cartesian, and this implies that so is the left square. In particular, the vertical arrows are all \'etale, the horizontal arrows are all finite and surjective and all schemes in \eqref{eqn:pull-back-sfs-0} are reduced. In particular, $W^{(j)} \surj W_j$. On the other hand, 
as $W \to X$ is finite, $W^{(j)} = \pi_j(W)$ is a reduced closed subscheme of $W_j$. Thus $W^{(j)} = W_j$. Since $Z$ and $Z^{(j)}$ are smooth over $k$ and $\phi$ is \'etale, it follows that $W$ and $W_j$ are smooth over $k$. In particular, $W^{(j)} = W_j$ is smooth over $k$. This finishes the proof. 
\end{proof}

\subsection{Additive higher Chow groups of \fs \ and \sfs-cycles}\label{section:ACH-sfs}
The goal of this paper is to prove the `\sfs-moving lemma' which will show that the cycle class groups of \sfs-cycles coincide with the additive higher Chow groups in the Milnor range for a smooth semi-local $k$-scheme essentially of finite type when $k$ is an infinite perfect field.

Let $m, n \geq 1$. Let $X$ be a smooth affine $k$-scheme and let $\Sigma \subset X$ be a finite set of points. It follows from Definition \ref{defn:sfs-0} that $\TZ^n_{\Sigma, \sfs}(X, n;m)$ is a subgroup of $\TZ^n_{\Sigma, \fs}(X, n;m)$.

\begin{defn}\label{defn:sfs-TCH}
We let
\[
\widetilde{\TH}^n _{\Sigma} (X, n; m) = \frac{{\ker}(\partial: \TZ^n _{\Sigma}(X, n;m) \to \TZ^n(X, n-1;m))}{{\rm im}  (\partial: \TZ^n (X, n+1;m) \to \TZ^n(X, n;m)) \cap \TZ^n _{\Sigma} (X, n;m)}.
\]
\[
\TH^n _{\Sigma, \fs} (X, n; m) = \frac{{\ker}(\partial: \TZ^n _{\Sigma, \fs}(X, n;m) \to \TZ^n(X, n-1;m))}{{\rm im}  (\partial: \TZ^n (X, n+1;m) \to \TZ^n(X, n;m)) \cap \TZ^n _{\Sigma, \fs} (X, n;m)}.
\]
\[
\TH^n _{\Sigma, \sfs} (X, n; m) = \frac{{\ker}(\partial: \TZ^n _{\Sigma, \sfs}(X, n;m) \to \TZ^n(X, n-1;m))}{{\rm im} (\partial: \TZ^n (X, n+1;m) \to  \TZ^n(X, n;m)) \cap \TZ^n _{\Sigma, \sfs} (X, n;m)}.
\]

We similarly define $\widetilde{\TH}_{\Sigma} ^n (V, n;m)$, $\TH^n _{\fs} (V, n;m)$, and $\TH^n _{\sfs} (V, n;m)$.
\end{defn}

If $X$ is not necessarily connected, note that the groups for $X$ are obtained simply by taking the direct sums of the corresponding groups over all connected components of $X$.

In the above, the definition of the group $\widetilde{\TH}^n_{\Sigma} (X, n;m)$ is slightly different from that of $\TH^n_{\Sigma} (X, n;m)$ in Definition \ref{defn:subcx ass}. However, we have:

\begin{lem}\label{lem:TH_M}
The natural surjection $\TH^n _{\Sigma} (X, n;m) \surj \widetilde{\TH}^n _{\Sigma} (X, n; m)$ is an isomorphism. Similarly, $\TH^n_{\Sigma} (V, n;m) \to \widetilde{\TH}_{\Sigma} ^n (V, n;m)$ is an isomorphism.
\end{lem}

\begin{proof}
By the moving lemma for additive higher Chow groups of smooth affine schemes of W. Kai \cite{Kai} (see \cite[Theorem 4.1]{KP3} for a sketch of its proof), the composition $\TH^n _{\Sigma} (X, n;m) \surj \widetilde{\TH}^n _{\Sigma} (X, n; m) \to \TH^n (X, n;m)$ is an isomorphism. Hence, the first arrow is injective. The proof for the second one is similar, except that we use \cite[Theorem 4.10]{KP3}.
\end{proof}

We thus have canonical maps
\begin{equation}\label{eqn:moving arrows}
\TH^n _{ \sfs} (V, n;m) \to \TH^n _{\fs} (V, n; m) \to {\TH}^n _{\Sigma} (V, n; m) \to \TH^n (V, n;m),
 \end{equation}
where the last map is an isomorphism by Lemma~\ref{lem:TH_M} and \cite{Kai}.
Our goal is to show that all other maps are also isomorphisms. 

\subsection{Reduction to localization at closed points}\label{sec:reduct_closed}

The semi-local $k$-schemes essentially of finite type we consider are obtained by localizing an affine $k$-scheme (see Lemma \ref{lem:FA}) at a finite set $\Sigma$ of points which may not necessarily be closed. In \S \ref{sec:reduct_closed}, we show that for the \sfs-moving lemma, it is possible to reduce to the case when all points of $\Sigma$ are actually closed. The following is the goal:

\begin{prop}\label{prop:no_closed}
Suppose the natural map $\TH^n _{\sfs} (V, n;m) \to \TH^n (V, n;m)$ is an isomorphism for every smooth semi-local $k$-scheme $V$ essentially of finite type, obtained by localizing at a finite set of closed points. 
Then the natural map $\TH^n_{\sfs} (V, n;m) \to \TH^n (V, n;m)$ is an isomorphism for every smooth semi-local $k$-scheme $V$ essentially of finite type.
\end{prop}

We prove the following first:

\begin{lem}\label{lem:non-clos-local-surj}
Let $V$ be a smooth semi-local $k$-scheme essentially of finite type, obtained by localizing an affine $k$-scheme $X$ at a finite set $\Sigma$ of, not necessarily closed, points. Let $\alpha \in \un{\TZ}^n (V, n;m)$.

Then there exist $(1)$ a smooth semi-local $k$-scheme $V'$ essentially of finite type, obtained by localizing an affine $k$-scheme at a finite set $\Sigma'$ of closed points with a flat localization map $V \to V'$ and $(2)$ a cycle $\alpha' \in \un{\TZ}^n (V', n;m)$ such that the flat pull-back map $\phi^{V'}_V : \underline{\TZ}^n (V', n;m) \to \underline{\TZ}^n (V, n;m)$ satisfies $\phi^{V'} _V (\alpha') = \alpha$.
If $\partial \alpha=0$, we can ensure $\partial \alpha'=0$. 
\end{lem}

\begin{proof}
By Lemma \ref{lem:FA}, we may assume that $V= X_{\Sigma}$, where $X$ is a smooth affine $k$-scheme of finite type. For the cycle $\alpha \in \underline{\TZ} ^n (V, n;m)$, by Lemma \ref{lem:Adm-spread}, there exists a smooth affine open neighborhood $U \subset X$ containing $\Sigma$ such that the Zariski closure $\alpha_U$ of $\alpha$ in $U \times B_n$ such that it is in $\un{\TZ}^n (U, n;m)$.
If $\partial \alpha=0$, we can shrink $U$ further (if necessary) so that
$\partial \alpha_U = 0$.

For each $p \in \Sigma$, there exists a closed point $\mathfrak{m}_p \in U$ that is a specialization of $p$. (It exists by the basic fact in commutative algebra that any proper ideal of a commutative ring with unit is contained in a maximal ideal.) We choose it so that a distinct pair of points of $\Sigma$ gives a distinct pair of points. Let $\Sigma' := \{ \mathfrak{m}_p \ | \ p \in \Sigma\}$, and take $V':= U_{\Sigma'}$. Here, $\alpha_U \in \un{\TZ}^n (U, n;m)$, and let $\alpha' \in \un{\TZ}^n (V', n;m)$ be its flat pull-back via the localization map $V' \to U$. This satisfies $\partial \alpha' = 0$ if $\partial \alpha = 0$. By the construction of $V'$, we also have the localization map $V \to V'$ and the flat pull-back map $\phi^{V'} _V : \un{\TZ}^n (V', n;m)\to \un{\TZ}^n (V, n,m)$. By the construction of $\alpha'$, we have $\phi^{V'} _V (\alpha') = \alpha$. This proves the lemma.
\end{proof}

We remark however that Lemma \ref{lem:non-clos-local-surj} does not say that the map $\phi^{V'}_V$ is surjective. It simply says that for each element $\alpha$, there is some $V'$ such that $\alpha$ can be an image of a cycle over $V'$.

\bigskip

{\bf Proof of Proposition \ref{prop:no_closed}.}  Since the map $\TH^n _{\sfs} (V, n;m) \to \TH^n (V, n;m)$ is automatically injective, it is enough to prove that this is surjective. Let $\alpha \in \TH^n (V, n;m)$ be an arbitrary cycle class, and choose its cycle representative in $\un{\TZ}^n (V, n;m)$, also denoted by $\alpha$. 
Being a cycle representing a class in $\TH^n (V, n; m)$, we have $\partial \alpha=0$.

By Lemma \ref{lem:non-clos-local-surj}, there exists now a smooth semi-local $k$-scheme $(V', \Sigma')$ essentially of finite type, obtained by localizing at a finite set of closed points, a cycle class $\alpha' \in \TH^n (V', n;m)$ and the localization map $\phi^{V'}_V : \TH^n (V', n;m) \to \TH^n (V, n;m) $ sends $\alpha'$ to $\alpha$.

On the other hand, the localization map $\phi^{V'}_V$ sends the $\sfs$-cycles over $V'$ to the $\sfs$-cycles over $V$. 
To see this, we first note that this map sends $\un{\TZ}_{\Sigma'} ^n (V', n;m)$ to $\un{\TZ}_{\Sigma} ^n (V, n;m)$ because the localization does not increase the dimensions of schemes, thus the proper intersection condition with $\Sigma'$ implies the proper intersection condition with $\Sigma$. Now, the $\sfs$-cycles are preserved under $\phi^{V'}_V$ because the localization (flat pull-back) of \fs-morphisms are \fs-morphisms by Lemma \ref{lem:fs bc}, while it is a basic fact in commutative algebra that a localization of a regular local ring is again a regular local ring. Hence, we have a commutative diagram:
\begin{equation}\label{eqn:red_cl}
\xymatrix{
\TH_{\sfs} ^n (V', n;m) \ar[d] _{\sfs_{V'}} \ar[r] ^{\phi_{\sfs}} & \TH_{\sfs} ^n (V, n;m) \ar[d] ^{\sfs_V} \\
\TH^n (V', n;m) \ar[r] ^{\phi} & \TH^n (V, n;m),}
\end{equation}
where $\phi= \phi^{V'} _V$ and $\phi_{\sfs}$ is the restriction of $\phi$. By construction, we have $\phi (\alpha') = \alpha$. By the given assumption, we have that $\sfs_{V'}$ is surjective, so that there exists $\alpha'' \in \TH^n_{\sfs} (V', n;m)$ such that $\sfs_{V'} (\alpha'') = \alpha'$. Hence $\alpha = \phi (\alpha') = \phi \circ  \sfs_{V'}  (\alpha'') = ^{\dagger} \sfs_V \circ \phi_{\sfs} (\alpha'')$, where $\dagger$ holds by the commutativity of the diagram \eqref{eqn:red_cl}. In particular, $\alpha \in {\rm im} (\sfs_V)$. Since $\alpha$ was arbitrary in $\TH^n (V, n;m)$, this shows that $\sfs_V$ is surjective, hence an isomorphism. 
\hfill$\square$

We have one further result:

\begin{lem}\label{lem:No-intersection}
Let $(V, \Sigma)$ be a smooth semi-local $k$-scheme essentially of finite type. Let $m, n \geq 1$ be integers. Let $\alpha \in \un{\TZ}^n_\Sigma(V, n;m)$ be such that $|\alpha|$ is finite over $V$. Then $\alpha$ does not intersect $V \times F$  for any proper face $F \subset B_n$ at all. In particular, $\partial(\alpha) =0$.
\end{lem}

\begin{proof}
We may assume that $\alpha= [Z]$ is an irreducible cycle and $V$ is integral. We prove that $Z \cap (V \times F)$ is empty. 

The composite $Z \cap (V \times F)  \hookrightarrow Z \to V$ is finite by the given assumption. Hence, its image in $V$ is closed and therefore must
intersect $\Sigma$ non-trivially if non-empty. It suffices therefore to show that
the fiber product $\Sigma \times_V Z \times_{B_n} F = Z \cap (\Sigma \times F)$
is empty.

However, by the given assumption that $Z \in \un{\TZ}^n _{\Sigma} (V, n;m)$, the proper intersection condition with $\Sigma$ reads:
$\codim _{\Sigma \times F} Z \cap (\Sigma \times F)) \geq n$.
Equivalently, 
$$\dim \ Z \cap (\Sigma \times F) \leq \dim \ (\Sigma \times F) - n = \dim \ F - n < 0.
$$
But this means $Z \cap (\Sigma \times F) = \emptyset$. This proves the lemma.
\end{proof}

\bigskip

\textbf{Convention:} Using Proposition \ref{prop:no_closed}, from now on, when we say a semi-local $k$-scheme essentially of finite type, it will mean that it is obtained by localizing at a finite set of closed points, unless we say otherwise.

\section{The \sfs-moving lemma in affine spaces}\label{sec:sfs-AS}
In this section, we prove a special case of Theorem \ref{thm:Main-1} when the underlying semi-local scheme is a localization an affine space over $k$. This will be a ground for the general case of the theorem. 

\subsection{The set-up for affine spaces}

We fix some notations that we shall use throughout this section.

Let $k$ be an infinite perfect field. Let $m, n, r \geq 1$ be integers. We let $\Sigma \subset \A^r_k = \Spec(k[x_1, \ldots , x_r])$ be a finite set of closed points. Let $V$ be the localization of $\A^r_k$ at $\Sigma$. Let $j: V \to \A^r_k$ be the inclusion map. Let $p_n: \A^r_k \times \A^1_k \times \overline{\square}^{n-1} \to  \A^r_k \times \A^1_k$ and $q: \A^r_k \times \A^1_k  \to \A^r_k$ denote the projection maps and let $q_n = q \circ p_n$. Using the automorphism $y \mapsto 1/(1 - y)$ of $\P^1_k$, we replace $(\square, {\infty, 0})$ by $(\A^1_k, {0, 1})$, and write $\square = \A^1_k$. 

For any closed subset $Y \subset \A^r_k \times \A^1_k \times {\square}^{n-1}$, let $\overline{Y}$ be its closure in $\A^r_k \times \A^1_k \times \overline{\square}^{n-1}$. We let $Z \in \TZ^n_{\Sigma}(\A^r_k, n;m)$ be an irreducible cycle. For an integer $s \ge 0$ and a point $g \in \A^r_k$, we consider the map (\emph{cf.} \cite{Kai})
\begin{equation}\label{eqn:Kai-homotopy}
\phi_{g,s}: \A^r_k \times \A^1_k \times \square \times \overline{\square}^{n-1} \to  \A^r_k \times \A^1_k \times \overline{\square}^{n-1};
\end{equation}
\[
\begin{array}{lll}
\phi_{g,s}(\un{x}, t, y, y_1, \ldots, y_{n-1}) & = & (\un{x} + yt^{s(m+1)}g, t, y_1, \ldots, y_{n-1}).
\end{array}
\]

Note that $\phi_{g,s}$ is strictly speaking defined over the residue field of $g$, but to simply notations we often won't make it explicit. If needed, one can take the scalar extension to the residue field of $g$ to turn $g$ into a rational point. 
For $a \in \square(k)$, we let $\phi_{g,s,a}$ be the composite map
\[
\A^r_k \times \A^1_k \times \overline{\square}^{n-1} \inj  \A^r_k \times \A^1_k \times \square \times \overline{\square}^{n-1}  \xrightarrow{\phi_{g,s}} \A^r_k \times \A^1_k \times \overline{\square}^{n-1},
\]
where the first arrow takes $(\un{x}, t, \un{y})$ to $(\un{x}, t, a, \un{y})$.

The evaluation of $\phi_{g,s}$ at $y = 1$ defines an isomorphism $\A^r_{k(g)} \times \A^1_{k(g)} \to \A^r_{k(g)} \times \A^1_{k(g)}$, given by $\phi_{g,s,1}(\un{x}, t) = (\un{x} + t^{s(m+1)}g, t)$. Let $\phi_{g, s,1} ^{\sharp}: k(g)[\un{x}, t] \to k(g)[ \un{x}, t]$ be the corresponding $k(g)$-algebra isomorphism.

\subsection{Some properties of the twisted translations}

Note that $\phi_{g,s}$ is a flat morphism. In particular, $\phi^*_{g,s}(Z)$ is an algebraic cycle on $\A^r_k \times \A^1_k \times \square^{n}$. In 
the next few lemmas, we verify some algebraic and geometric properties of $\phi^*_{g,s}(Z)$.

\begin{lem}\label{lem:monic} 
Let $f(\un{x}, t) \in k[\un{x}, t]$ be a nonzero polynomial. Then there is a nonempty open subset $U \subset \mathbb{A}^r_k$ such that for each $g \in U$ and sufficiently large $s \gg 0$ (not depending on $g$), the polynomial $\phi_{g,s,1}^{\sharp} (f)$ is monic in $t$ over $k(g)[\un{x}]$, i.e., integral over  $k(g)[\un{x}]$.
\end{lem}

\begin{proof}
Let $M: = \deg_t f$ and write $f(\un{x},t) = \sum_{i=0} ^M f_i (\un{x}) t^{M-i}$ for some $f_i \in k[ \un{x}]$ and $M\geq 0$. Since $f \not = 0$, we have $f_0 (\un{x}) \not = 0$. Let $d_i = \deg_{\un{x}} (f_i)$, which is the total degree in $\un{x}$. We first consider the case $r=1$ and take $U = \mathbb{A}^1_k \setminus \{ 0 \}$. Let $c_i \in k$ be the coefficient of the highest degree term of $f_i (\un{x})$. Since $f_0 (\un{x}) \not = 0$, we have $c_0 \in k^{\times}$. Then,
 $f(x + t ^{s (m+1)}g, t) = \sum_{i=0} ^M f_i (x + t^{s (m+1)} g) t^{M-i} = \sum_{i=0 } ^M c_i (g ^{d_i}t^{ d_i s (m+1) + M-i} + (\mbox{lower degree terms in } t))$. Let $i_0$ be the smallest integer such that $d_{i_0 } = \max \{ d_0, d_1, \ldots, d_{M} \}$. Here, $c_{i_0} \in k^{\times}$ by definition. 

If $d_{i_0 } = 0$, then each $f_i (x)$ is a constant, so $ f(x + t^{s (m+1)}g, t) $ gives an integral dependence in $t$ as desired. Suppose $d_{i_0 } > 0$. If $i_0 = 0$, then for each $i>0$ and each $s>0$, we have $d_0 s (m+1) + M \geq d_i s (m+1) + M > d_i s (m+1) + M-i$. Hence, the leading coefficient of the highest degree term in $t$ is $c_0 g^{d_0} \in k(g)^{\times}$, so, after dividing by this unit $c_0 g^{d_0}$, we get a monic polynomial in $t$. Hence it is integral. 

If $i_0 > 0$, then for each $i > i_0 $ and each $s>0$, we have $d_{i_0} s (m+1) +M- i_0 \geq d_i s (m+1) + M-i_0 > d_i s (m+1) + M- i$, while for $0 \leq i < i_0$, we have $d_i < d_{i_0}$ so that for every sufficiently large $s>0$, we have $d_i s (m+1) + M-i < d_{i_0} s (m+1) + M-{i_0}$. Note that this choice of $s$ depends only on $f$ and not on $g$. Hence, for every sufficiently large $s>0$ (not depending on $g$), again the leading coefficient of highest degree in $t$ is $c_{i_0} g^{d_{i_0}} \in k(g)^{\times}$. Hence after dividing by this unit, it gives the desired integral dependence relation. 

In case $r\geq 2$, the backbone of the proof is the same, but one problem is a possible cancellation of the highest degree terms in $t$, namely, if $d_i$ is the total degree of $f_i (x_1, \ldots, x_r)$, then possibly a multiple number of monomials in $\phi_{g,s,1} ^{\sharp} (f)$ could have the same total degree $d_i$. However, such $g$'s form a closed subscheme of $\mathbb{A}^r_k$ (depends on $f(\un{x}, t)$), so for a general $g \in U$ for some nonempty open subset $U \subset \mathbb{A}^r_k$, we can avoid it.
\end{proof}

In \cite[Proposition 2.3]{Kai} (or see \cite[Claim of proof of Theorem 4.1]{KP3}), W. Kai defines a positive integer $s(Z)$ associated to $Z$, which plays a crucial role in proving the modulus condition for $\phi^*_{g,s}(Z)$. 

\begin{lem}\label{lem:good-intersection}
Let $s \ge s(Z)$ be any integer. Then $\phi^*_{g,s}(Z) \in \TZ^n(\A^r_k, n+1;m)$ for any $g \in \A^r_k$.
\end{lem}

\begin{proof}
The modulus condition for $\phi^*_{g,s}(Z)$ follows from \cite[Proposition 2.3]{Kai} (see also \cite[Proof of Theorem 4.1]{KP3}). We show that $\phi^*_{g,s}(Z)$ intersects all faces of $\square^n$ properly. Let $F$ be a face of $\square^n$. If $F = \{0\} \times F'$ for some face $F'$ of $\square^{n-1}$, then the proper intersection follows directly from that of $Z$ with $F'$ since the map $\phi_{g,s,0}$ is identity. If $F = \{1\} \times F'$ for some face $F'$ of $\square^{n-1}$, then the proper intersection also follows from that of $Z$ with $F'$ since the map $\phi_{g,s,1}: \A^r_k \times \A^1_k \times F' \to \A^r_k \times \A^1_k \times F'$ is an isomorphism. If $F = \square \times F'$ for some face $F'$ of $\square^{n-1}$, then the map $\A^r_k \times \A^1_k \times \square \times F' \to \A^r_k \times \A^1_k \times F'$ is flat of relative dimension one and hence we get 
\[
\begin{array}{lll} 
 & & \dim(\phi^*_{g,s}(Z) \cap (\A^r_k \times \A^1_k \times \square \times F'))  =  \dim(\phi^*_{g,s}(Z \cap F')) \\
& = & \dim(Z \cap F') +1  \le  \dim(Z) + 1 - {\rm codim}_{\square^{n-1}}(F') \\
& = & \dim(\phi^*_{g,s}(Z)) - {\rm codim}_{\square^n}(\square \times F')  =  \dim(\phi^*_{g,s}(Z)) - {\rm codim}_{\square^n}(F).
\end{array}
\]
This proves the desired proper intersection of $\phi^*_{g,s}(Z)$. 
\end{proof}

\begin{lem}\label{lem:sfs-Ar-1}
Assume that $n =1$. For $g \in \A^r_k \setminus \{0\}$ and $s \gg 0$ as in Lemma \ref{lem:monic}, $\phi^*_{g,s,1}(Z)$ is finite and surjective over $\A^r_{k(g)}$.
\end{lem}

\begin{proof}
Since $\mathbb{A}^r_k \times \mathbb{A}^1_k$ is factorial, there exists an irreducible polynomial $f(\un{x}, t) \in k[\un{x}, t]$ such that $Z= \Spec ( k[ \un{x}, t]/ (f (\un{x}, t)))$. The modulus condition mandates that this cycle does not intersect the divisor $\{ t = 0 \}$ in $\mathbb{A}^r_k \times \mathbb{A}^1_k$, so that after scaling $f$ by a constant in $k^{\times}$, we must have $f = t h -1$ for some $h(\un{x}, t) \in k[\un{x}, t]$. By Lemma \ref{lem:monic}, $\phi_{g,s}^{\sharp} (th -1)$ is monic in $t$ for $g \in \mathbb{A}^r_k \setminus \{0\}$ and $s \gg 0$ up to scaling by a unit in $k(g)^{\times}$. This is equivalent to saying that $\phi^*_{g,s,1} (Z_{k(g)}) \to \mathbb{A}^r_{k(g)}$ is finite. As both have the same dimension and $\mathbb{A}^r_{k(g)}$ is integral, this morphism is automatically surjective.
\end{proof}

\subsection{The three types of cycles}

In order to generalize Lemma \ref{lem:sfs-Ar-1} to $n \ge 2$ case, we need to consider three types of cycles.

\begin{lem}\label{lem:cat1}
Suppose that the projection to the first factor $ Z \to \mathbb{A}^r_k$ is dominant. Then there is a dense open subset $U \subset \mathbb{A}_k ^r$ such that each $g \in U$ and integer $s>0$, the projection to the first factor $\phi^*_{g,s,1} (Z_{k(g)}) \to \mathbb{A}^r_{k(g)}$ is still dominant. 
\end{lem}

\begin{proof}This is immediate from the definition of $\phi_{g,s}$.
\end{proof}

\begin{lem}\label{lem:cat2}
Assume that $(a)$ the projection $q_n: Z \to \mathbb{A}^r_k$ is not dominant, while $(b)$ the projection $pr_2: Z \to \mathbb{A}^1_k$ is dominant. Then there is a dense open subset $U \subset \mathbb{A}^r_k$ such that for each $g \in U$ and $s>0$,  we have
\begin{enumerate}
\item $\dim (q_n (\phi^*_{g,s,1}(Z_{k(g)}))) = \dim (q_n(Z_{k(g)})) + 1$ and 
\item the projection $pr_2: \phi^*_{g,s,1}(Z_{k(g)}) \to \mathbb{A}^1_{k(g)}$ is dominant.
\end{enumerate}
\end{lem}

\begin{proof}
By $(b)$, the map $pr_2$ is a dominant morphism to a regular curve, thus it is flat by \cite[Proposition III-9.7, p.256]{Hartshorne}. In particular, $pr_2 (Z) \subset \mathbb{A}^1_k$ is a dense open subset. For each $g \in \mathbb{A}^r_k$ and $s>0$, we have a surjection $\Phi: q_n (Z_{k(g)}) \times pr_2 (Z_{k(g)}) \to q_n (\phi^*_{g,s,1} (Z_{k(g)}))$, given by sending $(x,t)$ to $x + t^{s (m+1)} g$. Thus, $\dim \ q_n (\phi^*_{g,s,1} (Z_{k(g)})) \leq \dim \  q_n (Z_{k(g)}) + 1$.

On the other hand, for each fixed closed point $t_0 \in pr_2 (Z)$, the set $\Phi (q_n (Z_{k(g)}), t_0)$ has the same dimension as that of $q_n(Z_{k(g)})$, while it is an equidimensional proper closed subset of $q_n(\phi^*_{g,s,1} (Z_{k(g)}))$ when $g$ is a general member, i.e., in an open subset of $\A^r_k$. Since $pr_2(Z)$ is dense open in $\A^1_k$ and hence of positive dimension, we must have $\dim (q_n (\phi^*_{g,s,1} (Z_{k(g)})) > \dim (q_n (Z_{k(g)}))$. This proves (1). The property (2) is obvious because $\phi_{g,s}$ does not modify the $\mathbb{A}^1_k$-coordinate.
\end{proof}

\begin{lem}\label{lem:cat3}
Assume that neither of the projections $q_n: Z \to \mathbb{A}^r_k$ and $pr_2: Z \to \mathbb{A}^1_k$ is dominant. Let $s \ge 1$ be any integer. Then there is a dense open subset $U\subset \mathbb{A}^r_k$ such that for each $g \in U$, there is an open neighborhood $\mathcal{W}_g \subset \mathbb{A}^r_{k(g)}$ of $\Sigma$ such that $\phi^*_{g,s,1}(Z_{k(g)})$ restricted over $\mathcal{W}_g$ is empty.
\end{lem}

\begin{proof}
Since $pr_2: Z \to \mathbb{A}^1_k$ is not dominant and $Z$ is irreducible, $pr_2 (Z)$ must be a singleton closed subset $\{ t_0 \}$. By the modulus condition that $Z$ satisfies, we must have $t_0 \not = 0$ and $Z \subset \A^r_k \times \{t_0\} \times \square^{n-1}_k$. It is therefore sufficient to prove the lemma by replacing $k$ by $k(t_0)$ and $\Sigma$ by $\pi^{-1}_{t_0}(\Sigma)$, where $\pi_{t_0}: \Spec(k(t_0)) \to \Spec(k)$ is the base change. We can thus assume that $t_0 \in k^{\times}$. Consider the proper closed subset $\overline{q_n (Z)} \subset \mathbb{A}^r_k$ of dimension $<r$ and the dense open complement $\mathcal{U}_0= \mathbb{A}^r_k \setminus \overline{q_n (Z)}$.

Because $Z$ restricted over $\mathcal{U}_0$ is empty, we see that the translation $\phi_{g,s,1}^* (Z_{k(g)})$ restricted to the translation $ \phi_{g,s,1}^* (\mathcal{U}_0)$ is empty for every $g \in \mathbb{A}_k ^r$. Hence, it is enough to show that for an open subset $U \subset \mathbb{A}_k ^r$, the set $\mathcal{W}_g:= \phi_{g,s,1} ^*(\mathcal{U}_0)$ contains $\Sigma$ for each $g \in U$. However, this is evident because $\Sigma$ is a finite set of closed point of $\mathbb{A}_k ^r$ while $\mathcal{U}_0$ is a dense open subset of $\mathbb{A}_k ^r$, and $\phi_{g,s,1}^*$ is translation by a nonzero constant factor ($t_0 ^{s (m+1)}$) of $g$. This proves the lemma.
\end{proof}

\subsection{Key lemmas}

The key to our \sfs-moving lemma for the localizations of $\A^r_k$ are the following two lemmas.

Let $W \subset \A^r_k \times B_n$ be a reduced closed subscheme and let $\overline{W}$ be its closure in $\A^r_k \times \overline{B}_n$ with reduced closed subscheme structure. We let $\overline{W}^o = \overline{W} \cap (\A^r_k \times \G_{m,k} \times \overline{\square}^{n-1})$. We fix a closed point $x \in \Sigma$ and integers $m,s \geq 1$.  

Define 
$$
\tuborg  & P_1: \mathbb{A}_k ^r \times \mathbb{A}_k ^r \times \mathbb{A}_k ^1 \times \overline{\square} ^{n-1} \to \mathbb{A}_k ^r, \\
 &  P_2: \mathbb{A}_k ^r \times \mathbb{A}_k ^r \times \mathbb{A}_k ^1 \times \overline{\square} ^{n-1} \to  \mathbb{A}_k ^r \times \mathbb{A}_k ^1 \times \overline{\square} ^{n-1}\sluttuborg$$
 to be the projection to the first factor, and the projection to the remaining factors. For a fixed $x \in \mathbb{A}_k ^r$, define $\iota_x: \mathbb{A}_k ^r \times \mathbb{A}_k ^1 \times \overline{\square}^{n-1} \to \mathbb{A}_k ^r \times \mathbb{A}_k ^r \times \mathbb{A}_k ^1 \times \overline{\square}^{n-1}$ to be the map $(g, t, \un{y}) \mapsto (g, x + t^{s(m+1)}g, t, \un{y})$. Let $\theta_x:= P_2 \circ \iota_x$ and $\omega_{\overline{W}, x} := (P_1 \circ \iota_x) |_{\theta_x ^{-1} (\overline{W})}$, where $\theta^{-1}_x(\overline{W})$ is given its reduced induced closed subscheme structure. We then have the commutative diagram
\begin{equation}\label{eqn:fs-generic}
\xymatrix@C.6pc{
&  \theta^{-1}_x(\overline{W}^o) \ar[d] \ar@{^{(}->}[r] & \A^r_k \times \G_{m,k} \times \square^{n-1} \ar@{^{(}->}[r]^-{\iota_x} \ar[d] & \A^r_k \times \A^r_k \times \G_{m,k} \times \square^{n-1} \ar[r]^-{P_2} \ar[d] & \A^r_k \times \G_{m,k} \times \square^{n-1} \ar[d] & \overline{W}^0 \ar@{^{(}->}[l] \ar[d] \\
\theta^{-1}_x(W) \ar@{^{(}->}[ur] \ar@{^{(}->}[r] \ar@/_2pc/[drrr]_{\omega_{W,x}} & \theta^{-1}_x(\overline{W}) \ar@{^{(}->}[r] \ar[drr]_{\omega_{\overline{W},x}} & \A^r_k \times \A^1_k \times \overline{\square}^{n-1} \ar@{^{(}->}[r]^-{\iota_x} & \A^r_k \times \A^r_k \times \A^1_k \times \overline{\square}^{n-1} \ar[r]^-{P_2} \ar[d]^-{P_1} &  \A^r_k \times \A^1_k \times \overline{\square}^{n-1} & \overline{W} \ar@{^{(}->}[l] \\
& & & \A^r_k, & &}
\end{equation}
where the top row's $\iota_x$, $ P_2$ are the restrictions of the second row, and $\omega_{W, x}$ is the natural composition. The vertical arrows are canonical open immersions. It is easy to check that $\iota_x$ is a closed immersion and  $\theta_x$ is an isomorphism on the top row. Using \eqref{eqn:Kai-homotopy} and \eqref{eqn:fs-generic}, one immediately verifies the following observation which we shall use often.

\begin{lem}\label{lem:fs-generic-*}
Let $x \in \mathbb{A}_k ^r$ be fixed. Then for each $g \in \A^r_k$, the map $\omega^{-1}_{\overline{W},x}(g) \to \phi^*_{g,s,1}(\overline{W})$, $(g,t, \un{y}) \mapsto (x, t, \un{y})$, is an isomorphism. The same holds for $W$ and $\overline{W}^o$ as well.
\end{lem}

Another lemma we shall use is the following.

\begin{lem}[{\cite[Lemma 1.2]{Bl1}}]\label{lem:blochlemma} Let $X$ be an algebraic $k$-scheme and $G$ a connected algebraic $k$-group acting on $X$. Let $A, B \subset X$ be closed subsets, and assume the fibers of the map $G \times A \to X$, $(g,a) \mapsto g\cdot a$ all have the same dimension, and that this map is dominant.

Moreover, suppose that for an over-field $K \supset k$ and a $K$-morphism $\psi: X_K \to G_K$, there is a nonempty open subset $U \subset X$ such that for every $x \in U_K$, a scheme point, we have $\tr.\deg_k k (\varphi \circ \psi (x), \pi (x)) \geq \dim(G)$, where $\pi: X_K \to X_k$ and $\varphi: G_K \to G_k$ are
the projection maps. Define $\phi: X_K \to X_K$ by $\phi (x) = \psi (x) \cdot x$ and suppose $\phi$ is an isomorphism. Then the intersection $\phi (A_K \cap U_K) \cap B_K$ is proper.
\end{lem}

\subsection{Applications of the key lemmas}

We apply the above two lemmas to our cycle $Z$ and various other closed subsets associated to it. 
Let $\eta \in \mathbb{A}^r _k$ denote the generic point and let $K:= k(\eta)$. We can regard $\eta \in \mathbb{A}_k ^r (K)$. 
Apply Lemma \ref{lem:blochlemma} with  
\[
X  = \A^r_k \times \A^1_k \times \overline{\square}^{n-1}_k, \ G = \A^r_k, \ \psi(\un{x}, t, \un{y}) = (\eta)t^{s(m+1)}, \ A = \Sigma \times \A^1_k \times \overline{\square}^{n-1}, \ \mbox{and} \ B = \overline{Z},
\]
where $G$ acts on $\A^r_k \times \A^1_k \times \overline{\square}^{n-1}$ by $g \cdot (\un{x}, t, \un{y}) = (g + \un{x}, t, \un{y})$. We let $\phi: X_K \to X_K$ be given by $\phi(\un{x}, t, \un{y}) = ((\eta)t^{s(m+1)} + \un{x}, t, \un{y})$. One checks immediately that the conditions of Lemma \ref{lem:blochlemma} are satisfied and we conclude that $\phi(A_K) \cap \overline{Z}_K$ has dimension at most zero. Comparing this with \eqref{eqn:fs-generic} and using Lemma \ref{lem:fs-generic-*}, this is equivalent to saying that the generic fiber of $\omega_{\overline{Z},x}$ is finite for every $x \in \Sigma$. 

It follows that if $Z'$ is an irreducible component of $\theta^{-1}_x(\overline{Z})$, then either the map $\omega_{\overline{Z},x}: Z' \to \A^r_k$ is not dominant or it is dominant and generically quasi-finite. In the dominant case, Chevalley's theorem on fiber dimensions (e.g., see \cite[Exercise II-3.22, p.95]{Hartshorne}) tells us that we must have $\dim(Z') = r$ and $Z' \to \A^r_k$ is generically finite. In any case, it follows that there is a dense open subset of $\A^r_k$ over which $Z' \to \A^r_k$ is quasi-finite (with possibly empty fibers). 

By taking the finite intersection of such dense open subsets, running over all irreducible components of $\theta^{-1}_x(\overline{Z})$ and all $x \in \Sigma$, we conclude that there is a dense open subset $U \subset \A^r_k$ such that for each $x \in \Sigma$, the map $\omega^{-1}_{\overline{Z},x}(U) \to U$ is quasi-finite. Using Lemma \ref{lem:fs-generic-*}, equivalently we get:

\begin{lem}\label{lem:fs-Kai-*}
For any integer $s \ge 1$, there is a dense open subset $U \subset \A^r_k$ such that for every $g \in U$, the set $(\Sigma \times \overline{B}_n)_{k(g)} \cap \phi^{*}_{g,s,1}(\overline{Z})_{k(g)} =  (\Sigma \times \overline{B}_n)_{k(g)} \cap \overline{\phi^{*}_{g,s,1}({Z}_{k(g)})}$ is finite.
\end{lem}

We can now show the following.

\begin{lem}\label{lem:fs-Kai}
Let $s \gg 0$ be as in Lemma \ref{lem:monic}. Assume that $Z$ is either dominant over $\A^r_k$ or restricts to zero on $V$. Then we can find a dense open $U \subset \A^r_k$ such that for $g \in U$, the scheme $\phi^*_{g,s,1}(Z)|_V$ is either empty or finite and surjective over $V$. 
\end{lem}

\begin{proof}
We can assume $n \ge 2$ by Lemma \ref{lem:sfs-Ar-1}. We let $U_1 \subset \A^r_k$ be the intersection of open subsets obtained in Lemmas \ref{lem:cat3} and \ref{lem:fs-Kai-*}. We can therefore assume that $\overline{\phi^*_{g,s,1}(Z_{k(g)})} \to \A^r_{k(g)}$ is dominant for all $g \in U_1$. 

For $g \in U_1$, there is a commutative diagram
\begin{equation}\label{eqn:sfs-affine*-0}
\xymatrix@C1pc{
\A^r_{k(g)} \times \mathbb{A}^1_{k(g)} \times \overline{\square}^{n-1}_{k(g)} \ar[r]^-{p_n} \ar[d]_{\phi_{g,s,1}} & \A^r_{k(g)} \times \mathbb{A}^1_{k(g)} \ar[d]^{\phi_{g,s,1}} \\
\A^r_{k(g)} \times \mathbb{A}^1_{k(g)} \times \overline{\square}^{n-1}_{k(g)} \ar[r]^-{p_n} & \A^r_{k(g)} \times \mathbb{A}^1_{k(g)},}
\end{equation} 
where the horizontal arrows are the projections.

If we let $W= p_n(\overline{Z}_{k(g)})$, it follows from Lemma \ref{lem:fs-Kai-*} that the composite map $\overline{\phi^*_{g,s,1}(Z_{k(g)})} \to \phi^*_{g,s,1}(W) \to \A^r_{k(g)}$ is quasi-finite over $\Sigma_{k(g)}$. Since $\overline{\phi^*_{g,s,1}(Z_{k(g)})} \to \A^r_{k(g)}$ is dominant by Lemma \ref{lem:cat1}, it follows from Chevalley's theorem on fiber dimensions (see \cite[Exercise II-3.22, p.95]{Hartshorne}) that there is an open neighborhood $U_g \subset \A^r_{k(g)}$ of $\Sigma_{k(g)}$ over which the map $\overline{\phi^*_{g,s,1}(Z_{k(g)})} \to \A^r_{k(g)}$ is quasi-finite with non-empty fibers. We then get maps $\overline{\phi^*_{g,s,1}(Z_{k(g)})} \cap q^{-1}_n(U_g) \xrightarrow{p_n}  \phi^*_{g,s,1}(W) \cap q^{-1}(U_g) \xrightarrow{q} U_g$, where the first map is projective and the composite map is quasi-finite with non-empty fibers. This implies that the first map is also quasi-finite, and hence, it is finite. Since $\overline{Z} \to W$ is dominant, so is the map $\overline{\phi^*_{g,s,1}(Z_{k(g)})} \to \phi^*_{g,s,1}(W)$ by \eqref{eqn:sfs-affine*-0}. It follows that $\overline{\phi^*_{g,s,1}(Z_{k(g)})} \to \phi^*_{g,s,1}(W)$ is finite and surjective over $U_g$.

On the other hand, we have shown in Lemma \ref{lem:sfs-Ar-1} that ${\phi}^*_{g,s,1}(W) \to \A^r_{k(g)}$ is finite and surjective over $\A^r_k$ for our choice of $s \gg 0$ and $g \in \A^r_k \setminus \{0\}$. We conclude that there is an open neighborhood $U_g \subset \A^r_{k(g)}$ of $\Sigma_{k(g)}$ over which $\phi^*_{g,s,1}(\overline{Z}_{k(g)}) \to \A^r_{k(g)}$ is finite and surjective.

To show this property for $\phi^*_{g,s,1}({Z}_{k(g)})$, we fix $x \in \Sigma$ and use the diagram \eqref{eqn:fs-generic} where we take $\overline{W} = Y := \overline{Z} \setminus Z$. To understand the generic fiber of $\omega_{Y,x}$, we apply Lemma \ref{lem:blochlemma} with
\begin{equation}\label{eqn:sfs-affine-0}
X = \A^r_k \times \mathbb{A}^1_k \times \overline{\square}^{n-1}_k, G = \A^r_k \ ,   \psi(\un{x},t, \un{y}) = (\eta) t^{s(m+1)}, A = \Sigma \times \mathbb{A}^1_k \times \overline{\square}^{n-1}_k,  B = Y,
\end{equation}
where $G$ acts on $\A^r_k \times \mathbb{A}^1_k \times \overline{\square}^{n-1}_k$ by $g \cdot (\un{x}, t, \un{y}) = (g + \un{x} , t, \un{y})$ as before. One checks immediately that the conditions of Lemma \ref{lem:blochlemma} are satisfied. It follows that the intersection $\phi_{\eta,s,1}(A_{k(\eta)}) \cap B_{k(\eta)}$ is proper. By a dimension counting, this means that $\phi_{\eta,s,1}(A_{k(\eta)}) \cap B_{k(\eta)} = \emptyset$. Equivalently, we have $A_{k(\eta)} \cap \phi^*_{\eta,s,1}(Y_{k(\eta)}) = \emptyset$. We conclude by Lemma \ref{lem:fs-generic-*} that for every $x \in \Sigma$, the map $\omega_{Y,x}:\theta^{-1}_x(Y) \to \A^r_k$ is not dominant. We can therefore find a dense open subset $U \subset U_1 \subset \A^r_k$ such that the fiber of $\omega_{Y,x}:\theta^{-1}_x(Y) \to \A^r_k$ is empty over $U$ for every $x \in \Sigma$. In other words, for every $g \in U$, the intersection $\phi^*_{g,s,1}(Y_{k(g)}) \cap A_{k(g)} =(\overline{\phi^*_{g,s,1}({Z}_{k(g)})} \setminus \phi^*_{g,s,1}({Z}_{k(g)})) \cap A_{k(g)}$ is empty. But this means that the map $\phi^*_{g,s,1}({Z}_{k(g)}) \to \A^r_{k(g)}$ is finite and surjective over an affine neighborhood of $\Sigma_{k(g)}$ (see Lemma \ref{lem:fs-smooth}).
\end{proof}

\begin{lem}\label{lem:sfs-Kai}
Assume that $Z \in \TZ^n _{\Sigma} (\A^r_k,n;m)$ is an irreducible cycle such that $Z \to \A^r_k$ is finite and surjective over an affine neighborhood of $\Sigma$. We can then find $s \gg 0$ and a dense open subset $U \subset \A^r_k$ such that for each $1 \le j \le n$ and for each $g \in U$, the scheme $(\phi^*_{g,s,1}(Z_{k(g)}))^{(j)}$ is regular over an affine neighborhood of $\Sigma_{k(g)}$.
\end{lem}

\begin{proof}
We take $W = Z_{\rm sing}$, the singular locus of $Z$, in \eqref{eqn:fs-generic} and consider the map $\omega_{Z_{\rm sing},x}: \theta^{-1}_x(Z_{\rm sing}) \to \A^r_k$ for $x \in \Sigma$. We had seen previously that the map $\theta_x$ on the top row of \eqref{eqn:fs-generic} is an isomorphism. In particular, the map $\theta_x:\theta^{-1}_x(Z_{\rm sing}) \to Z_{\rm sing}$ is an isomorphism. But this implies that $\dim(\theta^{-1}_x(Z_{\rm sing})) = \dim(Z_{\rm sing}) \le r-1$. It follows that the map $\omega_{Z_{\rm sing},x}: \theta^{-1}_x(Z_{\rm sing}) \to \A^r_k$ is not dominant. We can therefore find a dense open subset $U \subset \A^r_k$ such that the fibers of $\omega_x$ over $U$ are empty. By shrinking $U$ further, we can assume that this holds for all $x \in \Sigma$.

It follows from Lemma \ref{lem:fs-generic-*} that for every $g \in U$, the closed subscheme $(\phi^*_{g,s,1}(Z_{k(g)}))_{\rm sing} = \phi^*_{g,s,1}((Z_{k(g)})_{\rm sing}) = \phi^*_{g,s,1}((Z_{\rm sing})_{k(g)})$ does not meet $(\Sigma \times B_n)_{k(g)}$. Here, the last equality uses the perfectness of $k$. But this means that $\phi^*_{g,s,1}(Z_{k(g)})$ is regular at all points lying over $\Sigma_{k(g)}$. By choosing $s \gg 0$ as in Lemma \ref{lem:monic}, shrinking $U$ further, and using Lemma \ref{lem:fs-Kai}, we can assume that $\phi^*_{g,s,1}(Z_{k(g)})$ is finite and surjective over an affine neighborhood of $\Sigma_{k(g)}$. But then $\phi^*_{g,s,1}(Z_{k(g)})$ must be regular over an affine neighborhood of $\Sigma_{k(g)}$.

Let $Z^{(j)} \subset \A^r_k \times B_j$ be the projection of $Z$ to $B_j$ as in \S \ref{sec:SFS*} for $1 \le j \le n$. Since $Z \to \A^r_k$ is finite and surjective over an affine neighborhood of $\Sigma$, each $Z^{(j)}$ is also finite and surjective over an affine neighborhood of $\Sigma$. We can therefore repeat the above process successively for each $Z^{(j)}$ by shrinking $U$ further each time. In the end, we get a dense open subset $U \subset \A^r_k$ such that each $1 \le j \le n$ and for each $g \in U$, the scheme $\phi^*_{g,s,1}(Z^{(j)}_{k(g)})$ is regular over a common affine neighborhood of $\Sigma_{k(g)}$. Since the diagram 
\begin{equation}\label{eqn:sfs-affine*-1}
\xymatrix@C1pc{
\A^r_{k(g)} \times \mathbb{A}^1_{k(g)} \times {\square}^{n-1}_{k(g)} \ar[r]^-{\pi_j} \ar[d]_{\phi_{g,s,1}} & \A^r_{k(g)} \times \mathbb{A}^1_{k(g)} \times {\square}^{j-1}_{k(g)} \ar[d]^{\phi_{g,s,1}} \\
\A^r_{k(g)} \times \mathbb{A}^1_{k(g)} \times \overline{\square}^{n-1}_{k(g)} \ar[r]^-{\pi_j} & \A^r_{k(g)} \times \mathbb{A}^1_{k(g)} \times {\square}^{j-1}_{k(g)}}
\end{equation}
commutes and the vertical maps are isomorphisms, it follows that $\phi^*_{g,s,1}(Z^{(j)}_{k(g)}) = (\phi^*_{g,s,1}(Z_{k(g)}))^{(j)}$. We have therefore shown that there is a dense open subset $U \subset \A^r_k$ such that for every $g \in U$ and $1 \le j \le n$, the scheme $(\phi^*_{g,s,1}(Z_{k(g)}))^{(j)}$ is regular over a common affine neighborhood of $\Sigma_{k(g)}$. This finishes the proof.
\end{proof}

\begin{lem}\label{lem:support-Kai}
For every integer $s \ge 1$, there is a dense open subset $U \subset \A^r_k$ such that for every $g \in U$, one has $\phi^*_{g,s,1}(Z_{k(g)}) \cap (\Sigma \times \A^1_k \times F)_{k(g)} = \emptyset$ for every proper face $F$ of $\square^{n-1}$.
\end{lem}

\begin{proof}
We let $F$ be a proper face of $\square^{n-1}$ and let $W = Z \cap (\A^r_k \times \A^1_k \times F)$. We fix a point $x \in \Sigma$ and consider the diagram (see \eqref{eqn:fs-generic}):
\begin{equation}\label{eqn:fs-generic-0}
\xymatrix@C.8pc{
\theta^{-1}_x(W) \ar[drr]_{\omega_{W,x}} \ar@{^{(}->}[r] & \A^r_k \times \G_{m,k} \times F \ar@{^{(}->}[r]^-{\iota_x} & \A^r_k \times \A^r_k \times \G_{m,k} \times F \ar[r]^-{P_2}  \ar[d]^-{P_1} &  \A^r_k \times \G_{m,k} \times F  & W  \ar@{^{(}->}[l] \\
& & \A^r_k. & &}
\end{equation}

As in \eqref{eqn:fs-generic}, the map $\theta_x = P_2 \circ \iota_x$ is an isomorphism. Note also that (see Lemma \ref{lem:fs-generic-*}) for any $g \in \A^r_k$, the map $\omega^{-1}_{W,x}(g) \to \phi^*_{g,s,1}(Z) \cap (\{x\} \times \A^1_k \times F)$, which sends $(g,t, \un{y})$ to $(x, t, \un{y})$, is an isomorphism. It follows therefore %from \cite[Lemma 3.4]{Kai}
 that the map $\omega_{W,x}$ is not dominant. Equivalently, there exists a dense open $U \subset \A^r_k$ such that the fibers of $\omega_{W,x}$ over $U$ are empty. Shrinking $U$ further if necessary, we can assume that this happens for all $x \in \Sigma$. It is clear that for every $g \in U$, the set $\phi^*_{g,s,1}(Z_{k(g)}) \cap (\Sigma \times \A^1_k \times F)_{k(g)}$ is empty. This proves the lemma.
\end{proof}

\subsection{The proof of the moving lemma for affine spaces}

We can now prove the main result of this section, the \sfs-moving lemma for the localizations of $\A^r_k$. We begin with the following intermediate modification step.

\begin{lem}\label{lem:partial-dominance}
Let $k$ be an infinite field and let $\alpha \in \TZ^n(\A^r_k, n;m)$. Let $V= \Spec (\mathcal{O}_{\mathbb{A}^r_k, \Sigma})$ for a finite subset $\Sigma \subset \mathbb{A}^r_k$ of closed points, with the localization map $j: V \to \A^r_k$. Assume that $\partial (j^*(\alpha)) = 0$. Then there are cycles $\beta \in \TZ^n(\A^r_k,n;m)$ and $\gamma \in \TZ^n(\A^r_k, n+1;m)$ with $\partial(j^*(\gamma)) = j^*(\alpha) - j^*(\beta)$ such that each component of $\beta$ is either dominant over $\A^r_k$ or restricts to zero on $V$.
\end{lem}

\begin{proof}
We choose an integer $s \gg 0$ which is at least as large as the the integer $s(Z)$ and the one chosen in Lemmas \ref{lem:cat2} and \ref{lem:cat3} for every irreducible component $Z$ of $\alpha$. It follows from Lemma \ref{lem:good-intersection} that $\phi^*_{g,s}(\alpha)$ intersects all faces of $\square^{n}$ properly. Taking the face $F = \{1\} \times \square^{n-1}$ (and using the containment lemma \cite[Proposition 2.2]{KP3}), we see that $\phi^*_{g,s,1}(\alpha) \in \TZ^{n}(\A^r_{k(g)}, n; m)$ for all $g \in \A^r_k$. We can also assume that $s \gg 0$ is large enough so that Lemma \ref{lem:good-intersection} holds also for each boundary of each component of $\alpha$.

We let $U \subset \A^r_k$ be any dense open which is contained in the intersection of the ones given by Lemmas \ref{lem:cat2} and \ref{lem:cat3} for all irreducible components of $|\alpha|$. We let $g \in U(k)$ be any element. It follows by our choice of $g$ that if $Z$ is a component of $\alpha$, then $\phi^*_{g,s,1}(Z)$ is either dominant over $\A^r_k$, or it restricts to zero on $V$, or satisfies conditions (1) and (2) of Lemma \ref{lem:cat2}.

We now compute
\begin{eqnarray}
\notag \phi^*_{g,s} \circ \partial (\alpha) & = &\phi^*_{g,s} (\sum_{i=1} ^{n-1}(-1)^i (\partial^1_i - \partial^0_i) (\alpha))  {=}^\dagger \sum_{i=1} ^{n-1} (-1)^i(\partial^1_{i+1} - \partial^0_{i+1}) (\phi^*_{g,s}(\alpha)) \\
\notag & = & -  \sum_{i=2} ^n (-1)^i(\partial^1_{i} - \partial^0_{i}) (\phi^*_{g,s}(\alpha)),
\end{eqnarray}
where ${=}^\dagger$ follows from \eqref{eqn:Kai-homotopy}.
On the other hand, we have
\begin{eqnarray}
\notag \partial \circ \phi^*_{g,s}(\alpha) & = &  \sum_{i=1} ^n (-1)^i(\partial^1_{i} - \partial^0_{i}) (\phi^*_{g,s}(\alpha)) \\
\notag & = & (-1)(\partial^1_{1} - \partial^0_{1})(\phi^*_{g,s}(\alpha)) + \sum_{i=2} ^n (-1)^i(\partial^1_{i} - \partial^0_{i}) (\phi^*_{g,s}(\alpha)).
\end{eqnarray}

It follows that $\partial (\phi^*_{g,s}(\alpha)) + \phi^*_{g,s}(\partial (\alpha))= (\partial^0_{1} - \partial^1_{1})(\phi^*_{g,s}(\alpha)) =\alpha - \phi^*_{g,s,1}(\alpha)$. Lemma \ref{lem:good-intersection} says that $\phi^*_{g,s}(\alpha) \in \TZ^n(\A^r_k, n+1;m)$. If we let $\gamma = \phi^*_{g,s}(\alpha)$ and $\beta = \phi^*_{g,s,1}(\alpha)$, we see that $\partial(j^*(\gamma)) = j^*(\alpha) - j^*(\beta)$. It also follows that $\partial(j^*(\beta)) = 0$.

We now replace $\alpha$ by $\beta$ in $\TZ^n(\A^r_k,n;m)$ and repeat the above process. It follows from Lemmas \ref{lem:cat1}, \ref{lem:cat2} and \ref{lem:cat3} that after finite steps, we arrive at new cycles $\beta \in \TZ^n(\A^r_k,n;m)$ and $\gamma \in \TZ^n(\A^r_k,n+1;m)$ such that $\partial(j^*(\gamma)) = j^*(\alpha) - j^*(\beta)$. Moreover, each component of $\beta$ is either dominant over $\A^r_k$ or restricts to zero on $V$.
\end{proof}

\begin{thm}\label{thm:Ar-spread-case}
Let $k$ be an infinite perfect field and let $\alpha \in \TZ^n(\A^r_k, n;m)$. Let $V= \Spec (\mathcal{O}_{\mathbb{A}^r_k, \Sigma})$ for a finite subset $\Sigma \subset \mathbb{A}^r_k$ of closed points, with the localization map $j: V \to \A^r_k$. Assume that $\partial (j^*(\alpha)) = 0$. Then there are cycles $\beta \in \TZ^n_{\sfs} (V,n;m)$ and $\gamma \in \TZ^n(V, n+1;m)$ such that $\partial(\gamma) = j^*(\alpha) - \beta$.
\end{thm}

\begin{proof}
By applying Lemma \ref{lem:partial-dominance} and removing those components of the resulting new cycle $\alpha$ which restrict to zero on $V$, we can assume that every component of $\alpha$ is dominant over $\A^r_k$. Note that this does not change $\partial(j^*(\alpha))$.

We now choose an integer $s \gg 0$ which is at least as large as the the integer $s(Z)$ and the one chosen in Lemmas \ref{lem:fs-Kai} and \ref{lem:sfs-Kai} for every irreducible component $Z$ of $\alpha$. It follows from Lemma \ref{lem:good-intersection} that $\phi^*_{g,s}(\alpha)$ intersects all faces of $\square^{n}$ properly and $\phi^*_{g,s,1}(\alpha) \in \TZ^{n}(\A^r_{k(g)}, n; m)$ for all $g \in \A^r_k$ (see the proof of Lemma \ref{lem:partial-dominance}). We can also assume that $s \gg 0$ is large enough so that Lemma \ref{lem:good-intersection} holds also for each boundary of each component of $\alpha$.

We let $U \subset \A^r_k$ be any dense open which is contained in the intersection of the ones given by Lemmas \ref{lem:fs-Kai}, \ref{lem:sfs-Kai} and \ref{lem:support-Kai} for all irreducible components of $\alpha$. Since $U$ is rational and $k$ is infinite, $U(k)$ is a dense subset of $U$. We let $g \in U(k)$ be any element. We claim that $j^*(\phi^*_{g,s,1}(\alpha)) \in \TZ^n_{\sfs} (V,n;m)$, where $\phi^*_{g,s}(-)$ is defined on $\TZ^n(\A^r_k, n;m)$ by the usual linear extension. By Lemmas \ref{lem:fs-Kai} and \ref{lem:sfs-Kai}, we only need to show that $\phi^*_{g,s,1}(\alpha) \in \TZ^n_{\Sigma} (\A^r_k,n;m)$. But this is equivalent to showing that $(\Sigma \times \A^1_k \times F) \cap |\phi^*_{g,s,1}(\alpha)| = \emptyset$ for every proper face $F$ of $\square^{n-1}$, which in turn follows from Lemma \ref{lem:support-Kai}. The claim is thus proven.

A computation identical to the one in the proof of Lemma \ref{lem:partial-dominance} shows that $\partial (\phi^*_{g,s}(\alpha)) + \phi^*_{g,s}(\partial (\alpha)) = (\partial^0_{1} - \partial^1_{1})(\phi^*_{g,s}(\alpha)) = \alpha - \phi^*_{g,s,1}(\alpha)$. Lemma \ref{lem:good-intersection} says that $\phi^*_{g,s}(\alpha) \in \TZ^n(\A^r_k, n+1;m)$. If $\partial (j^*(\alpha)) = 0$, we can set $\gamma =  j^* \circ \phi^*_{g,s}(\alpha)$ and $\beta = j^*(\phi^*_{g,s,1}(\alpha))$. We get $\partial(\gamma) = j^*(\alpha) - \beta$ and we have shown above that $\beta \in \TZ^n_{\sfs}(V, n;m)$. The theorem is now proven.
\end{proof}

\begin{remk}\label{remk:sfs-affine-Chow}
The proof of Theorem \ref{thm:Ar-spread-case} (where we take $n \ge 2$, replace $B_n$ by $\square^{n-1}$ and take $s =0$ everywhere in the proof) also shows that if $n \ge 1$ and $\alpha \in z^n(\A^r_k,n)$ is a higher Chow cycle with $\partial(j^*(\alpha)) = 0$, then we can find $\gamma \in z^n(V,n+1)$ and $\beta \in z^n_{\sfs}(V,n)$ such that $\partial(\gamma) = j^*(\alpha) - \beta$. Note that 
$n = 0$ case of this result is trivial.
\end{remk} 

\section{The \fs-property of residual cycles}\label{sec:LP}

Let $k$ be an infinite perfect field. In this section, we discuss some results on linear projections in projective spaces, and show how these projections can be used to equip the residual cycle of a given cycle with certain finiteness properties over the base scheme. The main result of \S \ref{sec:LP} is Theorem \ref{thm:fs-present-intro}. It will be used later in proving the \fs-moving lemma (see Lemma \ref{lem:fs-move-*}), a precursor to the final \sfs-moving lemma.

For $0 \le n < N$ and a linear subspace $H \subset \P^N_k$ defined over $k$, let $\Gr(n, H)$ be the Grassmannian scheme of $n$-dimensional linear subspaces of $\P^N_k$ contained in $H$. This is a homogeneous space of dimension $(\dim(H)-n)(n+1)$. Unless we specify the field of definition, a linear subspace of $\P^N_k$ will mean a $k$-linear subspace.

Given two closed subschemes $Y, Y' \subset \P^N_k$, let ${\rm Sec}(Y,  Y') \subset \P^N_k$ be the union of all lines $\ell_{yy'}$ joining distinct points $y \in Y, y' \in Y'$. 
In general, we have $\dim({\rm Sec}(Y,Y')) \le \dim(Y) + \dim(Y') + 1$. If $Y = Y'$, the scheme ${\rm Sec}(Y, Y') = {\rm Sec}(Y)$ is the secant variety of $Y$. If $Y' = L$ is a linear subspace, then ${\rm Sec}(Y,L) = C_L(Y)$ is the cone over $Y$ with vertices in $L$.  

\subsection{Containment and avoidance}\label{sec:CAL}
Let $0 \le m \le n < N$ be integers and let $S, T \subset \P^N_k$ be two disjoint subsets. 

\begin{defn}\label{defn:LP-0}
We denote the set of $n$-dimensional linear subspaces of $\P^N_k$ containing $S$ by $\Gr_S(n, \P^N_k)$. We write $\Gr_S(n, \P^N_k)$ as $\Gr_x(n, \P^N_k)$ if $S = \{x\}$ is a closed point. We denote the set of $n$-dimensional linear subspaces of $\P^N_k$ which do not intersect $S$ by $\Gr(S, n, \P^N_k)$. If $S = \{x\}$, we write $\Gr(S, n, \P^N_k)$ as $\Gr(x, n, \P^N_k)$. We let $\Gr_S(T, n, \P^N_k) := \Gr_S(n, \P^N_k) \cap \Gr(T, n, \P^N_k)$. For any linear subspace $L\subset \P_k ^N$, we define $\Gr_S(n, L)$ and $\Gr(T, n, L)$ similarly.
\end{defn}

One checks that, when $M \subset \P_k ^N$ is a linear subspace of dimension $m$, then $\Gr_M(n, \P^N_k)$ is a homogeneous space which is an irreducible closed subscheme of $\Gr(n, \P^N_k)$ of dimension $(N-n)(n-m)$. The following result is elementary. We leave the proof as an exercise.

\begin{lem}\label{lem:elem-0}
Let $N > n$. $(1)$ If $S' \subset S$, then $\Gr(S, n, \P^N_k) \subset \Gr(S', n, \P^N_k)$. $(2)$ For any finite closed set $S \subset \P^N_k$ , $\Gr(S, n, \P^N_k) \subset \Gr(n, \P^N_k)$ is a dense open subset. 
\end{lem}

\begin{lem}\label{lem:fs-open-1}
Let $X \subset \P^N_k$ be a closed subscheme of dimension $r \ge 1$ with $N \gg r$ and let $H \subset \P^N_k$ be a hyperplane, not containing any irreducible component of $X$. Then $\Gr(X, N-r-1, H)$ is a dense open subset of $\Gr(N-r-1, H)$.
\end{lem}

\begin{proof}
Consider the incidence scheme $S = \{(x, L) \in X \times \Gr(N-r-1, H)| x \in L\}$. We have the obvious projection maps $X \overset{\pi_1}{\leftarrow} S \overset{\pi_2}{\rightarrow} \Gr (N-r-1, H).$

Each fiber of $\pi_1$ over $X \setminus ( X \cap H) $ is empty. It is a smooth morphism over $ X \cap H$ with its fiber over $x \in X \cap H$ to be $\Gr_x (N-r-1, H)$ , whose dimension is $( (N-1) - (N-r-1)) (N-r-1 - 0) = r (N-r-1)$. 
It follows that $\dim(S) \leq \dim( X \cap H) + \dim \ \Gr_x ( N-r-1, H) = r-1 + r (N-r-1) = r (N-r) -1$. Thus, $\pi_2(S)$ is a closed subscheme of $\Gr(N-r-1, H)$ of dimension $\leq r(N-r) -1$ which is less than $\dim \ \Gr(N-r-1, H) = r(N-r)$. Hence, $\Gr(X, N-r-1, H) =  \Gr(N-r-1, H) \setminus \pi_2(S)$ is a dense open subset.
\end{proof}

\subsection{Transverse intersection}\label{sec:Transverse}
For a reduced scheme $X$, let $X_{\rm sing} \subset X$ be the singular locus of $X$ and let $X_{\rm sm}$ be its complement. For a closed subscheme $X \subset \P^N_k$, let $\Gr^{\rm tr}(X, n, \P^N_k)$ denote the set of $n$-dimensional linear subspaces which \emph{do not} intersect $X_{\rm sing}$, and whose intersection with $X_{\rm sm}$ is transverse (if not empty). We let $\Gr ^{\rm tr} (X,S, n, \mathbb{P} ^N_k) = \Gr (S, n, \mathbb{P}_k ^N) \cap \Gr ^{\rm tr} (X, n, \mathbb{P}^N_k)$ and $\Gr ^{\rm tr}_S(X, n, \mathbb{P} ^N_k) = \Gr_S(n, \mathbb{P}_k ^N) \cap \Gr ^{\rm tr} (X, n, \mathbb{P}^N_k)$. For a linear subspace $H \subset \P^N_k$, we define $\Gr ^{\rm tr} (X,S, n; H)$ and $\Gr ^{\rm tr} _S(X, n; H)$ similarly. 

\begin{lem}\label{lem:fs-open-2}
Let $r \ge 2$ be an integer and suppose $N \gg r$. Let $H \subset \P^N_k$ be a hyperplane. Let $L \subset \P^N_k$ be a linear subspace of dimension $N-r+1$ intersecting $H$ transversely and let $X \subset L$ be a curve (not necessarily connected) none of whose components is contained in $H$. Then the set of linear subspaces in $\Gr^{\rm tr}(L, X, N-2, H)$ is a dense open subset of $\Gr(N-2, H)$.
\end{lem}

\begin{proof}Observe that $\Gr^{\rm tr}(L, N-2, H)$ is a dense open subset of $\Gr(N-2, H)$. Consider the map $\nu_L: \Gr^{\rm tr}(L, N-2, H) \to \Gr(N-r-1, L \cap H)$ given by $\nu_L(M) = L \cap M$. This $\nu_L$ is a smooth surjective morphism of relative dimension $2(r-1)$. It follows from Lemma \ref{lem:fs-open-1} that $\Gr(X, N-r-1, L\cap H)$ is a dense open subset of $\Gr(N-r-1, L \cap H)$, so $\nu^{-1}_L(\Gr(X, N-r-1, L\cap H))$ is a dense open subset of $\Gr^{\rm tr}(L, N-2, H)$, and hence a dense open subset of $\Gr(N-2, H)$.
\end{proof}

\subsection{Affine Veronese embedding and Linear projection}\label{subsubsection:AVE}
Recall that for positive integers $m, d \ge 1$, the Veronese embedding $v_{m,d}: \P^m_k \inj \P^N_k$ is a closed embedding given by $v_{m,d}([\underline{x}]) = [M_0(\underline{x}), \ldots , M_N(\underline{x})] = [M(\underline{x})]$, where $N =  {m+d \choose m}-1$ and $\{M_0, \ldots , M_N\}$ are all monomials in $\{x_0, \ldots , x_m\}$ of degree $d$, arranged in the lexicographic order.

If $[y_0, \ldots , y_N] \in \P^N_k$ denotes the projective coordinates, it is clear that $v_{m,d}^{-1}\left(\{y_0 = 0\}\right) = \{x^d_0 = 0\}$. In particular, the Veronese embedding yields Cartesian squares

\begin{equation}\label{eqn:Veronese}
\xymatrix@C1pc{
\A^m_k \ar[d]_{v_{m,d}} \ar[r] & \P^m_k \ar[d]^{v_{m,d}} &  dH_{m,0} \ar[l] \ar[d] ^{v_{m,d}}\\
\A^N_k \ar[r] & \P^N_k &  H_{N,0}, \ar[l]} 
\end{equation}
where $H_{m,0} \subset \P^m_k$ is the hyperplane $\{x_0 = 0\}$ and the vertical arrows are all closed embeddings. The closed embedding $v_{m,d}: \A^m_k \inj \A^N_k$ is  given by $v_{m,d}(y_1, \ldots , y_m) = (M_1', \ldots , M_N' )$, where $\{M_1', \ldots , M_N' \}$ is the induced ordered set of all monomials in $\{y_1, \ldots , y_m\}$ of degree bounded by $d$. 

Let $1 \le r < N$ be two integers. Recall (e.g., see \cite[Lemma 6.1]{KP}) that when $L \subset \P^N_k$ is a linear subspace of dimension $N-r-1$, there is an associated projection map $\phi_L:\P^N_k \setminus L\to \P^r_k$, where $\P^r_k$ is a linear subspace of $\P^N_k$ such that $L \cap \mathbb{P}_k ^r = \emptyset$. This map $\phi_{L}$ defines a vector bundle over $\P^r_k$ of rank $N-r$, whose fiber over a point $x \in \P^r_k$ is the affine space $C_x(L) \setminus L$, where $C_x(L) = {\rm Sec}(\{x\}, L)$. 

\begin{remk}\label{remk:vector-bunde}
The referee asked whether the above vector bundle $\phi_L:\P^N_k \setminus L\to \P^r_k$ is isomorphic to $\mathcal{O}_{\P^r_k}(1)^{\oplus (N-r)}$. Indeed, $\phi_L$ is (up to an isomorphism) the projection map of quotient stacks $\phi_L: [{((\A^{r+1} \setminus \{0\} ) \times_k V)}/{\G_m}] \to [ ( \A^{r+1} \setminus \{0\}) /{\G_m}]$, where $V = k^{N-r}$ and the $\G_m$-action everywhere is by scalar multiplication. Since $[{((\A^{r+1} \setminus \{0\}) \times_k V)}/{\G_m}] \cong [( \A^{r+1} \setminus \{0\}) /{\G_m}] \times_{B\G_m} [{V}/{\G_m}]$, one identifies $\phi_L$ with the map $\P^r_k \times_{B\G_m}  \pi^*({V(1)^{\oplus (N-r)}}) \to \P^r_k$, where $V(1)$ is the line bundle on $B\G_m := [{\Spec(k)}/{\G_m}]$ associated to the $1$-dimensional $\G_m$-representation given by the scalar multiplication on $k$, and $\pi: [(\A^{r+1} \setminus \{0\})/{\G_m}] \to B\G_m$ is the canonical projection. 

Note that in general, if we let $\G_m$ act on $k$ by weight $n \in \Z$ (i.e., $\lambda \cdot x = \lambda^n x$) and let $V(n)$ denote the corresponding line bundle on $B\G_m$, then $\pi^*(V(n))$ is isomorphic to $\mathcal{O}_{\P^r_k}(n)$. Hence the above $\pi^* (V (1) ^{\oplus (N-r)})$ is isomorphic to $\mathcal{O}_{\P ^r _k } (1)^{\oplus (N-r)}$, as wished.
\end{remk}

\begin{defn}\label{defn:LP-X}
Recall that if $X \subset \mathbb{P}^N_k$ is a closed subscheme with $X \cap L = \emptyset$, then $\phi_L$ restricted to $X$ defines a projection $\phi_L: {\phi_L}|_X: X \to \P^r_k$. We call it the linear projection of $X$ away from $L$. Since this is a morphism of projective schemes with affine fibers, it must be a finite morphism. In particular, $\dim(X) \le r$.
\end{defn}

We shall use the following situation often: let $H \subset \P^N_k$ be a hyperplane containing $L$ and $X \subset \mathbb{P}^N_k$ a closed subscheme with $X \cap L = \emptyset$ and $X \not\subset H$. Then $\phi_{L}$ defines the Cartesian squares of morphisms
\begin{equation}\label{eqn:proj-0}
\xymatrix@C1pc{
X\setminus H \ar[r] \ar[d] & X \ar[d] & X \cap H \ar[l] \ar[d] \\
\P^r_k \setminus H \ar[r] & \P^r_k & \P^r_k \cap H. \ar[l]}
\end{equation}
Together with \eqref{eqn:Veronese}, we deduce the following fact, which we use often:

\begin{lem}\label{lem:Ver-affine-map}
Let $X \inj \A^m_k$ be an affine scheme of dimension $r \ge 1$ and let $\overline{X} \inj \P^m_k$ be its projective closure. Then, for $d \ge 1$, the Veronese embedding $v_{m,d} : \P^m_k \inj \P^N_k$ and the linear projection away from $L \in \Gr(N-r-1, \P^N_k)(k)$ yield a Cartesian diagram with finite vertical maps
\begin{equation}\label{eqn:Ver-affine-map-0}
\xymatrix@C1pc{
X \ar[r] \ar[d]_{\phi_L} & \overline{X} \ar[d]^{\phi_L} \\
\A^r_k \ar[r] & \P^r_k}
\end{equation}
if $L \in \Gr(\overline{X}, N-r-1, H_{N,0}) (k)$, where $H_{N,0} = \{ y_0 = 0 \} \subset \mathbb{P}^N_k$ as in \eqref{eqn:Veronese}.
\end{lem}

\subsection{The Set-up}\label{sec:Set-up}
Let $k$ be an infinite perfect field. Here, we introduce the basic set-up that will be used for most of the paper. This set of assumptions will be referred to as the Set-up of \S \ref{sec:Set-up}. 

(1) \emph{The spaces:} Let $X$ be an equidimensional reduced projective $k$-scheme of dimension $r \geq 1$ with a given closed embedding $\eta: X \hookrightarrow \mathbb{P}_k ^N$ with $N \gg r$ and of degree $d+1 \gg 0$. We let $\widehat{B}$ be a smooth projective geometrically integral $k$-scheme of positive dimension and let $B \subset \widehat{B}$ be a nonempty affine open subset with $F: = \widehat{B} \setminus B$. Let $\Sigma \subset X_{\rm sm}$ be a finite set of closed points. 

(2) \emph{The linear projections:} 
Suppose that $H \subset \P^N_k$ is a hyperplane not meeting $\Sigma$, and that $X \setminus (X \cap H) \subset X_{\rm sm}$. For $L \in \Gr(X, N-r-1, H)(k)$, let $\phi_L: X \to \P^r_k$ be the linear projection away from $L$. If $L$ is fixed in a given context, we often drop it from the notation of $\phi_L$ and write as $\phi: X \to \P^r_k$. We write $\widehat{\phi} = \widehat{\phi}_L =\phi_L \times {\rm id}_{\widehat{B}}: X \times \widehat{B} \to \P^r_k \times \widehat{B}$.

(3) \emph{The cycles:} 
Let $Z \subset X \times \widehat{B}$ be a reduced closed subscheme with irreducible components $\{Z_1, \ldots , Z_s\}$, each of dimension $r$. We suppose that both $X \times F$ and $H \times \widehat{B}$ intersect properly with each irreducible component of $Z$. We let $\widehat{f}: Z \to X$ and $\widehat{g}: Z \to \widehat{B}$ denote the restrictions of the projection maps. Let $E \subset \widehat{B}$ be a closed subset containing $F$ such that no component of $Z$ is contained in $\widehat{g}^{-1}(E)$. We suppose that each projection $Z_i \to \widehat{B}$ is non-constant.

(4) \emph{The residual schemes and residual sets:}
Let $L^+(Z)$ be the closure of $\widehat{\phi}^{-1}(\widehat{\phi}(Z)) \setminus Z$ in $X \times \widehat{B}$ with the reduced closed subscheme structure. For any closed point $x \in \overline{X}$, we write $L^+(\{x\})$ as $L^+(x)$.
We let $L^+(\Sigma) =  \bigcup_{x \in \Sigma} L^+(x)$.

\subsection{A Nisnevich property of linear projections}\label{sec:Nis-LP}
The first result on `moving' our cycle $Z$ is the following:

\begin{lem}\label{lem:fs-proj-1et}
We are under the Set-up of \S \ref{sec:Set-up}. After replacing the embedding $\eta: X \inj \P^N_k$ by a bigger one via a Veronese embedding if necessary, there exists a dense open subset $\mathcal{U} \subset \Gr (X, N-r-1, H)$ such that each $L \in \mathcal{U}(k)$ satisfies the following:
\begin{enumerate}
\item $\phi_L$ is {\'e}tale at $\Sigma$.
\item $\phi_L(x) \neq \phi_L(x')$ for each pair of distinct points $x \neq x' \in \Sigma$.
\item $k(\phi_L(x)) \xrightarrow{\simeq} k(x)$ for all $x \in \Sigma$.
\item $L^{+}(x) \neq \emptyset$ for all $x \in \Sigma$. 
\item $L^{+}(x) \cap \widehat{f} (\widehat{g}^{-1}(E)) = \emptyset$ for
all $x \in \Sigma$.
\item
$L^{+}(x) \cap \widehat{f}(Z_i) = \emptyset$ for all $x \in 
\Sigma$ if  
$\widehat{f}: Z_i \to X$ is not dominant over any irreducible component of $X$. 
\end{enumerate}
%Furthermore, the above holds even after we replace $\phi_L$, $\Sigma$, $Z_i$, $E$, etc. by $\phi_{L, \bar{k}}$, $\Sigma_{\bar{k}}$, $Z_{i, \bar{k}}$, $E_{\bar{k}}$, etc. for an algebraic closure $\bar{k}$ of $k$.
\end{lem}

\begin{proof}
Replacing the given embedding $X \inj \P^N_k$ by its composition with a Veronese embedding, we may begin with a closed embedding $X \inj \P^N_k$ such that $N \gg r$ and the degree of $X$ in $\P^N_k$ is bigger than one.

\textbf{Step 1.} \emph{First suppose that $k$ is algebraically closed.} 
It follows from our assumption that $\dim(\widehat{g}^{-1}(E)) \le r-1$. Since $\widehat{f}$ is projective, it follows that $\widehat{f}(\widehat{g}^{-1}(E))$ is a closed subset of $X$ of dimension at most $r-1$. We let $W \subset X$ be the union of  $X_{\rm sing}$, $\widehat{f}(\widehat{g}^{-1}(E))$ and the images of all components of $Z$ which are not dominant over $X$. This is a closed subset of $X$ such that $\dim(W) \le r-1$. In particular, $\dim({\rm Sec} (D_1, W \cup D_2)) \le r$ for any finite closed subsets $D_1, D_2 \subset X$. It follows from Lemma \ref{lem:fs-open-1} that $\mathcal{U}_1 :=  \bigcap_{x \in \Sigma} \Gr(X \cup {\rm Sec}(\{x\}, W \cup (\Sigma \setminus \{x\})), N-r-1, H)$ is dense open in $\Gr(N-r-1, H)$. Furthermore, any $L \in \mathcal{U}_1(k)$ satisfies (5) and (6) by construction. 

We continue the proof of the rest of the properties. 
Let $T_{\Sigma, X} \subset \P^N_k$ be the union of the tangent spaces to $X$ at all points of $\Sigma$. Since $\Sigma \subset X_{\rm sm}$, we have $T_{\Sigma, X} = T_{\Sigma, X_{\rm sm}}$, which is a finite union of linear subspaces of dimension $r$. For each $x \in \Sigma$, the set $\sZ_x = X \cup T_{\Sigma, X} \cup {\rm Sec}(\{x\}, X_{\rm sing} \cup (\Sigma \setminus \{x\}))$ is closed in $\P^N_k$ of dimension $r$. Therefore, the set $\mathcal{U} =   \bigcap_{x \in \Sigma} \Gr(\sZ_x, N-r-1, H) \cap \mathcal{U}_1$ is dense open in $\Gr(N-r-1, H)$ by  Lemma \ref{lem:fs-open-1}. By construction, each $L \in \mathcal{U}(k)$ defines the finite surjective map $\phi_L: X \to \P^r_k$, which is unramified at $\Sigma$ and separates the points of $\Sigma$. In particular, (2) holds.

Since $X_{\rm sm}$ is regular and dense in $X$, it follows that $\phi_L|_{X_{\rm sm}} :X_{\rm sm} \to \P^r_k$ is a dominant and quasi-finite morphism between regular $k$-schemes. In particular, the map $\mathcal{O}_{\P^r_k, \phi_L(x)} \to \mathcal{O}_{X, x}$ is a local homomorphism of regular local rings with the finite closed fiber for each $x \in \Sigma$. It follows from \cite[Proposition (6.1.5), p.136]{EGA4-2} (or \cite[Theorem 23.1, p.179]{Matsumura}) that $\phi_L$ is flat at each point of $\Sigma$. Hence $\phi_L$ is {\'e}tale at $\Sigma$, being flat and unramified, proving (1).

Since $k = \bar{k}$, the isomorphisms of the residue fields, (3) is evident. The property (4) follows because $\deg (\phi_L)>1$ by the assumptions on the chosen Veronese embedding of $X$. This proves the lemma in Step 1 when $k$ is algebraically closed.

\textbf{Step 2.} \emph{Now suppose that $k$ is any infinite perfect field.} 
Let $\bar{k}$ be an algebraic closure. For any $k$-scheme $A$, let $\pi_A: A_{\bar{k}} \to A$ be the base change to $\bar{k}$. We have that $\Sigma_{\bar{k}} = \pi^{-1}_X(\Sigma)$ is still a finite closed set of the regular scheme $X_{{\rm sm}, \bar{k}}$. By Step 1 applied to $X_{\bar{k}}$, $H_{\bar{k}}$ and $\Sigma_{\bar{k}}$, there exists a dense open $\mathcal{U}' \subset \Gr(N-r-1, H_{\bar{k}})$ where the mentioned properties (1) $\sim$ (6) hold. 

Since $k$ is perfect, there exists a finite Galois extension $k \subset k'$ in $\bar{k}$ such that $\mathcal{U}' $ is defined over $k'$. Let ${\mathcal{U}}:= \bigcap_{\sigma \in {\rm Gal} (k'/k)} \sigma \cdot \mathcal{U}'$. This is a nonempty open subset defined over the radicial closure of $k$ in $k'$, but since $k$ is perfect, this radicial closure is equal to $k$. Hence $\mathcal{U} \subset \Gr (N-r-1, H)$ and it is defined over $k$.  (\emph{cf.} \cite[Lemma 3.4.3]{CTHK}). Here we have $\mathcal{U}_{\bar{k}} \subset \mathcal{U}'$. Now, for each $L \in \mathcal{U} (k)$, we have $X \cap L = \emptyset$ by our choice of the open set. So, we get a finite surjective map $\phi_L: X \to \mathbb{P}_k ^r$ over $k$.

We prove that $\phi_L$ is \'etale at each point $x \in \Sigma$. Let $y:= \phi_L (x)$. By the faithfully flat descent (\cite[Corollaire (17.7.3)-(ii), p.72]{EGA4-4}), the map $\phi_L : \Spec (\mathcal{O}_{X,x}) \to \Spec (\mathcal{O}_{\mathbb{P}_k ^r, y})$ is \'etale if and only if its faithfully flat base change $\phi_{L, \bar{k}} : \Spec (\mathcal{O}_{X_{\bar{k}, x_{\bar{k}}}}) \to \Spec (\mathcal{O}_{\mathbb{P}^r_{\bar{k}, y_{\bar{k}}}})$ of the semi-local schemes via $\Spec (\bar{k}) \to \Spec (k)$ is \'etale. Here, $x_{\bar{k}}:= \pi_X ^{-1} (x)$ and $y_{\bar{k}}:= \pi_{\mathbb{P}_k ^r} ^{-1} (y)$. But Step 1 shows that the latter map $\phi_{L, \bar{k}}$ is \'etale at each point of the set $x_{\bar{k}} \subset \Sigma_{\bar{k}}$, thus so is the former $\phi_L$ at $x$. This proves (1).

Since $\phi_{L, \bar{k}}$ separates the points of $\Sigma_{\bar{k}}$ by construction, (2) is obvious. Furthermore, this shows that for each $x \in \Sigma$, the map $\phi_{L, \bar{k}} : \pi_X ^{-1} (x) \to \pi_{\mathbb{P}^r_k} ^{-1} (y)$ is injective, where $y=\phi_L (x)$. Hence by Lemma \ref{lem:base-change-kbar} below, we have $k(x) = k(y)$, which proves (3). The property (4) is evident because $\deg (\phi_L) > 1$ and $k(\phi (x)) \simeq k(x)$ for each $x \in \Sigma$ by (3). 

The conditions (5) and (6) are apparent for any $L \in \mathcal{U}(k)$ because 
for every $x \in \Sigma$, we have that $(L_{\bar{k}})^+(x') \cap
\widehat{f}_{\bar{k}}(\widehat{g}^{-1}_{\bar{k}}(E_{\bar{k}})) = 
\emptyset = (L_{\bar{k}})^+(x') \cap \widehat{f}_{\bar{k}}(Z_{i, \bar{k}})$
for all $x'$ lying in the finite set $\pi^{-1}_X(x) \subset \Sigma_{\bar{k}}$.
Note here that if $Z_i$ is not dominant over a component of $X$, then
no component of $Z_{i, \bar{k}}$ can be dominant over any component of
$X_{\bar{k}}$. This finishes the proof of the lemma.
\end{proof}

We used the following in the middle of the proof of the above Lemma \ref{lem:fs-proj-1et}.

\begin{lem}\label{lem:base-change-kbar}
Let $k$ be an infinite perfect field and let $\phi: X \to Y$ be a finite morphism of $k$-schemes. Consider the base change Cartesian square
\begin{equation}\label{eqn:base-change}
\xymatrix@C1pc{
X_{\bar{k}} \ar[d]_{\pi_X} \ar[r]^{\phi_{\bar{k}}} & Y_{\bar{k}} \ar[d]^{\pi_{Y}} \\
X \ar[r]^{\phi} & Y.}
\end{equation}
Let $x \in X$ be a closed point and let $y:= \phi (x)$. Then one has $| \pi_Y ^{-1} (y)| \leq | \pi_X ^{-1} (x)|$. The equality holds if and only if $[k(x): k(y)] = 1$. Furthermore, this equality holds when the map $\phi_{\bar{k}}: \pi_X ^{-1} (x) \to \pi_Y ^{-1} (y)$ is injective.
\end{lem}

\begin{proof}
Since $k$ is perfect, we have $|\pi_X^{-1}(x)| = [k(x):k]$ and $|\pi_{Y}^{-1}(y)| = [k(y) : k]$. So, the field extensions $k \inj k(y) \inj k(x)$ imply the first and the second assertions. If the map $\phi_{L_{\bar{k}}}: \pi_X^{-1}(x) \to \pi_{Y}^{-1}(y)$ is injective, then $|\pi_{Y}^{-1}(y)| \ge |\pi_X^{-1}(x)|$. The last part of the lemma thus follows.
\end{proof}

\subsection{Some algebraic results}\label{subsection:Com-Alg}
We discuss some algebraic results that will be needed.

\begin{lem}\label{lem:elem-com-alg}
Let $f:A \to B$ be an injective finite unramified local homomorphism of noetherian local rings, that induces an isomorphism of the residue fields. Then $f$ is an isomorphism. 
\end{lem}

\begin{proof}
Let $\fm_A$ and $\fm_B$ be the maximal ideals of $A$ and $B$, respectively. Since $f$ is finite, to show that $f$ is surjective 
it suffices to show that $A/{\fm_A} \to B/({\fm_A}B)$ is surjective by Nakayama's lemma. But this follows because the map $A/{\fm_A} \to B/{\fm_B}$ is an isomorphism and so is the map $B/({\fm_A}B) \to B/{\fm_B}$ as $f$ is unramified.
\end{proof}

\begin{lem}\label{lem:GP-0}
Let $f: Y' \to Y$ be a finite surjective morphism of regular $k$-schemes. Let $W \subset Y$ be an irreducible closed subset and let $y \in W$ be a closed point. Let $S = f^{-1}(y)$ and $W' = f^{-1} (W) $. Let $x \in S$ and let $Z \subset W'$ be an irreducible component passing through $x$. Suppose that $f$ is {\'e}tale at $x$ and $k(y) \xrightarrow{\simeq} k(x)$. Then $Z \cap S = \{x\}$ if and only if $Z$ is the only component of $W'$ passing through $x$.
\end{lem}

\begin{proof}
We first observe that $f$ must be a flat morphism (see \cite[Exercise III-10.9, p.276]{Hartshorne}). We next note that any irreducible component of $W'$ that passes through $x$ will be in the connected component of $Y'$ containing $x$. So, we may assume $Y'$ is connected. On the other hand, $W\subset Y$ being irreducible, it must belong to a unique connected component of $Y$. Hence, we may also assume that $Y$ is connected.

Now, first suppose $S= \{ x \}$. We claim that $f$ is an isomorphism locally around $y$, so that the lemma holds trivially. Indeed, it follows from Lemma \ref{lem:elem-com-alg} that the map $\mathcal{O}_{Y, y} \to \mathcal{O}_{Y', x}$ is an isomorphism. This implies that $f$ is a finite and flat map with $[k(Y'): k(Y)] = 1$ (see \cite[Exercise 5.1.25, p.176]{Liu}) and hence must be an isomorphism.

We now suppose $|S| > 1$. Consider the commutative diagram of semi-local rings
\begin{equation}\label{eqn:GP-0-0}
\xymatrix@C2pc{
\mathcal{O}_{Y, y} \ar[r]^{\alpha_1} \ar@{->>}[d]^{\beta_1} & 
\mathcal{O}_{Y',S} \ar@{->>}[d]^{\beta_2} \ar[r]^{\alpha_2} & 
\mathcal{O}_{Y',x} \ar@{->>}[d]^{\beta_3} \\
\mathcal{O}_{W,y}  \ar@/_3.5pc/[drr] ^{\gamma'} \ar[r]^{\alpha_3} \ar[dr]_{\gamma} & 
\mathcal{O}_{W',S} \ar[r]^{\alpha_4} \ar@{->>}[d]^{\beta_4} & \mathcal{O}_{W',x} 
\ar@{->>}[d]^{\beta_5} \\
& \mathcal{O}_{Z,S} \ar[r]^{\alpha_5} & \mathcal{O}_{Z,x},}
\end{equation}
where $\gamma:= \beta_4 \circ \alpha_3$ and $\gamma':= \alpha_5 \circ \gamma$. Here, $\alpha_1$ and $\alpha_3$ are finite and flat, and $\alpha_2\circ \alpha_1$ is \'etale. The lemma is equivalent to that $\alpha_5$ is an isomorphism if and only if $\beta_5$ is.

Suppose $\alpha_5$ is an isomorphism. Since $\beta_4$ is surjective and $\alpha_3$ is finite, the map $\gamma$ is finite. Thus, $\gamma'$ is a finite map of local rings. Since $\alpha_2 \circ \alpha_1$ is {\'e}tale, the map $\alpha_4 \circ \alpha_3$ is also {\'e}tale. Since $\beta_5$ is surjective, we see that $\gamma'$ is unramified. Thus, $\gamma'$ is a finite and unramified map of local rings. Since $Z \to W$ is surjective and $k(y) \simeq k(x)$, the map $\gamma'$ is an isomorphism by Lemma \ref{lem:elem-com-alg}. In particular, $\alpha_4 \circ \alpha_3$ is an {\'e}tale map of local rings such that $\beta_5 \circ \alpha_4 \circ \alpha_3$ is an isomorphism, in particular, {\'e}tale. It follows that $\beta_5$ is {\'e}tale, by \cite[Proposition (17.3.4), p.62]{EGA4-4}. Thus, $\beta_5$ is a surjective \'etale map of local rings.  But it can happen only if $\beta_5$ is an isomorphism.

Conversely, suppose that $\beta_5$ is an isomorphism. Let $\mathfrak{p}$ be the minimal prime of $\mathcal{O}_{W', S}$ such that ${\mathcal{O}_{W',S}}/{\mathfrak{p}} = \mathcal{O}_{Z,S}$ and let $\{\mathfrak{p}_1, \ldots , \mathfrak{p}_m\}$ denote the set of distinct minimal primes of $\mathcal{O}_{W',S}$ different from $\mathfrak{p}$. To show that $\alpha_5$ is an isomorphism, we need to show that $\mathfrak{p} + \mathfrak{p}_i = \mathcal{O}_{W',S}$ for all $1\le i \le m$.

{\bf Claim 1:} \emph{$\mathfrak{p}_i \mathcal{O}_{W', x} = \mathcal{O}_{W',x}$ for all $1 \le i \le m$.}

$(\because)$ Note that $\mathcal{O}_{W', x}$ is an integral domain because $\mathcal{O}_{Z,x}$ is an integral domain and $\beta_5$ is an isomorphism. Thus, we must have either $\mathfrak{p}_i \mathcal{O}_{W', x} = 0$, or $\mathfrak{p}_i \mathcal{O}_{W', x} = \mathcal{O}_{W', x}$. In the first case, we have $\mathfrak{p}_i \mathcal{O}_{Z, x} = 0$ as $\beta_5$ is an isomorphism. Equivalently, $\alpha_5 \circ \beta_4(\mathfrak{p}_i) = 0$. Since $\mathfrak{p}_i \not = \mathfrak{p}$, and $\mathfrak{p}_i, \mathfrak{p}$ are minimal, there is $a_i \in \mathfrak{p}_i \setminus \mathfrak{p}$ such that $\beta_4(a_i) \neq 0$. Hence, $\alpha_5 \circ \beta_4(a_i) \neq 0$, because $\alpha_5$ is injective being a localization of an integral domain. This is a contradiction. Thus, we must have $\mathfrak{p}_i \mathcal{O}_{W', x} = \mathcal{O}_{W', x}$ for each $i$, proving Claim 1.

Let $\fm$ be the maximal ideal of $\mathcal{O}_{W',S}$ defining the closed point $x$. By Claim 1, for any $1 \le i \le m$, there exists $a_i \in \mathfrak{p}_i \setminus \fm$ in $\mathcal{O}_{W', S}$ such that $\alpha_4(a_i)$ is invertible. Let $a =  \prod_{i=1} ^m a_i$. We see that there are nonzero elements $b, c \in \mathcal{O}_{Z',S}$ with $c \notin \fm$ such that $c(1-ab) = 0$.

{\bf Claim 2:} \emph{$1-ab \in \mathfrak{p}$.}

$(\because)$ Let $v = 1-ab$. Then, we have $cv = 0 \in \mathfrak{m}$ with $c \not \in \mathfrak{m}$, so that $v \in \mathfrak{m}$ and $\alpha_4 (v) = 0$. Toward contradiction, suppose $v \notin \mathfrak{p}$. Then $v \in \fm \setminus \mathfrak{p}$, so that $\beta_4(v) \neq 0$. Thus $\beta_5 \circ \alpha_4(v) = \alpha_5 \circ \beta_4(v) \neq 0$ because $\alpha_5$ is injective. But this contradicts that $\alpha_4(v) = 0$. Hence, we have $v \in \mathfrak{p}$, proving Claim 2.

By Claim 2, we have $v \in \mathfrak{p}$, $ab \in \mathfrak{p}_i$ for all $i$, while $v-ab = 1$. This shows that $\mathfrak{p} + \mathfrak{p}_i = \mathcal{O}_{W',S}$ for all $1 \le i \le m$. Thus, $\alpha_5$ is an isomorphism.
\end{proof}

\subsection{Birationality under linear projections}\label{sec:Birational}
Using Lemma \ref{lem:fs-proj-1et}, we shall show that the linear projections often give birational morphisms when restricted to a given integral closed subscheme. But first, we derive the following consequence of the results we proved in \S~\ref{sec:Nis-LP} and \S~\ref{subsection:Com-Alg}. We continue to work with the set-up of \S~\ref{sec:Set-up}.

We use a trick of ``marking" irreducible components: for each $1 \le i \le s$, we fix a closed point $\alpha_i \in (Z_i)_{\rm sm}$ such that (1) $\alpha_i \not \in Z_j$ for $j \not = i$, (2) $x_i = \widehat{f} (\alpha_i) \in X_{\rm sm} $ but not in $\Sigma$, and (3) $b_i = \widehat{g} (\alpha_i) \in B$. Note here that $\alpha_i \in (Z_i)_{\rm sm}$ and $x_i \in X_{\rm sm}$ can be achieved as follows: by the assumptions of the Set-up of \S \ref{sec:Set-up}, each $Z_i$ intersects $H \times \widehat{B}$ properly and $X \setminus (X \cap H) \subset X_{\rm sm}$. Then any choice of a point in $Z_i |_{(X \setminus (X \cap H)) \times B}$ maps to a point of $X_{\rm sm}$. Moreover, perfectness of $k$ implies that $(Z_i)_{\rm sm} \cap Z_i |_{(X \setminus (X \cap H)) \times B} \neq \emptyset$. Let $\Xi = \{x_1, \ldots , x_s\} \cup \Sigma$ and $E = \{b_1, \ldots , b_s\} \cup F$. Since $Z_i \not \subset X \times F$ and $Z_i \to \widehat{B}$ is non-constant by the Set-up of \S \ref{sec:Set-up}, no component of $Z$ lies in $\widehat{g} ^{-1} (E)$.

\begin{lem}\label{lem:fs_sansZ-*}
After replacing the embedding $X \inj \P^N_k$ by a bigger one via a Veronese embedding if necessary, there is a dense open subset $\mathcal{U} \subset \Gr (X, N-r-1, H)$ such that each $L \in \mathcal{U}(k)$ has the property that \emph{$Z_i \cap \widehat{\phi}^{-1}_L (\widehat{\phi}_L(\alpha_i)) = \{ \alpha_i \}$} for all $1 \le i \le s$.
\end{lem}

\begin{proof}
We let $\pi: \Spec(\bar{k}) \to \Spec(k)$ denote the base change map. For any $A \in \Sch_k$, we shall write $\pi_A: A_{\bar{k}} \to A$ simply as $\pi$ using a shorthand.

We fix $i$. Let $\beta_i:= \widehat{\phi}_L (\alpha_i)$. Let $\pi^{-1} (\alpha_i) = \{ \alpha_{ij} \}_j$, which is a finite set of points, and let $x_{ij}:= \widehat{f}_{\bar{k}} (\alpha_{ij})$, $b_{ij} := \widehat{g}_{\bar{k}} (\alpha_{ij})$. Note that all of $\alpha_{ij}$ and $x_{ij}$ lie in the smooth loci of $(Z_{i})_{\bar{k}}$ and $X_{\bar{k}}$, respectively. 

We let $\Xi_i:= \{ x_{ij} \}_j \cup \Sigma_{\bar{k}}$ and $E_i:= \{ b_{ij} \}_{j} \cup F_{\bar{k}}$. 

Applying Lemma \ref{lem:fs-proj-1et} over $\bar{k}$ for the above $\Xi_i$ (for $\Sigma$ there) and $E_i$ (for $E$ there), we obtain a dense open subset $\mathcal{U}'_i \subset \Gr (X_{\bar{k}}, N-r-1, H_{\bar{k}})$ such that every $L \in \mathcal{U}' (k)$ satisfies the properties (1)$\sim$ (6) there. Repeating the argument of Lemma \ref{lem:fs-proj-1et} in Step 2, we obtain a dense open subset $\mathcal{U}_i \subset \Gr (X, N-r-1, H)$ such that for every $L \in \mathcal{U}_i (k)$, we have $L_{\bar{k}} \in \mathcal{U}'_i (\bar{k})$. 

We show that the following map is bijective:
\begin{equation}\label{eqn:new4.8}
\widehat{\phi}_{L, \bar{k}} : \pi^{-1} (\alpha_i) \to \pi^{-1} (\beta_i).
\end{equation}

Suppose this is not injective, i.e. for some $j < j'$, we have $\widehat{\phi}_{L, \bar{k}} (\alpha_{ij}) = \widehat{\phi}_{L, \bar{k}} (\alpha_{ij'})$. Then $b_{ij} = \widehat{g}_{\bar{k}} (\alpha_{ij}) = \widehat{g}_{\bar{k}} (\alpha_{ij'})= b_{ij'}$. Since $\bar{k}$ is algebraically closed, we can write $\alpha_{ij} = (x_{ij}, b_{ij})$ and $\alpha_{ij'} = (x_{ij'} , b_{ij'})$. Since $b_{ij} = b_{ij'}$ and $\alpha_{ij} \not = \alpha_{ij'}$, we must have $x_{ij} \not = x_{ij'}$. 

But at the same time, we have
\[
\widehat{\phi}_{L, \bar{k}}(x_{ij}) =\widehat{f}_{\bar{k}}(\widehat{\phi}_{L, \bar{k}}(\alpha_{ij})) =\widehat{f}_{\bar{k}}(\widehat{\phi}_{L, \bar{k}}(\alpha_{ij'})) = \widehat{\phi}_{L, \bar{k}}(x_{ij'}).
\] 
In particular, $x_{ij'} \in L_{\bar{k}} ^+ (x_{ij})$. Since $\widehat{g}_{\bar{k}} (\alpha_{ij'}) = b_{ij'} \in E_i$, we thus have $\widehat{f}_{\bar{k}}(\alpha_{ij'}) = x_{ij'} \in (L_{\bar{k}})^+(x_{ij}) \cap \widehat{f}_{\bar{k}}(\widehat{g}^{-1}_{\bar{k}}(E_i))$. But this contradicts the property (5) of Lemma \ref{lem:fs-proj-1et} satisfied by $L_{\bar{k}}$. Hence the map \eqref{eqn:new4.8} is injective. 

On the other hand, we have
$$\pi^{-1}(\beta_i) \times_{\Spec (k (\beta_i))} \Spec(k(\alpha_i)) =\Spec((\bar{k} \otimes_k k(\beta_i)) \otimes_{k(\beta_i)} k(\alpha_i)) =\Spec(\bar{k} \otimes_k k(\alpha_i)) = \pi^{-1}(\alpha_i)
$$
so that it follows that the map \eqref{eqn:new4.8} is surjective, as well, thus bijective.

Going back to the proof of the lemma, first note that we clearly have $Z_i \cap \widehat{\phi}^{-1}_L (\beta_i) \supset \{ \alpha_i \}$. For the inclusion in the other direction, toward contradiction suppose there is $\alpha' \in \widehat{\phi}^{-1}_L (\beta_i) \setminus \{ \alpha_i \}$ such that $\alpha' \in Z_i$. Clearly we have $\pi^{-1} (\alpha') \cap \pi^{-1} (\alpha_i) = \emptyset$. On the other hand, we have $\widehat{\phi}_{L, \bar{k}} ( \pi ^{-1} (\alpha')) \subset \pi^{-1} (\beta_i) = \widehat{\phi} _{L, \bar{k}} (\pi^{-1} (\alpha_i))$, where the second equality holds by the bijectivity of \eqref{eqn:new4.8}. 

Hence there is some $\alpha'_{j} \in \pi^{-1} (\alpha')$ and $\alpha_{ij'} \in \pi^{-1} (\alpha_i)$ such that 
$$ (a) \ \ \alpha'_{j} \not = \alpha_{ij'}, \mbox{ while } (b) \ \ \widehat{\phi}_{L, \bar{k}} (\alpha'_j) = \widehat{\phi}_{L, \bar{k}} (\alpha_{ij'}).$$

The property (b) implies that $\widehat{g}_{\bar{k}} (\alpha'_j) = \widehat{g}_{\bar{k}} (\alpha_{ij'}) = b_{ij'}$. Since $\bar{k}$ is algebraically closed, for $x':= \widehat{f}_{\bar{k}} (\alpha'_j)$, we can express $\alpha'_j = (x', b_{ij'})$ and $\alpha_{ij'} = (x_{ij'}, b_{ij'})$. Because $\alpha'_j \not \alpha_{ij'}$ by (a), we must have $x' \not = x_{ij'}= \widehat{f}_{\bar{k}} ( \alpha_{ij'})$. In particular, $x' \in L_{\bar{k}}^+ (x_{ij'})$. But $\widehat{g}_{\bar{k}} = b_{ij'} \in E_i$ so that we obtain $x' \in L_{\bar{k}} ^+ (x_{ij'}) \cap \widehat{f}_{\bar{k}} (\widehat{g}_{\bar{k}} ^{-1} (E_i))$. But, this contradicts the property (5) of Lemma \ref{lem:fs-proj-1et} satisfied by $L_{\bar{k}}$. Hence no such $\alpha'$ exists.
Our proof then is over by taking $\mathcal{U}:= \bigcap_{i=1} ^s \mathcal{U}_i$. 
\end{proof}

Combined with Lemma \ref{lem:GP-0}, we immediately have:

\begin{cor}\label{cor:fs_sansZ}
For each linear projection $L$ as in Lemma \ref{lem:fs_sansZ-*}
and each $1 \le i \le s$, one has that $Z_i$ is the only irreducible component of $\widehat{\phi}_L ^{-1} (\widehat{\phi}_L (Z_i))$ passing through a given
marked point $\alpha_i \in Z_i \setminus \bigcup_{j \not = i} Z_j$. 
\end{cor}

We can now prove the birationality of a given finite set of integral closed subschemes of $X \times \widehat{B}_n$ under suitable linear projections.

\begin{lem}\label{lem:fs_sansZ}
For a suitable choice of the set $E$ in the Set-up of \S \ref{sec:Set-up}, after replacing the embedding $X \inj \P^N_k$ by a bigger one via a Veronese embedding if necessary, there is a dense open subset $\mathcal{U} \subset \Gr (X, N-r-1, H)$ such that for each $L \in \mathcal{U}(k)$, the induced map $\widehat{\phi}_L: Z_i \to \widehat{\phi}_L(Z_i)$ is birational for all $1 \le i \le s$.
\end{lem}

\begin{proof}
We follow the choices of $\alpha_i \in Z_i, \ \Xi$ and $E$ that we made just before Lemma~\ref{lem:fs_sansZ-*}. We shall prove the lemma for this $E$. We let $\mathcal{U}  \subset \Gr (X, N-r-1, H)$ be as given by Lemma~\ref{lem:fs_sansZ-*} and fix $L \in \mathcal{U}(k)$. 
We let $T_i := \widehat{\phi}_L(Z_i)$ and $\beta_i := \widehat{\phi}_L(\alpha_i)$. To show that $\widehat{\phi}_L: Z_i \to T_i$ is birational, we prove a stronger assertion that the map $\mathcal{O}_{T_i, \beta_i} \to \mathcal{O}_{Z_i, \beta_i}$ of semi-local rings is an isomorphism, where $\mathcal{O}_{Z_i, \beta_i} := \mathcal{O}_{Z_i,  Z_i \cap \widehat{\phi}^{-1}_L(\beta_i)}$. Consider the maps
\begin{equation}\label{eqn:nicomp}
\mathcal{O}_{T_i, \beta_i} \to \mathcal{O}_{Z_i, \beta_i} \to \mathcal{O}_{Z_i, \alpha_i}.
\end{equation}

It follows from Lemma~\ref{lem:fs_sansZ-*} that $Z_i \cap \widehat{\phi}^{-1}_L (\beta_i) = \{ \alpha_i \}$. In particular, the second map of \eqref{eqn:nicomp} is an isomorphism, actually the identity map.
By the condition (1) of Lemma \ref{lem:fs-proj-1et}, the map $\phi_L$ is \'etale in an affine open neighborhood $U'$ of $\Xi$, and thus $\phi_L$ is \'etale at $\alpha_i$. In particular, the composite map in \eqref{eqn:nicomp} is unramified. By the condition (3) of Lemma \ref{lem:fs-proj-1et}, we have $k (\beta_i) \overset{\simeq}{\to} k (\alpha_i)$. Hence, the first map of \eqref{eqn:nicomp} is an injective finite unramified map of local rings, that induces an isomorphism of the residue fields. It is therefore an isomorphism by Lemma \ref{lem:elem-com-alg}. This completes the proof.
\end{proof}

\begin{comment}
In the Claim of the proof of Lemma \ref{lem:fs_sansZ}, we proved that $Z_i \cap \widehat{\phi}_L ^{-1} (\beta_i) = \{ \alpha_i \}$. Combined with Lemma \ref{lem:GP-0}, we immediately have:

\begin{cor}\label{cor:fs_sansZ}
For a linear projection $L$ as in Lemma \ref{lem:fs_sansZ}, 
each $1 \le i \le s$ has the property that
$Z_i$ is the only irreducible component of $\widehat{\phi}_L ^{-1} (\widehat{\phi}_L (Z_i))$ passing through a given marking point $\alpha_i \in Z_i \setminus \bigcup_{j \not = i} Z_j$. 
\end{cor}
\end{comment}

\subsection{A presentation lemma for moving to fs-cycles}\label{sec:fs-main}
The final result of  \S \ref{sec:LP} is the following Theorem \ref{thm:fs-present-intro}, that will be used in the proof of the \fs-moving lemma, specifically, in the proof of Lemma \ref{lem:fs-move-*}.

\begin{thm}\label{thm:fs-present-intro}
Under the Set-up of \S \ref{sec:Set-up}, let $Z_i ^0:= Z_i |_{X \times B}$ and $Z^0:= Z |_{X \times B}$. 

Then for a suitable choice of the set $E$ in the Set-up, after replacing the embedding ${X} \inj \P^N_k$ by its composition with a suitable Veronese embedding, there is a dense open subset $\mathcal{U} \subset \Gr (X, N-r-1, H)$ such that each $L \in \mathcal{U}(k)$ satisfies the following:
\begin{enumerate}
\item $\phi_L$ is {\'e}tale at $\Sigma$.
\item $\phi_L$ separates the points of $\Sigma$.
\item $k(\phi_L(x)) \xrightarrow{\simeq} k(x)$ for all $x \in \Sigma$.
\item There exists an affine open neighborhood $U \subset X$ of $\Sigma$ such that
\begin{enumerate}
\item [(4a)] if ${Z}_i^0$ is an irreducible component of ${Z}^0$ that is dominant over an irreducible component of $X$, then for each component $Z'$ of $L^+ ({Z}_i^0)$, the map $Z'_U \to U$ is fs over $U$.
\item [(4b)] if ${Z}_i^0$ is an irreducible component of ${Z}^0$ that is not dominant over any irreducible component of $X$, then $L^+ ({Z}_i^0)_U = 0$.
\end{enumerate}
\end{enumerate}
\end{thm}

\begin{proof}
For $1 \leq i \leq s$, choose closed points $\alpha_i \in {Z}_i^0 \setminus ( \bigcup_{j \not = i} Z_j ^0)$ such that $x_i := \widehat{f} (\alpha_i) \in X_{\rm sm}$ and $b_i:= \widehat{g} (\alpha_i) \in {B}$, as we did in Lemma \ref{lem:fs_sansZ-*}. Let $\Xi:= \{ x_1, \cdots, x_r \} \cup \Sigma \subset X_{\rm sm}$ and $E := \{ b_1, \cdots, b_r \} \cup F \subset \widehat{B}$. Since  $Z_i \not \subset X \times F$ and $Z_i \to \widehat{B}$ is non-constant, it is not contained in $\widehat{g}^{-1} (E)$. We choose $\mathcal{U} \subset \Gr (X, N-r-1, H)$ as given by
Lemma~\ref{lem:fs_sansZ-*} and fix $L \in \mathcal{U}(k)$. In particular, all the properties of Lemma \ref{lem:fs-proj-1et} holds, so that we have the conditions (1) $\sim$ (3) of the theorem. 

To prove (4), first note that the irreducible components of $L^{+}({Z}_i^0)$ are exactly the restrictions to $X \times B$ of the irreducible components of $L^{+}({Z}_i)$. Let $Z_i$ be an irreducible component of $Z$ dominant over an irreducible component of $X$. Let $Z'$ be an irreducible component of $L^+ ({Z}_i)$. We prove that $Z' \cap (\Sigma \times F) = \emptyset$.

Suppose, on the contrary, that there is a closed point $\lambda \in Z' \cap (\Sigma \times F)$. This means that there is a closed point $\lambda' \in Z_i$ such that $\widehat{\phi}_L(\lambda) = \widehat{\phi}_L(\lambda')$. We claim in this case that 
\begin{equation}\label{eqn:fs-present-intro-0}
\{\lambda'\} = \widehat{\phi}^{-1}_L(\widehat{\phi}_L(\lambda)) \cap Z_i = \{\lambda\}. 
\end{equation}

Suppose we have shown that $\lambda' = \lambda$. Then we get $\lambda \in Z_i$ and ~\eqref{eqn:fs-present-intro-0} becomes equivalent to showing that $\widehat{\phi}^{-1}_L(\widehat{\phi}_L(\lambda)) \cap Z_i = \{\lambda\}$. But the proof of this equality is simply a repetition of the argument of Lemma~\ref{lem:fs_sansZ-*}. Hence, the claim is reduced to showing that $\lambda' = \lambda$. 

Let's do it. First consider the case when $k$ is algebraically closed. We can then uniquely write $\lambda = (x, b)$ for some closed points $x \in \Sigma$ and $b \in F$, and $\lambda' = (x', b)$, where $x' \in \widehat{\phi}^{-1}_L(\widehat{\phi}_L(x))$. If $x' \not = x$, 
then $x' \in L^+ (x)$ and $x' \in \widehat{f} (\widehat{g}^{-1}(E))$, which contradicts the condition (5) of Lemma \ref{lem:fs-proj-1et}. Hence, we must have $x' = x$ so that $\lambda' = \lambda$.

If $k$ is not algebraically closed, we argue as in the proof of Lemma~\ref{lem:fs_sansZ-*}.
Suppose again that $\lambda' \neq \lambda$. Then for the base change map $\pi: \Spec (\bar{k}) \to \Spec (k)$, we have $\pi^{-1}(\lambda') \cap \pi^{-1}(\lambda) = \emptyset$. Let $\beta' := \widehat{\phi}_L(\lambda') = \widehat{\phi}_L(\lambda)$. We show as in the argument of Lemma~\ref{lem:fs_sansZ-*} that the map $\widehat{\phi}_{L, \bar{k}}: \pi^{-1}(\lambda') \to \pi^{-1}(\beta')$ is bijective. Using this, we continue following the proof of Lemma~\ref{lem:fs_sansZ-*}, to get closed points $\widetilde{\lambda} \in \pi^{-1}(\lambda)$ and $\widetilde{\lambda'} \in \pi^{-1}(\lambda')$ such that $\widehat{f}_{\bar{k}}(\widetilde{\lambda}) \in (L_{\bar{k}})^+(\widehat{f}_{\bar{k}}(\widetilde{\lambda'})) \cap \widehat{f}_{\bar{k}}(\widehat{g}^{-1}_{\bar{k}}(E_i))$. But this contradicts the property (5) of Lemma \ref{lem:fs-proj-1et} for $L_{\bar{k}}$, which violates our choice of $L$. This proves ~\eqref{eqn:fs-present-intro-0}.

Coming back to the proof of $Z' \cap (\Sigma \times F) = \emptyset$, we now note using Corollary \ref{cor:fs_sansZ} that $Z' \not = Z_i$. 
So, the two deductions $\lambda \in Z' \cap {Z}_i$ and $ \widehat{\phi}^{-1}_L(\widehat{\phi}_L(\lambda)) \cap Z_i = \{\lambda\}$ from \eqref{eqn:fs-present-intro-0} together contradict Lemma \ref{lem:GP-0}. Hence, we must have $Z' \cap (\Sigma \times F)= \emptyset$, as desired. 

Now, by Lemma \ref{lem:fs-smooth}, there is an affine open neighborhood $U_{i, Z'} \subset X_{\rm sm}$ of $\Sigma$ such that $Z' _{U_{i, Z'}} \to U_{i, Z'}$ is fs. We take $U_1:= \bigcap U_{i, Z'}$ where the intersection is taken over all $i$ such that $Z_i$ dominant over a component of $X$ and the irreducible components $Z'$. This open set $U_1$ works for (4a).

About the property (4b), let $Z_i$ be an irreducible component of $Z$ which is not dominant over $X$. Let $Z'$ be a component of $L^{+}({Z}_i)$. In this case, we repeat the proof of (4a) above, where we now apply condition (6) of Lemma \ref{lem:fs-proj-1et}, to conclude that $Z' \cap (\Sigma \times {\widehat{B}}) = \emptyset$.

It follows that $\widehat{f}(L^{+}({Z}_i))$ is a closed subset of $X$ disjoint from $\Sigma$. Hence, we can apply Lemma \ref{lem:FA} to obtain an affine open neighborhood $U_i '$ of $\Sigma$ in $X$ such that $L^{+}({Z}_i)_{U_i'} = \emptyset$. We take $U_2:= \bigcap U_{i}'$, where the intersection is taken over all $i$ such that $Z_i$ is not dominant over any component of $X$. This open set $U_2$ works for (4b). Taking $U:= U_1 \cap U_2$, we have (4), and this
concludes the proof of the theorem.
\end{proof}

\section{Regularity of the original cycle over residual points}\label{section:LP2}
The focus of the remaining sections is to achieve the \sfs-property of the residual cycle of $Z$ along $\Sigma$ via more refined linear projections. In order to achieve this, we first ensure that our original cycle $Z$ is regular at all points lying over the residual set $L^+(\Sigma)$ of $\Sigma \subset X$. We later show that this regularity of $Z$ at all points lying over $L^+ (\Sigma)$ implies the regularity of the residual cycle of $Z$ along $\Sigma$. The goal of this section is to achieve the first one when $k$ is algebraically closed. The general case will be considered later.

\subsection{A basic algebraic result} 
We first discuss the following.

\begin{lem}\label{lem:AD-open-refined-elem}
Let $k$ be an algebraically closed field. Let $X \subset \mathbb{P}_k ^N$ be a reduced closed subscheme of dimension $1$. Suppose $N\gg 1$ and let $x \neq y$ be two closed points on $X_{\rm sm}$. Let $\Gr_{x+2y}(N-1, \P_k^N) \subset \Gr(N-1, \P_k^N)$ be the set of hyperplanes containing $\{x,y\}$ that do not intersect $X$ transversely at $y$. Then $\Gr_{\{x,y\}} ( N-1, \mathbb{P}_k ^N) \simeq \mathbb{P}_k ^{N-2}$ and $\Gr_{x+2y}(N-1, \P_k^N) \simeq \P_k^{N-3}$. 
\end{lem}

\begin{proof}
Recall that $\Gr_{\{ x, y \}} (N-1, \mathbb{P}_k ^N) \subset \Gr (N-1, \mathbb{P}_k ^N)$ is the set of hyperplanes containing $\{ x, y \}$. Since $x \not = y$, by elementary linear algebra on ranks of linear systems, we immediately have $\Gr_{\{x,y\}} ( N-1, \mathbb{P}_k ^N) \simeq \mathbb{P}_k ^{N-2}$. We prove the second assertion. Since $N \gg 1$, we can find a linear form $s_1 \in W = {H}^0(\P_k^N, \mathcal{O}(1))$ which does not vanish anywhere in $\{x,y\}$. This yields a $k$-linear map  $\alpha : W \to {\mathcal{O}_{X,\{x,y\}}}/{\fm_x\fm^2_y}=: \mathcal{O}_{\{ x + 2y \}}$ given by $\alpha(s) = s/{s_1}$. Since $k$ is algebraically closed, the ideal $\fm_y$ is generated by linear forms vanishing at $y$. Hence, the composite map $W \to {\mathcal{O}_{X,\{x,y\}}}/{\fm_x\fm^2_y} \surj {\mathcal{O}_{X,y}}/{\fm^2_y}=: \mathcal{O}_{\{ 2y \}}$ is surjective and $\alpha^{-1}(\fm^2_y)$ is precisely the set of linear forms in $W$ not transverse to $X$ at $y$. 

We first claim that $\alpha$ is surjective. Since $x, y$ are two distinct regular closed points of $X$, the set $\Gr_y(x, N-1, \P_k^N)$ is nonempty and hence, ${\fm_y}/{\fm_x\fm_y} \xrightarrow{\simeq} \mathcal{O}_{\{x\}}$ and there is a commutative diagram of short exact sequences:
\begin{equation}\label{eqn:Ref-0-1}
\xymatrix@C.8pc{
0 \ar[r] & \alpha^{-1}(\fm_x\fm_y) \ar[r] \ar[d] & \alpha^{-1}(\fm_y) \ar[r] \ar@{->>}[d] & \mathcal{O}_{\{x\}} \ar[r] \ar@{=}[d] & 0 \\
0 \ar[r] & {\fm_x \fm_y}/{\fm_x \fm^2_y} \ar[r] & {\fm_y}/{\fm^2_y} \ar[r] & \mathcal{O}_{\{x\}} \ar[r] & 0.}
\end{equation}
In particular, the first vertical map is surjective. Since $\Gr_x(y, N-1, \P_k^N) \neq \emptyset$, we conclude that $\alpha$ is surjective. 

To finish the proof, we look at the commutative diagram with exact rows
\begin{equation}\label{eqn:Ref-0-2}
\xymatrix@C.8pc{
0 \ar[r] & \ker (\alpha) \ar[r] \ar[d] & W \ar[r]^{\alpha \ \ \ } \ar@{=}[d] & \mathcal{O}_{\{x+2y\}} 
\ar[r] \ar[d] & 0 \\
0 \ar[r] & \alpha^{-1}(\fm^2_y) \ar[r] & W \ar[r] & \mathcal{O}_{\{2y\}} \ar[r] & 0.}
\end{equation}
Since the last vertical arrow is surjective with one-dimensional kernel, by the snake lemma, the first vertical arrow is injective with one-dimensional cokernel. Since $\mathbb{P}(\alpha^{-1}(\fm^2_y)) \simeq \P_k^{N-2}$, we conclude that $\Gr_{x+2y}(N-1, \P_k^N) \simeq \mathbb{P}(\ker (\alpha))\simeq \P_k^{N-3}$.
\end{proof}

\subsection{The Set-up$+ (\fs)$}\label{sec:Set-up2}

We suppose $k$ is an infinite perfect field. The set-up we now use repeatedly is the following situation, that we call the Set-up $+ (\fs)$: 

(1) \emph{The Set-up:} We still suppose the Set-up of \S \ref{sec:Set-up}, not necessarily specifying some closed subset $E \subset \widehat{B}$.

(2) \emph{The \fs-property:} There exists an affine open neighborhood $X_{\rm fs} \subset X_{\rm sm}$ of $\Sigma$, that is dense open in $X$, such that the projection $Z \to X$ is fs over $X_{\rm fs}$.

\subsection{Regularity of the original cycle over residual points}
We now discuss two central results:
Lemmas~\ref{lem:AD-open-refined-r=1} and ~\ref{lem:AD-open-refined-0}.
Recall that $X$ is equidimensional under the above assumptions.

\begin{lem}\label{lem:AD-open-refined-r=1}
Let $k$ be an algebraically closed field. Suppose $r=1$. We are under the Set-up $+ (\fs)$ of \S \ref{sec:Set-up2}. Let $x \in X_{\rm fs}$ be a closed point and let $S \subset X \setminus \{ x \}$ be another finite set of closed points. 

After replacing $\mathbb{P}_k ^N$ by a bigger projective space via a Veronese embedding if necessary, there exists a dense open subset $\mathcal{U}_S \subset \Gr_x(N-1, \mathbb{P}_k ^N)$ such that each $L \in \mathcal{U}_S(k)$ satisfies the following:
\begin{enumerate}
\item $L \cap ( (X \setminus X_{\rm fs}) \cup S)= \emptyset$.
\item $L$ intersects $X_{\fs}$ transversely.
\item $L \cap X$ consists of $(d+1)$-distinct closed points $c_0= x, c_1, \ldots, c_d$.
\item $Z$ is regular at all points lying over $\{ c_1, \ldots, c_d\}$. In particular, each component $Z_i$ does not meet other irreducible components at points lying over $\{ c_1, \ldots, c_d \}$.
\end{enumerate}
\end{lem}

\begin{proof}
Since $\dim(Z_{\rm sing}) = 0$, we see that $\widehat{f}(Z_{\rm sing})$ is a finite closed subset of $X$. Since $X_{\fs}$ is dense in $X$, we have $|X \setminus X_{\fs}| < \infty$. Hence, $T: = \left(\widehat{f}(Z_{\rm sing}) \cup (X \setminus X_{\fs}) \cup S\right) \setminus \{x\}$ is a finite closed subset of $X$. Thus the hyperplanes disjoint from $T$ form a dense open subset $\Gr (T, N-1, \mathbb{P}_k ^N)$ of $\Gr(N-1, \P_k^N)$ by Lemma \ref{lem:elem-0}. The set $\mathcal{U}_1:=\Gr^{\tr} (X, N-1, \mathbb{P}_k ^N) \cap \Gr (T, N-1, \mathbb{P}_k ^N)$ is dense open in $\Gr(N-1, \mathbb{P}_k ^N)$. If we show that $\mathcal{U}_S:= \mathcal{U}_1 \cap \Gr_x(N-1, \mathbb{P}_k ^N) \neq \emptyset$, then this set will be dense open in $\Gr_x(N-1, \mathbb{P}_k ^N)$. It is moreover clear that any $L \in \mathcal{U}_S(k)$ satisfies (1) $\sim$ (4). It remains to show that $\Gr^{\tr} (X, N-1, \mathbb{P}_k ^N) \cap \Gr_x(N-1, \mathbb{P}_k ^N)$ and $\Gr (T, N-1, \mathbb{P}_k ^N) \cap \Gr_x(N-1, \mathbb{P}_k ^N)$ are both nonempty.

Let $V$ be the set of linear forms in $ {H}^0(\P_k^N, \mathcal{O}(1))$ that vanish at $x$.  Note that $\dim \ |V| = N-1$ and that the maximal ideal $\fm_x \subset \mathcal{O}_{X,x}$ is generated by the members of $V$. Let $\mathcal{B} \subset X \times |V|$ be the incidence scheme consisting of pairs $(y, L)$ such that $L$ passes through $y$, but not transverse to $X$ at $y$. We study the fiber of $\pi_1: \mathcal{B} \to X$ over each $y \in X_{\rm sm} \setminus \{ x \}$.

Choose $s_1 \in V$ such that $s_1  (x) = 0$ but $s_1  (y) \not = 0$. Consider the map $\beta: V \to \mathcal{O}_{X,y} / \mathfrak{m}_y ^2$ given by $\beta (s) = s/ s_1 $. Since $\dim \ |V| = N-1$, while $\Gr_{\{x, y \}} (N-1, \mathbb{P}^N) \simeq \mathbb{P}^{N-2}$ and $\Gr_{\{x + 2y \}} (N-1, \mathbb{P}^N) \simeq \mathbb{P}^{N-3}$ by Lemma \ref{lem:AD-open-refined-elem}, we see that $\beta$ is surjective and $\mathbb{P}(\ker(\beta)) = \pi_1 ^{-1} (y)$ has dimension at most $N-3$, because $\dim_k (\mathcal{O}_{X,y}/ \mathfrak{m}_y ^2) = 2$. Thus, $\dim (\mathcal{B}) \leq \dim \ X + \dim ( \pi_1 ^{-1} (y)) \leq 1 + N-3 = N-2$. Hence, its image in $|V|$ under the projection $\pi_2: X \times |V| \to |V|$ is a proper closed subset (note that $X$ is projective). Since $N \gg0$, its complement $\Gr_x ^{\tr} (X, N-1, \mathbb{P}_k ^N)$ in $\Gr_x (N-1, \mathbb{P}_k ^N)$ is a dense open subset. Since $\dim(\Gr_x (N-1, \mathbb{P}_k ^N)) = N-1$ and $T \subset X$ is a finite set of closed points different from $x$, the assertion that $\Gr (T, N-1, \mathbb{P}_k ^N) \cap \Gr_x(N-1, \mathbb{P}_k ^N)$ is nonempty follows from Lemma \ref{lem:AD-open-refined-elem}. We have therefore finished the proof.
\end{proof}

In \S~\ref{sec:B-sep}, we will obtain a slightly stronger version of Lemma~\ref{lem:AD-open-refined-r=1}. This is done in Lemma \ref{lem:Non-collinear}. The difference in the latter lemma from the former is that (following the notations of Lemma \ref{lem:AD-open-refined-r=1}), after a possible reimbedding, we may impose an additional property that for $L \cap X = \{ c_0=x, c_1, \cdots, c_d \}$, no three points of them are collinear. 

At one bad extreme case, suppose $X$ is contained in a $2$-dimensional projective space. Then for any hyperplane $L$, which is a line, the hyperplane section $L \cap X$ is entirely collinear. This is an important obstacle to avoid. We will show in Lemma \ref{lem:Plane-case} that, after taking a Veronese reimbedding for a high enough degree $d \geq 3$, we can always avoid it. It will be improved for the higher dimensional case in Lemma \ref{lem:Higher-dim}. These two are some technical grounds needed in \S \ref{sec:B-sep}.

Once we can avoid the above extreme case using a Veronese reimbedding, then one can employ the following well-known general result (cf. \cite[Ch III, p.109]{ACGH}):
\begin{thm}[General position theorem]\label{thm:GPT} Let $N \geq 2$. Let $C \subset \mathbb{P}^N$ be an irreducible nondegenerate, possibly singular, curve of degree $d$. Then a general hyperplane meets $C$ in $d$ points, any $N$ of which are linearly independent.
\end{thm}

Recall that a closed embedding $X \subset \P^n_k$ of an integral projective scheme $X$ is said to be nondegenerate if no hyperplane of $\P^n_k$ contains $X$. We won't give the proof of Theorem \ref{thm:GPT} here. We mention that Theorem \ref{thm:GPT} for $N=2$ is immediate, while, for $N \geq 3$ reduces to the following special case (cf. \emph{loc.cit.}), that is more relevant to the paper:

\begin{lem}\label{lem:GPTcol}
Let $C \subset \mathbb{P}^N$ with $N \geq 3$ be an irreducible nondegenerate, possibly singular, curve of degree $d$. Then a general hyperplane meets $C$ in $d$ points, no three of which are collinear.
\end{lem}

\begin{remk}
To give a bit of the flavor of the proof of Lemma \ref{lem:GPTcol}, we remark that with some efforts (cf. \cite[pp.110-111]{ACGH} or imitate \cite[Proposition IV-3.8, p.311]{Hartshorne}), one can argue that if Lemma \ref{lem:GPTcol} fails, then all tangent lines to $C$ passes through a single fixed point $p \in C$. Then a linear projection from $p$ would shrink the entire curve $C$ to a point in $\mathbb{P}^{N-1}$. Since $C$ is nondegenerate, we can argue this cannot happen.

Such a curve in $\mathbb{P}^N$ all of whose tangent lines pass through a fixed point is called \emph{strange} (see \cite[p.311]{Hartshorne}). We remark that in case $C$ is nonsingular, it is known that the only nonsingular strange curves in any $\mathbb{P}^N$ are either a line or a conic in $\mathbb{P}^2$ in characteristic $2$ (see \cite[Theorem, Appendix to Ch II, p.76]{Samuel} or \cite[Theorem IV-3.9, p.312]{Hartshorne}).

We thank the referee for pointing to us that some technical part of our construction of the paper is relevant to non-collinearity of configurations of points, and strange curves. \qed
\end{remk}

Combined with the Bertini theorem (\cite[Theorem~1]{KA} or \cite{Jou}), we immediately extend Lemma \ref{lem:GPTcol} to the following higher dimensional version, which we use:

\begin{prop}[Linear general position theorem]\label{prop:GPTcol}
Let $X \subset \mathbb{P}^N$ with $N \geq 3$ be a nondegenerate, possibly singular, variety of degree $d$. Let $r= \dim \ X \geq 1$. Then for a general sequence of hyperplanes $H_1, \cdots, H_r$ in $\mathbb{P}^N$, the intersection $X \cap H_1 \cap \cdots \cap H_r$ has $d$ points, no three of which are collinear.
\end{prop}

Note that the above Proposition \ref{prop:GPTcol} holds for schemes that are nondegenerate in the projective spaces of dimension at least $3$. This is another view of why we had a pathology about non-collinearity when $X$ was contained in a $2$-dimensional projective space in the paragraph before Theorem \ref{thm:GPT}.

As said before, to avoid this problem, we need to replace the embedding by a bigger Veronese embedding. This is discussed now in the following:

\begin{lem}\label{lem:Plane-case}
Let $C \subset \P^n_k$ be a reduced projective curve. Suppose that there exists a $2$-dimensional linear subspace $L \subset \P^n_k$ such that $C \subset L$. Let $\vartheta : \P^n_k \inj \P^N_k$ be the $d$-uple Veronese embedding with $d \ge 3$. Then the image of each irreducible component of $C$ via $\vartheta$ does not lie inside a $2$-dimensional linear subspace of $\P^N_k$.
\end{lem}
\begin{proof}
We can assume $C$ is an irreducible curve in order to prove the lemma. After a linear change of coordinates in $\P^n_k$, we may assume that $\P^n_k = \P(V)$ and $L = \P(W)$, where $V$ is an $(n+1)$-dimensional $k$-vector space with a basis $\{ x_0, \cdots, x_n\}$, and $W = {\rm Span}_k \{ x_0, x_1, x_2 \}$ is a subspace of $V$.
 For any closed embedding $f : C \inj \P^m_k$, we let $d_f(C)$ denote the degree of $C$ under $f$.
  
Let $\iota : C \inj L$ be the closed embedding as given in the assumption of the lemma. Let $d_0:= d_{\iota} (C)  \ge 1$. Since $L$ is linear in $\P^n_k$, the degree of $C$ under the composite of the embeddings $C \inj L \inj \P^n_k$ is also $d_0$.

Toward contradiction, suppose that there is a $2$-dimensional linear subspace $L' \subset \P^N_k$ such that $\vartheta(C) \subset L'$, where $\vartheta : \P^n_k \inj \P^N_k$ is the $d$-uple Veronese embedding with $d \ge 3$. We denote the resulting embedding $C \inj L'$ by $\vartheta|_C$.

By our choice of the embedding $L \inj \P^n_k$, we have a commutative diagram
\begin{equation}\label{eqn:Plane-case-0}
  \xymatrix@C.8pc{
    C \ar@{^{(}->}[r]^-{\iota} \ar@{^{(}->}[dr] &
    L \ar@{^{(}->}[r] \ar@{^{(}->}[d]^-{\vartheta'} & \P^n_k 
\ar@{^{(}->}[d]^-{\vartheta} \\
    & M \ar@{^{(}->}[r] & \P^N_k,}
\end{equation}
where $M \cong \P^r_k$ (with $r = {(d+1)(d+2)}/2 - 1$). 
The horizontal arrows in the right
square are linear embeddings
and the vertical arrows are the $d$-uple Veronese embeddings.

The linearity of the inclusion $L' \inj \P^N_k$ implies that $d_{\vartheta|_C}(C)$ coincides with the degree of $C$ under the composite closed embedding $C \inj L' \inj \P^N_k$. By the same argument, the degree of $C$ for this composite embedding coincides with the degree of $C$ for the composite embedding $C \inj L \inj M$. Since $\vartheta'$ is the $d$-uple Veronese embedding, it follows that the degree of $C$ for the latter composite embedding is $d_0 d$. We conclude that $d_{\vartheta|_C}(C) = d_0 d$.

If we now apply the degree-genus adjunction formula for plane curves to the embedding $\iota$, we get $g_a(C) = \frac{(d_0-1)(d_0-2)}{2}$, where $g_a(C)$ is the arithmetic genus of $C$. The same formula for the embedding $\vartheta|_C$ yields $g_a(C) = \frac{(d_0 d-1)(d_0 d-2)}{2}$. 

Hence $(d_0 -1)(d_0 -2) = (d_0 d -1) (d_0 d -2)$, i.e. $d_0 ^2 (d^2 -1) - 3 d_0 (d-1) = 0$. This factors into 
\begin{equation}\label{eqn:arith g}
d_0 (d-1) (d_0 (d+1) -3) = 0.
\end{equation}
 Since $d_0 \geq 1$ and $d \geq 3$, the left hand side of \eqref{eqn:arith g} is $\geq 1 \cdot 2 \cdot (1 \cdot 4 - 3) > 0 $, so that the equality of \eqref{eqn:arith g} cannot hold, thus a contradiction. This proves the lemma.
\end{proof}

An analogue of Lemma~\ref{lem:Plane-case} in higher dimensions is the following.

\begin{lem}\label{lem:Higher-dim}
Let $\iota : X \inj\P^n_k$ be a reduced projective scheme of pure dimension $r \ge 2$. Assume that the degree of each irreducible component of $X$ in $\P^n_k$ is at least two. Let $\Sigma \subset X$ be a finite set of closed points. For an integer $d \ge 1$, let $\vartheta : \P^n_k \inj \P^N_k$ be the $d$-uple Veronese embedding.

Then for all sufficiently large $d \geq 3$, (depending on $X, \Sigma, n$ and the degrees of the irreducible components of $X$ in $\P^n_k$), 
a general intersection $H_1 \cap \cdots \cap H_{r-1} \cap \vartheta(X)$ of $X$ with hyperplanes $H_i$'s in $\Gr_{\Sigma}(N-1, \P^N_k)(k)$ is a reduced curve, none of whose irreducible component is contained in a $2$-dimensional linear subspace of $\P^N_k$.
\end{lem}

\begin{proof}
By the Bertini theorems of Kleiman-Altman \cite[Theorem~1]{KA}, an intersection of $\vartheta(X)$ with $(r-1)$ general hyperplanes containing $\Sigma$ in a large enough $d$-uple Veronese embedding $\vartheta$ is a curve $C$, whose intersection with every irreducible component of $\vartheta(X)$ is again irreducible. Since $k$ is perfect and $X$ is reduced, it is actually geometrically reduced. It follows therefore from the Bertini theorem of Jouanolou \cite[Th{\'e}or{\`e}me 6.3]{Jou} that $C$ can be chosen to be reduced.

Let $X_1, \ldots , X_t$ be the irreducible components of $X$ and let $C_1, \ldots , C_t$ denote the irreducible components of $C$.

Let $s_i$ be the degree of $X_i$ in $\P^n_k$ so that the degree of $X$ in $\P^n_k$ is $s = \sum_{i=1} ^r s_i$ (see \cite[Proposition~I-7.6, p.52]{Hartshorne}). Let $C = H_1 \cap \cdots \cap H_{r-1} \cap \vartheta(X)$ be as above. Let $d_{\iota}(C_i)$ denote the degree of $C_i$ in $\P^n_k$ via the inclusion $\iota: C \inj X \inj \P^n_k$ and let $d_{\vartheta}(C_i)$ denote the degree of $C_i$ in $\P^N_k$. Each of the hyperplanes $H_1, \ldots , H_{r-1} \subset \P^N_k$ restricts to a unique hypersurface of degree $d$ in $\P^n_k$. Since these hyperplanes are sufficiently general, an elementary degree computation shows that $d_{\iota}(C_i) = d^{r-1}s_i$ and $d_{\vartheta}(C_i) = d^rs_i$ for each $1 \le i \le t$. We need to show that if $d$ is sufficiently large, then each $C_i = H_1 \cap \cdots \cap H_{r-1} \cap \vartheta(X_i)$ is not contained in a $2$-dimensional linear subspace of $\P^N_k$.
To show this, we can assume that $X$ and $C$ are irreducible.
In particular, $d_{\iota}(C) = sd^{r-1}$ and $d_{\vartheta}(C) = sd^r$.

We shall prove our assertion as an application of Castelnuovo's bound for the genus of curves. Let $3 \le n' \le n$ be the smallest integer such that $X \subset \P^{n'}_k \subset \P^n_k$, where the first embedding is nondegenerate and the second embedding is linear. Note that the lower bound on $n'$ is forced by our assumption on the lower bounds of the dimension of $X$ and its degree in $\P^n_k$. 

Since $H_1, \ldots , H_{r-1}$ restrict to general hypersurfaces of degree $d$ in $\P^n_k$, we see that they restrict to hypersurfaces of the same degree in $\P^{n'}_k$. Since a hypersurface (of degree at least two) section of a nondegenerate closed subvariety of a projective space is necessarily nondegenerate (looking at the homogeneous coordinate rings), we conclude that the composite embedding $C \inj X \inj \P^{n'}_k$ is also nondegenerate. Furthermore, the degrees of $X$ and $C$ inside $\P^{n'}_k$ are the same as their respective degrees inside $\P^n_k$.

Let $m \ge 1$ and $0 \le \epsilon < n'-1$ be two integers such that $sd^{r-1} - 1 = m(n'-1) + \epsilon$. 
It follows from Castelnuovo's bound on the arithmetic genus (see \cite[Chapter~3]{EH}, \cite[Ch III, p.116]{ACGH}, and see \cite[Remark following Lemma~2.1]{Ballico} for singular curves) of $C$ that
\begin{equation}\label{eqn:Higher-dim-0}
g_a(C) \le \frac{(n'-1)m(m-1)}{2} + m \epsilon.
 \end{equation}

Since $n'-1 \ge 2$ and $d$ is sufficiently large, we can assume $m < sd^{r-1} - 1$. We thus get

\begin{equation}\label{eqn:Higher-dim-1}
\begin{array}{lll} 
2 g_a(C) & \le & (n'-1)m(m-1) + 2m \epsilon  <  (n'-1)m(m-1) + 2m(n'-1) \\
& = & (n'-1)m(m+1)  \le  (sd^{r-1} - 1)(sd^{r-1} -1) = (sd^{r-1} -1)^2.
\end{array}
\end{equation}

Now toward contradiction, suppose that inside $\P^N_k$, the curve $C$ is contained in a $2$-dimensional linear subspace $L \subset \P^N_k$. Since $d_{\vartheta}(C)$ is equal to the degree of $C$ inside $L$, the degree-genus adjunction formula for the
embedding $C \inj L \cong \P^2_k$, yields
$2g_a(C) = (sd^r-1)(sd^r -2)$. Note that if we let $e':= sd^{r-1} -1$, then
\begin{equation}\label{eqn:Higher-dim-2}
\begin{array}{lll}
2g_a(C) & = & (sd^r-1)(sd^r -2) = (d (sd ^{r-1} -1) + d -1) ( d (sd ^{r-1} -1) + d -2) \\
& = & d^2 (e')^2 + (2d-3) d (e') + (d-1)(d-2),
\end{array}
\end{equation}
and because $d  \geq 3$ and $s>0$, we have $2g_a (C) >  (e')^2  + e'  + 0 \geq (e')^2. $

On the other hand, from ~\eqref{eqn:Higher-dim-1} we had $2g_a(C) \leq (e')^2$. This is a contradiction.  
\end{proof}

We now present the aforementioned improvement of Lemma \ref{lem:AD-open-refined-r=1}:

\begin{lem}\label{lem:Non-collinear}
Let $X \inj \P^N_k$ and $x \in X_{\rm fs}$ be as in Lemma \ref{lem:AD-open-refined-r=1}. After replacing $\P^N_k$ by a bigger projective space via a Veronese embedding, there exists a dense open $\mathcal{U}_S \subset \Gr_x(N-1, \P^N_k)$ such that every $L \in \mathcal{U}_S(k)$ satisfies the following.
\begin{enumerate}
\item The conditions (1) $\sim$ (4) of Lemma \ref{lem:AD-open-refined-r=1}.
\item No three points of $L \cap X = \{x = c_0, c_1, \ldots , c_d\}$ are collinear.
\end{enumerate}
\end{lem}
\begin{proof}
Suppose first that $X$ does not lie inside any 2-dimensional linear subspace of $\P^N_k$.
In this case, we choose $\mathcal{U}_S$ just as in Lemma \ref{lem:AD-open-refined-r=1} so that (1) holds. The condition (2) holds by Proposition \ref{prop:GPTcol}. Hence the lemma is proven in this case.

Suppose now that $X$ lies inside a 2-dimensional linear space of $\P^N_k$. In this case, we choose a suitable Veronese embedding $\P^N_k \inj \P^{N'}_k$ such that the image of each irreducible component of $X$ does not lie in any $2$-dimensional linear subspace of $\P^{N'}_k$ applying Lemma \ref{lem:Plane-case}. Then after re-embedding if necessary, we have a nonempty open subset $\mathcal{U}_S \subset \Gr_x(N'-1, \P^{N'}_k)$ such that the conditions (1) $\sim$ (4) of Lemma \ref{lem:AD-open-refined-r=1} hold. 

In doing so, we can make sure that $X$ is nondegenerate in a projective space of dimension at least $3$. Then the condition (1) holds by the choice of $\mathcal{U}_S$, while the condition (2) holds by Proposition \ref{prop:GPTcol}.
This proves the lemma. 
\end{proof}

The following result generalizes Lemma ~\ref{lem:Non-collinear} to higher dimensional $r \ge 1$.

\begin{lem}\label{lem:AD-open-refined-0}
Let $k$ be an algebraically closed field. Suppose $r\geq1$. We are under the Set-up $+ (\fs)$ of \S \ref{sec:Set-up2}. Let $x \in X_{\rm fs}$ be a closed point and let $S \subset X \setminus \{ x \}$ be another finite set of closed points. 

After replacing $\mathbb{P}_k ^N$ by a bigger projective space via a Veronese embedding if necessary, we have the following property: given any hyperplane $H_0 \subset \mathbb{P}^N_k$ disjoint from $S \cup \{ x \}$ and a general $L_0 \in \Gr ^{\tr} _{S \cup \{x\}}(H_0, N-r+1, \mathbb{P}_k ^N)(k)$, there exists a dense open subset $\mathcal{U}_S \subset \Gr_x ^{\tr} (L_0, N-1, \mathbb{P}_k ^N)$ such that each $L \in \mathcal{U}_S(k)$ satisfies the following.
\begin{enumerate}
\item $L \cap L_0 \cap ( (X \setminus X_{\rm fs}) \cup S)= \emptyset$.
\item $L\cap L_0$ intersects $X_{\fs}$ transversely.
\item $L \cap L_0 \cap X$ has $(d+1)$-distinct closed points $c_0= x, c_1, \ldots, c_d$.
\item $Z$ is regular at all points lying over $\{ c_1, \ldots, c_d\}$. In particular, each component $Z_i$ does not meet other irreducible components at points lying over $\{ c_1, \ldots, c_d \}$.
\item $L_0 \cap X$ is an equidimensional reduced curve none of whose irreducible component lies inside a $2$-dimensional linear subspace of $\P^N_k$.
\item No three points of $L \cap L_0 \cap X = \{x = c_0, c_1, \ldots , c_d\}$ are collinear.
\end{enumerate}
\end{lem}

\begin{proof}
  In case $r=1$, we have $\Gr (N-r+1, \mathbb{P}_k ^N) = \Gr (N, \mathbb{P}_k ^N) = \{ \mathbb{P}_k ^N \}$ so that $L_0 = \mathbb{P}_k ^N$ and Lemma \ref{lem:AD-open-refined-0} follows from Lemmas ~\ref{lem:Plane-case} and ~\ref{lem:Non-collinear}. Hence we may assume $r \geq 2$. Let $X_1, \ldots , X_t$ be the irreducible components of $X$.

We saw in the proof of Lemma~\ref{lem:Higher-dim} that the Bertini theorems of Kleiman-Altman \cite[Theorem~1]{KA} and Jouanolou \cite{Jou} imply that an intersection of $X$ with $(r-1)$ general hyperplanes containing $S \cup \{ x \}$ in a large enough Veronese embedding of $\P^N_k$ is a reduced curve $C$ whose intersection with every irreducible component of $X$ is irreducible. This curve $C$ contains $S \cup \{x\}$. We can also ensure that no component of $C$ is contained in $\widehat{f}(Z_{\rm sing}) \cup (X \setminus X_{\fs})$, it is regular at points away from $X_{\rm sing}$, and for each component of $Z |_{C \times \widehat{B}}$, its projection to $ \widehat{B}$ is non-constant.

Hence, after replacing the embedding $\eta: X \inj \P_k^N$ by its composition with a Veronese embedding of $\P_k^N$, we can find an $(r-1)$-tuple of general hyperplanes $(H_1, \ldots , H_{r-1})$, each in $\Gr_{S \cup \{x\}}(N-1, \P_k^N)$, such that the linear subspace $L_0 = H_1 \cap \cdots \cap H_{r-1}$ has the following properties.
\begin{listabc}
\item $L_0$ is transverse to $H_0$.
\item $C = L_0 \cap X$  is a reduced curve none of whose components lies in $\widehat{f}(Z_{\rm sing}) \cup (X \setminus X_{\fs})$.
\item
$C \cap X_i$ is irreducible for each $1 \le i \le t$.
\item $C$ is regular at points away from $X_{\rm sing}$.
\item For each component of $Z |_{C \times \widehat{B}}$, the projection to $\widehat{B}$ is non-constant.
\end{listabc}
 
Let $S' := (C \setminus \{x\}) \cap (\widehat{f} (Z_{\rm sing}) \cup (X \setminus X_{\fs}) \cup S)$, which is a finite closed subset of $C$.

Note from the definition of the degree of the embedding $\eta : X \inj \P_k^N$ that a general hyperplane inside $L_0$ will intersect $C$ at $(d+1)$ distinct closed points. Applying Lemma \ref{lem:Non-collinear} to the curve $C$, the finite set $S'$, and $L_0 \simeq \mathbb{P}_k ^{N-r+1}$ (which is regarded as the ambient projective space for $C$), there exists a dense open subset $\mathcal{U}_{C,S'} \subset \Gr_x (N-r, L_0)$ that satisfies the assertions $(1) \sim (2)$ of Lemma \ref{lem:Non-collinear}. 
Note that as $N \gg r$, the subset $\Gr^{\rm tr}(L_0, N-1, \P_k^N)$ is dense open in $\Gr (N-1, \mathbb{P}_k ^N)$. 

Consider the regular map
\begin{equation}\label{eqn:refined-adm-0-1}
\theta_{L_{0}}: \Gr^{\rm tr} (L_0, N-1, \P_k^N) \to \Gr(N-r, L_0),
\end{equation}
given by $\theta_{L_0}(L) = L \cap L_0.$

One checks that $\theta_{L_0}$ is a surjective smooth morphism of relative dimension $r-1$. Since $\theta_{L_0}$ is a smooth and surjective morphism such that $\theta_{L_0}^{-1}(\Gr_x(N-r, L_0)) = \Gr^{\rm tr}_x(L_0, N-1, \P_k^N)$, we see that $\mathcal{U}_S:=\theta_{L_0}^{-1}(\mathcal{U}_{C, S'})$ is a dense open subset of $\Gr^{\rm tr}_x(L_0, N-1, \P_k^N)$.

We want to show that each $L \in \mathcal{U}_S (k)$ satisfies the desired conditions $(1) \sim (4)$. This is a tautology, but let us write it in detail: suppose $L \in \mathcal{U}_S(k)$, i.e., $\theta_{L_0}(L) \cap S' = \emptyset$ and $\theta_{L_0}(L) = L \cap L_0$ satisfies $(1)\sim (4)$ with $Z$ replaced by $Z|_{C \times \widehat{B}}$. Since $\theta_{L_0}(L) \cap \left( (X \setminus X_{\fs}) \cup S\right)= L \cap (L_0 \cap X) \cap \left( (X \setminus X_{\fs}) \cup S\right) = \theta_{L_0}(L) \cap C \cap \left((X \setminus X_{\fs} ) \cup S\right) \subset \theta_{L_0}(L) \cap S'$, and since $x \in X_{\fs}$, we see that $\theta_{L_0}(L) \cap ( (X \setminus X_{\fs}) \cup S) = \emptyset$, proving (1).

Since $L$ intersects $L_0$ transversely, which in turn intersects $X$ transversely along $X_{\rm sm}$ by (b) and (d) above, we see that $\theta_{L_0}(L)$ intersects $X$ transversely along $X_{\rm sm}$, proving (2). Also, $\theta_{L_0}(L) \cap X = \theta_{L_0}(L) \cap C = \{x=c_0, c_1, \ldots , c_d\}$ with $c_i \neq c_j$ for $i \neq j$, proving (3). Finally, since $(C \cap \widehat{f}(Z_{\rm sing})) \setminus \{x\} \subset S'$ and since $\theta_{L_0}(L) \cap S' = \emptyset$, we see that $Z$ is regular at all points lying over $c_i$ for $1 \le i \le d$, proving (4). 

We now prove (5). First of all, if the degree of any irreducible component of $X$ inside $\P^N_k$ was less than or equal to two, before we do anything else, we first could have replaced $\P^N_k$ by its suitable Veronese embedding so as to ensure that the degree of any irreducible component of $X$ is bigger than $2$. In doing so, we see using Lemma~\ref{lem:Higher-dim} that  the intersection $L_0$ of general $(r-1)$ hyperplanes $H_1, \ldots , H_{r-1}$ lying in $\Gr_{S \cup \{x\}}(N-1, \P^N_k)$ will have the property that $L_0$ will satisfy the above (a) $\sim$ (e), and $L_0 \cap X$ will be a reduced curve none of whose irreducible component is contained in a $2$-dimensional linear subspace of $\P^N_k$. Note that since $X$ is equidimensional and $L_0$ is general, 
the curve $L_0 \cap X$ will have this property too. This proves (5).
The last property (6) is a direct consequence of (5), the condition (2) of Lemma \ref{lem:Non-collinear}, which we already achieved from the beginning, and Proposition \ref{prop:GPTcol}.
\end{proof}

Later, the set $\{c_1, \ldots , c_d\}$ that we obtained in Lemma \ref{lem:AD-open-refined-0} will be taken to be $L^+ (x)$ for $x \in \Sigma$, where $\Sigma$ is the given set of finitely many closed regular points of $X$. 
This means the regularity of $Z$ at points lying over the residual points $L^+ (\Sigma)$. We will come back to this discussion, and it will be finished in Proposition \ref{prop:Main-I}.

\section{Vertical separation of residual fibers}\label{sec:pre-Ad-set}
In this section, we prove some results which we shall need in order to prove the regularity of the residual cycle of $Z$ along $\Sigma$. The main goal is to show that the distinct fibers, of the projection $Z \to X$ to the ``horizontal axis" over the residual points of $\Sigma$ (for a suitable linear projection) are mapped to disjoint sets under the projection $\widehat{g} : Z \to \widehat{B}$ to the ``vertical axis." We call this property of linear projections, \emph{the vertical separation of residual fibers}.
We continue to use the Set-up $+ (\fs)$ of \S \ref{sec:Set-up2}.

\subsection{Separating residual fibers of $Z$ along $\widehat{B}$: the local case}\label{sec:B-sep}

Let $k$ be an algebraically closed field. In Lemma \ref{lem:AD-open-refined-r=1}, under certain assumptions, we found a nonempty open subset of a Grassmannian such that each member $L$ satisfies the properties (1) $\sim$ (4) there. In Lemma \ref{lem:Non-collinear}, after choosing a Veronese reimbedding into a bigger projective space, we achieved an additional non-collinearity of any three points of the hyperplane sections. It was generalized to Lemma \ref{lem:AD-open-refined-0} for $r \geq 1$.

In \S \ref{sec:B-sep}, we want to further strengthen them, by constructing a nonempty open subset for which we have an additional separation property, which will be called the property $(I)$. This is eventually done in Proposition \ref{prop:Ad-open-final-I}.

Up to Lemma \ref{lem:Ad-open-final-I}, we assume the following. We suppose $r=1$. We let $x \in X_{\rm fs}$ be a closed point and let $S \subset X \setminus \{ x \}$ be another finite set of closed points. For any map $W \to X$ and a closed point $y \in X$, let $W_y$ be the reduced fiber of $W$ over $y$. We work under the set-up of Lemma \ref{lem:Non-collinear}, which includes Lemma~\ref{lem:AD-open-refined-r=1}.

Since we want to prove a property called $(I)$ by a kind of double induction argument on the pairs of numbers $(m,n)$ with $0 \le m \le n \le d-1$, we find it convenient to temporarily introduce some intermediate notations.

\begin{defn}\label{defn:refined-admissible}
For $1 \le n \le d-1$ and $0 \le m \le n$, we say that a member $\underline{H} = (H, c_1, \cdots , c_d) \in \Gr_x(N-1, \P_k^N) (k) \times X^d$ is \emph{$(Z, x, m,n)$-admissible}, if $H$ satisfies the properties (1) and (2) of Lemma \ref{lem:Non-collinear} with $H \cap X = \{x = c_0, c_1, \ldots , c_d\}$, together with the additional property:
\begin{equation}\label{eqn:Imn}
(I)_{m,n}:= \tuborg \widehat{g} (Z_{c_i}) \cap \widehat{g}(Z_{c_j}) = \emptyset \ \ \mbox{ for } 0 \leq i \not = j \leq n, \\
\widehat{g} (Z_{c_i}) \cap \widehat{g} (Z_{c_{n+1}}) = \emptyset \ \ \mbox{ for } 0 \leq i \leq m.
\sluttuborg
\end{equation}
We remark that for $n=0$ (thus we have just $(I)_{0,0}$), the first condition of \eqref{eqn:Imn} is empty.
\end{defn}

Before anything else, we note the following elementary fact:

\begin{lem}\label{lem:Ad-open-finite}
The projections $\widehat{f}: Z \to X$ and $\widehat{g}: Z \to \widehat{B}$ are finite and the sets $\widehat{g}( Z_x) \subset \widehat{B}$ and $\widehat{g}^{-1} (\widehat{g} ( Z_x)) \subset Z$ are finite subsets of closed points.
\end{lem}

\begin{proof}
Note that $\widehat{f}:Z \to X$ is a projective morphism of reduced curves such that its restriction over the dense open subset $X_{\rm fs}$ of $X$ is fs. Hence $\widehat{f}$ is a projective quasi-finite morphism, hence a finite morphism. Since $\widehat{g}$ is a projective morphism from a curve which is non-constant on each component of the source, it must also be finite. Since $Z_x$ is a finite set, as $Z$ is fs over $X_{\rm fs}$ and $x \in X_{\rm fs}$, the lemma now follows.
\end{proof}

Let $V_d \subset X^d$ be the nonempty open subset whose coordinates are all distinct from each other and distinct from $x$ as well. More precisely, this is the complement of the union of all the small diagonals $\Delta_{i,j} \subset X^d$ defined by the equation $y_i = y_j$ for $1 \leq i< j \leq d$ as well as the subschemes given by $y_i = x$ for $1 \leq i \leq d$.
Let $\pi : X^d \to {\rm Sym}^d(X)= X^d/ \mathfrak{S}_d$ be the quotient
map for the action by the symmetric group $\mathfrak{S}_d$ which permutes the
coordinates. Since $\mathfrak{S}_d$ acts freely on $V_d \subset X^d$, the restriction $\pi: V_d \to \pi(V_d)$ is finite {\'e}tale of degree $d!$.

Inside $V_d$, we consider the following subsets of `bad points' that do not satisfy the analogue of the condition $(I)_{m,n}$ for $(y_1, \ldots, y_d) \in V_d$. That is, for $y_0:= x$, let $D_0:= \emptyset$, while for $n\geq 1$, let $D_n \subset V_d$ be the subset of points $(y_1, \ldots, y_d)$ such that $\widehat{g} (Z_{y_i}) \cap \widehat{g} (Z_{y_j}) \not = \emptyset$ for some $0 \leq i \not = j \leq n$ and $G_m^n \subset V_d$ be the subset of points such that $\widehat{g} (Z_{y_i}) \cap \widehat{g} (Z_{y_{n+1}}) \not = \emptyset$ for some
$0 \leq i \leq m$. 

Express $D_n = D_{n,1} \cup D_{n,2}$, where $D_{n,1}$ consists of the points $(y_1, \ldots, y_d) \in V_d$ such that $\widehat{g}(Z_{y_i}) \cap \widehat{g} (Z_{y_j}) \not = \emptyset$ for some $1 \leq i \not = j \leq n$, while $D_{n,2}$ consists of the points $(y_1, \ldots, y_d) \in V_d$ such that $\widehat{g}(Z_{y_0}) \cap \widehat{g} (Z_{y_i}) \not = \emptyset$ for some $1 \leq i \leq n$. We also write $G_m^n = \bigcup_{i=0} ^m G^n_{m,i} $, where $G^n_{m,i}$ consists of the points $(y_1, \ldots, y_d) \in V_d$ such that $\widehat{g}(Z_{y_i}) \cap \widehat{g} (Z_{y_{n+1}}) \not = \emptyset$ for $0 \leq i \leq m$.
We check these `bad sets' are closed:

\begin{lem}\label{lem:DGclosed}
The subsets $D_{n,i}$ for $i=1,2$ and $G^n_{m,i}$ for $0 \leq i \leq m$ are closed subsets of $V_d$. In particular, $D_n$ and $G_m ^n$ are closed subsets of $V_d$.
\end{lem}
\begin{proof}
Let $E_{n,1} \subset \widehat{B}^d$ be the subset of points $(b_1, \ldots, b_d)$ such that $b_i = b_j$ for some $1 \leq i \not = j \leq n$. Let $E_{n,2} \subset \widehat{B}^d$ be the subset of points $(b_1, \ldots, b_d)$ such that $b_i \in \widehat{g}(Z_x)$ for some $1 \leq i \leq n$. The set $E_{n,1}$ is certainly closed in $\widehat{B}^d$, while $E_{n,2}$ is closed in $\widehat{B}^d$ because $\widehat{g}(Z_x)$ is finite by Lemma \ref{lem:Ad-open-finite}. One checks that $D_{n,i} = \widehat{f}^{\times d} (( \widehat{g}^{\times d})^{-1} (E_{n,i})) \cap V_d$ for $i=1,2$, where $\widehat{f}^{\times d}: Z^d \to X^d$ and $\widehat{g}^{\times d}: Z^d \to \widehat{B}^d$ are the direct products of $\widehat{f}$ and $\widehat{g}$. Since $\widehat{f}^{\times d}$ is finite by Lemma \ref{lem:Ad-open-finite}, this shows that $D_{n,i}$ is closed in $V_d$.

Similarly, let $J^n _{m,0} \subset \widehat{B}^d$ be the subset of points $(b_1, \ldots, b_d)$ such that $b_{n+1} \in \widehat{g}(Z_x)$. This is closed since $\widehat{g} (Z_x)$ is finite by Lemma \ref{lem:Ad-open-finite}. For $1 \leq i \leq m$, let $J^n_{m,i} \subset \widehat{B}^d$ be the subset of points $(b_1, \ldots, b_d)$ such that $b_i = b_{n+1}$. This is also closed. One checks that $G^n _{m,i} = \widehat{f}^{\times d} (( \widehat{g}^{\times d})^{-1} (J^n_{m,i})) \cap V_d$, and this shows that $G^n _{m,i}$ is closed in $V_d$ for $0 \leq i \leq m$.
\end{proof}

Coming back to the story, we let let $\mathcal{U}_S \subset \Gr_x (N-1, \mathbb{P}_k ^N)$ be the nonempty open 
set of Lemma \ref{lem:Non-collinear}. Let $\mathcal{U}_S \to  {\rm Sym}^d(X)$ be the map given by $L \mapsto \sum_{i=1} ^d [c_i]$, where $L \cap X = \{x = c_0, c_1, \ldots , c_d\}$. By the condition (3) of Lemma \ref{lem:AD-open-refined-r=1}, its image is in $\pi (V_d)$. Define $\mathcal{V}_S$ by the Cartesian diagram
\begin{equation}\label{eqn:Ad-nonempty-0}
\xymatrix@C1pc{
{\mathcal{V}}_{S} \ar[d]_{\psi} \ar[r]^{e} & V_d \ar[d]^{\pi} \\
\mathcal{U}_{S} \ar[r] & \pi(V_d),}
\end{equation}
so that $\psi$ is a finite surjective {\'e}tale map. 
The set $ \mathcal{V}_S \setminus e^{-1} (D_n \cup G^n_m)$ is open in $\mathcal{V}_S$ by Lemma \ref{lem:DGclosed}. Via the open map $\psi$, we define the open subset $\mathcal{U}_{m,n} ^S := \psi ( \mathcal{V}_S \setminus e^{-1} (D_n \cup G^n_m)) \subset \mathcal{U}_S$. This is open in $\Gr_x (N-1, \mathbb{P}_k ^N)$. %We have therefore shown:

\begin{lem}\label{lem:Ad-open-final-I}
For $0 \leq n \leq d-1$ and $0 \leq m \leq n$, the subset $\mathcal{U}_{m,n} ^{S}  \subset \Gr_x (N-1, \mathbb{P}_k ^N)$ is nonempty. In particular, it is a dense open subset of $\Gr_x(N-1, \P^n_k)$.
\end{lem}

\begin{proof}

\textbf{Step 1.} \emph{$\mathcal{U}_{0,0} ^{S} \not = \emptyset$.}

Note that the condition $(I)_{0,0}$ is independent of the choice of an $x$-fixing order on $L \cap X$. Let $T = S \cup \left(\widehat{f}(\widehat{g}^{-1} (\widehat{g}(Z_x))) \setminus \{x\}\right)$. This is a finite closed subset of $X$ by Lemma \ref{lem:Ad-open-finite}. Applying Lemma \ref{lem:AD-open-refined-r=1} to $T$ (in the place of $S$ there), we obtain a dense open subset $\mathcal{U}_T $ of $\mathcal{U}_S \subset \Gr_x(N-1, \P_k^N)$.
On the other hand, the condition (1) (in Lemma \ref{lem:AD-open-refined-r=1}) for $T$ implies that for each $L \in \mathcal{U}_T (k)$, we have $L \cap  \left(\widehat{f}(\widehat{g}^{-1} (\widehat{g}(Z_x))) \setminus \{x\}\right)= \emptyset$, which shows that $\widehat{g} (Z_{c_0}) \cap \widehat{g} (Z_{c_j})= \emptyset$ for each $j \not = 0$ when $L \cap X = \{x = c_0, c_1, \ldots , c_d\}$, for every $x$-fixing order on $L \cap X$. Thus $(I)_{0,0}$ holds, and $\mathcal{U}_T \subset \mathcal{U}_{0,0} ^S$, in particular $\mathcal{U}^{S}_{0,0} \not = \emptyset$.

\vskip .2cm

\textbf{Step 2.} \emph{For $0 \leq n \leq d-2$, if $\mathcal{U}_{n,n} ^S \not = \emptyset$, then $\mathcal{U}_{0, n+1} ^{S} \not = \emptyset$}.

If $\mathcal{U}^{S}_{n,n} \neq \emptyset$, then it is a dense open subset of $\Gr_x(N-1, \P_k^N)$. In particular, for the dense open subset $\mathcal{U}_T \subset \Gr_x (N-1, \mathbb{P}_k ^N)$ of Step 1, 
the intersection $\mathcal{U}_{n,n} ^{S} \cap \mathcal{U}_T$ is dense open in $\Gr_x(N-1, \P_k^N)$. But, by definition, one notes that $\mathcal{U}_{n,n} ^{S} \cap \mathcal{U}_T \subset \mathcal{U}_{0, n+1} ^{S}$ 
so that $ \mathcal{U}_{0, n+1} ^{S} \not = \emptyset$. 

\vskip .2cm

\textbf{Step 3.} \emph{For $0 \leq n \leq d-1$ and $0 \leq m \leq n-1$, if $\mathcal{U}_{m,n} ^{S} \not = \emptyset$, then $\mathcal{U}_{m+1, n} ^{S} \not = \emptyset$.}

If $\mathcal{U}^{S}_{m,n} \not = \emptyset$, then it is dense open in $\Gr_x (N-1, \mathbb{P}_k ^N)$. For the dense open subset $\mathcal{U}_T \subset \Gr_x (N-1, \mathbb{P}_k ^N)$ of Step 1, the intersection $\mathcal{U}_{m,n} ^{S} \cap \mathcal{U}_T$ is therefore nonempty dense open in $\Gr_x(N-1, \P_k^N)$. 

Fix an element $L_0 ' \in (\mathcal{U}_{m,n} ^{S} \cap \mathcal{U}_T) (k)$ and let $L'_0 \cap X = \{x = c_0, c_1, \ldots , c_d\}$. Since every $k$-point of $\mathcal{U}_{m,n} ^{S}$ satisfies the condition (2) of Lemma~\ref{lem:Non-collinear}, we know that no three points of $L'_0 \cap X$ are collinear. Thus $\{ c_0, c_{m+1}, c_{n+1} \}$ are not collinear so that when $\ell={\rm Sec} (\{c_0\}, \{c_{m+1}\})$ is the line joining $c_0$ and $c_{m+1}$, it does not pass through $c_{n+1}$.

We let $P= {\rm Sec}(\{c_{n+1}\}, \ell)$. The subspace $\Gr_{\ell} (N-1, \P_k ^N)$ is of dimension $N-2$ and $\Gr_{P} (N-1, \P_k ^N)$ is a closed subspace of $\Gr_{\ell} (N-1, \P_k ^N)$ of dimension $N-3$ (see Lemma \ref{lem:AD-open-refined-elem}). Because we may assume $N \geq 3$, there is a one-parameter family (actually isomorphic to $\P_k ^1$) $\mathcal{B}$ in $\Gr_x (N-1, \P_k ^N)$ such that (i) $\{L_0'\} \in \mathcal{B}$, (ii) every member of the family $\mathcal{B}$ passes through both of $\{ c_0, c_{m+1} \}$ and (iii) a general member does not pass through $c_{n+1}$. Since $\mathcal{U}_{m,n} ^{S} \cap \mathcal{U}_T$ is dense open in $\Gr_x (N-1, \mathbb{P}_k ^N)$ and $L'_0 \in \mathcal{U}_{m,n} ^{S} \cap \mathcal{U}_T \cap \Gr_{L} (N-1, \P_k ^N)$, the latter is dense open in $\Gr_{\ell} (N-1, \P_k ^N)$. Hence, a general member of $\mathcal{B}$ is contained in $\mathcal{U}_{m,n} ^{S} \cap \mathcal{U}_T$. 

Let $W \subset \sB \cap \mathcal{U}^{S}_{m,n}  \cap \mathcal{U}_T$ be a smooth affine irreducible (rational) curve passing through $\{L_0 '  \}$. Consider again the quotient map $\pi: X^d \to {\rm Sym} ^d (X) = X^d/\mathfrak{S}_d$, and the finite \'etale map $\pi:V_d \to \pi (V_d)$ for the open set $V_d$ defined previously in \eqref{eqn:Ad-nonempty-0}. Consider the map $W \to \pi (V_d)$ given by $L \mapsto \sum_{i=1} ^d [ y_i]$, where 
$L \cap X = \{ x= y_0, y_1, \ldots, y_d \}$. This yields the Cartesian product 
\begin{equation}\label{eqn:Ad-nonempty-1}
\xymatrix@C1pc{
W' \ar@{^{(}->}[r] \ar[dr] \ar@/^1.2pc/[rr]^-{e} & 
\widetilde{W} \ar[d]^-{\psi} \ar[r] & V_d \ar[d]^{\pi} \\
& W \ar[r] & \pi(V_d)}
\end{equation}
so that $\psi$ is finite and {\'e}tale. Note also that the members of $\widetilde{W}$ can be represented by $\un{L} = (L, y_1, \ldots, y_d ) \in W \times V_d$ such that $L\cap X = \{ x = y_0, y_1, \ldots, y_d \}$. We let $W' \subset \widetilde{W}$ be the component containing the point $(L'_0, c_1, \ldots, c_d)$. For the `bad' closed subsets $D_n, G^n_{m+1} \subset V_d$ of Lemma \ref{lem:DGclosed}, we have:

\vskip .2cm

\textbf{Claim:} \emph{$\mathcal{Y}:=e^{-1} (D_n \cup G^n_{m+1})$ is a proper 
closed subset of $W'$.}

$(\because)$ That this is a closed subset of $W'$ follows by Lemma \ref{lem:DGclosed}. We need to show that this is a proper subset. Note that $D_n = D_{n,1} \cup D_{n,2}$ and $G^n _{m+1} = \bigcup_{i=0} ^{m+1} G^n_{m+1, i}$. We analyze each piece of them in what follows.

\emph{Case 1:} We first show that $e^{-1} (D_{n,2}) = \emptyset$ and $e^{-1} (G^n _{m+1, 0}) = \emptyset$.

Note that we had $W \subset \mathcal{B} \cap \mathcal{U}_{m,n} ^{S} \cap \mathcal{U}_T$, where $\mathcal{U}_T$ is as in Lemma \ref{lem:AD-open-refined-r=1}. Here, the condition (1) of Lemma \ref{lem:AD-open-refined-r=1} (and $S$ replaced by $T$) reads as `$L \cap ( (X \setminus X_{\rm fs}) \cup T) = \emptyset$' for each $L \in \mathcal{U}_T (k)$. So, for every $L \in W(k)$, this is disjoint from $T = S \cup (\widehat{f} (\widehat{g}^{-1} (\widehat{g} (Z_x))) \setminus \{ x \})$. Hence, if $e^{-1} (D_{n,2}) \not = \emptyset$, then it gives an element $L \in W(k)$ such that $L \cap X = \{ x = y_0, y_1, \ldots, y_d\}$ satisfies $\widehat{g} (Z_{y_0}) \cap \widehat{g}( Z_{y_i}) \not = \emptyset$ for some $1 \leq i \leq n$, so that $L$ intersects with a point of $T$, contradicting the above choice of $W$. Hence $e^{-1} (D_{n,2}) = \emptyset$. An identical argument shows that $e^{-1}(G^n_{m+1,0}) = \emptyset$.

\emph{Case 2:} We now show that $e^{-1} (D_{n,1})$ and $e^{-1} (G^n_{m+1, i})$ for $1 \leq i \leq m$ are finite. 

To do so, it is enough to show that these closed subsets are proper in $W'$, as $W'$ is an irreducible curve. Suppose $e^{-1} (D_{n,1}) = W'$. In particular $\un{L}_0': = (L_0', c_1, \ldots, c_d) \in e^{-1} (D_{n,1})$, so that $(c_1, \ldots, c_d ) \in D_{n,1}$, so $\widehat{g}(Z_{c_i}) \cap \widehat{g}(Z_{c_j}) \not = \emptyset$ for some $1 \leq i \not = j \leq n$. But, this contradicts that $L_0 ' \in \mathcal{U}_{m,n} ^S (k)$. Hence, $e^{-1} (D_{n,1})$ is proper closed in $W'$. By the same argument, we have $|e^{-1} (G^n_{m+1, i}) |< \infty$. 

\emph{Case 3:} It remains to show that $|e^{-1} (G^n _{m+1, m+1})|< \infty$. 

To do so, we will make use of our choice of $W$ that $W \subset \mathcal{B}$. Recall that $\mathcal{B} \subset \Gr_x (N-1, \mathbb{P}_k ^N)$ is a one-parameter family containing $\{ L_0 '\}$ such that every member of $\mathcal{B}$ passes through $\{ c_0, c_{m+1} \}$, while a general member does not pass through 
$c_{n+1}$. 

Consider the composite $q: W' \overset{e}{\to} V_d \to X^2$, where the last arrow takes $(y_1, \ldots, y_d)$ to $(y_{m+1}, y_{n+1}) \in X^2$. 
Since every $L \in W (k) \subset \mathcal{B}(k) $ contains $c_{m+1}$ by construction, the composition of $q$ with the first projection $X^2 \to X$, taking $(y_{m+1}, y_{n+1})$ to $y_{m+1}$, is the constant map that takes all of $W'$ to $c_{m+1} \in X$. On the other hand, the general member $L \in W(k)$ does not contain $c_{n+1}$. This implies that the composite of $q$ with the second projection $X^2 \to X$, taking $(y_{m+1}, y_{n+1})$ to $y_{n+1}$, is non-constant.
Hence, the map $q$ is non-constant and the image $q(W')$ in $X^2$ is an irreducible curve contained in $\{ c_{m+1} \} \times X \cong X$ (recall that $k$ is assumed to be algebraically closed.)

Write it as $W' \overset{u}{\to} q (W') \overset{v}{\to} X$, where $u$ is induced by $q$ and $v$ is the projection to the coordinate $y_{n+1}$.
Since both $u$ and $v$ are non-constant morphisms of irreducible curves, they are dominant and quasi-finite.
In particular, the composite $v \circ u$ is quasi-finite. Note that by definition, $e^{-1} (G^n_{m+1, m+1}) \subset \{ (L, y_1, \ldots, y_d) \in W' | \ y_{n+1} \in S_1 \} = (v \circ u)^{-1} (S_1)$, where $S_1 := f( \widehat{g}^{-1} (\widehat{g} (Z_{c_{m+1}})))$. 
Since $\widehat{f}$ and $\widehat{g}$ are finite by Lemma \ref{lem:Ad-open-finite}, the set $S_1$ is finite, thus $(v \circ u)^{-1} (S_1)$ is a finite set.
 Hence, we have $|e^{-1} (G^n _{m+1, m+1}) |< \infty$, being a subset of a finite set. This finishes the proof of Claim.
\bigskip

Back to the proof of Step 3, since the set $\mathcal{Y}$ of Claim is finite, the subset $W' \setminus \mathcal{Y} \subset W'$ is nonempty open. Since $\psi$ is an open map and $W' \subset \widetilde{W}$ is open subset such that $W' \to W$ is surjective, it follows that $\psi (W' \setminus \mathcal{Y}) \subset W$ is a nonempty (thus dense) open subset. By construction, $\psi (W' \setminus \mathcal{Y}) \subset \mathcal{U}_{m+1, n}^{S}$. In particular, we get $ \mathcal{U}_{m+1, n}^{S} \not = \emptyset$. This proves Step 3.

Back to the proof of the lemma, by inductively applying the above three steps, we deduce that each $\mathcal{U}_{m,n} ^{S}$ is a dense open subset of $\Gr_{x} (N-1, \mathbb{P}_k ^N)$.
\end{proof}

Now we allow $r \geq 1$. We can strengthen Lemma \ref{lem:AD-open-refined-0} as follows.

\begin{prop}\label{prop:Ad-open-final-I}
We follow the notations and the assumptions of Lemma \ref{lem:AD-open-refined-0}. Let $r \geq 1$. After replacing $\mathbb{P}_k ^N$ by a bigger projective space via Veronese if necessary, we have the following property: given any hyperplane $H_0 \subset \mathbb{P}^N_k$ disjoint from $S \cup \{ x \}$ and a general 
$L_0 \in \Gr ^{\tr} _{S \cup \{x\}}(H_0, N-r+1, \mathbb{P}_k ^N)(k)$, there exists a dense open subset $\mathcal{U}^S_x \subset \Gr_x ^{\tr} (L_0, N-1, \mathbb{P}_k ^N)$ such that each $L \in \mathcal{U}^S_x(k)$ satisfies the properties $(1) \sim 
(6)$ of Lemma \ref{lem:AD-open-refined-0} as well as the additional property $(I) : \widehat{g} (Z_{c_i}) \cap \widehat{g} (Z_{c_j}) = \emptyset$ for each pair $0 \leq i \not =j \leq d$. 
\end{prop}

\begin{proof}
  The $r=1$ case of the proposition follows from Lemma \ref{lem:Ad-open-final-I} with $(m,n) = (d-1, d-1)$. 
  So we assume $r \ge 2$. 

We use an argument of reduction to the $r=1$ case as we did in Lemma \ref{lem:AD-open-refined-0}. Using the notations there, choose a reimbedding $\eta: X \inj \P_k^N$, a general $L_0 \in \Gr(N-r+1, \P_k^N)(k)$ and $C = L_0 \cap X$ as in Lemma \ref{lem:AD-open-refined-0}. Let $S' := (C \setminus \{x\})\cap (\widehat{f}(Z_{\rm sing}) \cup (X \setminus X_{\fs} ) \cup S)$ and $W = Z|_{C \times \widehat{B}}$. Applying the `$r=1$' case of the proposition (proven in Lemma \ref{lem:Ad-open-final-I}) to $C$, $S'$ and $W$, with the identification $L_0 \simeq \mathbb{P}_k ^{N-r+1}$, there is a dense open subset $\mathcal{U}' \subset \Gr_x (N-r, L_0)$ that satisfies the properties of Proposition \ref{prop:Ad-open-final-I} for $r=1$ case. (In terms of the notations of Lemma \ref{lem:Ad-open-final-I}, we have $\mathcal{U}' =\mathcal{U}_{d-1, d-1} ^{S'}$.)
Note that Lemma \ref{lem:Ad-open-final-I} is applicable to $C$
by property (6) of Lemma~\ref{lem:AD-open-refined-0}.

Recall now that we had a smooth surjective morphism of varieties
$\theta_{L_0}: \Gr_x ^{\rm tr} (L_0, N-1, \P_k^N) \to \Gr_x(N-r, L_0)$ from \eqref{eqn:refined-adm-0-1}. So, the inverse image $\mathcal{U}^S_x:=\theta_{L_0} ^{-1} (\mathcal{U}')$ is a dense open subset of $\Gr_x^{\rm tr}(L_0, N-1, \P_k^N)$. We claim that this $\mathcal{U}^S_x$ fulfills the requirements of the proposition for $r \geq 2$ case.

Indeed, since $W = Z|_{C \times \widehat{B}}$, we see that $Z_y = W_y$ and hence $\widehat{g}(Z_y) = \widehat{g}(W_y)$ for any closed point $y \in C$. Hence, for $L \in \Gr_x ^{\tr} (L_0, N-1, \mathbb{P}_k ^N)(k)$ with $\theta_{L_0} (L) \cap X = (L \cap L_0) \cap C = \{ x = c_0, c_1, \ldots, c_d \}$, the condition $(I)$ is satisfied if and only if the condition $(I)$ is satisfied for $\theta_{L_0} (L)$ with $X$ replaced by the curve $C$. This means $L \in \mathcal{U}^S_x (k)$ satisfies the proposition, as desired.
\end{proof}

\subsection{Separating residual fibers of $Z$ along $\widehat{B}$: the semi-local case}\label{sec:B-sep*}
Note that in the statement of Proposition \ref{prop:Ad-open-final-I}, the dense open subset that we found depends on the choice of a single regular closed point $x \in X$. We want to extend it to a finite subset $\Sigma$ of regular points. This issue will be completely resolved in Proposition \ref{prop:Main-I}  by using the `cone admissibility' condition, which we develop as the property (3) of the following Proposition \ref{prop:Semi-local-open-1}. 
One further aspect on \'etaleness is studied in \S \ref{sec:Et-base}.

Recall that when $M \subset \mathbb{P}_k ^N$ is a linear subspace and $x \in \mathbb{P}_k ^N$ is a closed point, after the base change $\Spec (k(x)) \to \Spec (k)$, the cone $C_x (M)= {\rm Sec} (\{x\}, M)$ is the smallest linear subspace containing both $x$ and $M$. When $x \not \in M$, we have $\dim (C_x (M)) = \dim (M) + 1$. In this article, we need to use the cones only when $k$ is algebraically closed, so that no confusion will arise.

\begin{prop}\label{prop:Semi-local-open-1}
Let $k$ be an algebraically closed field. We are under the Set-up $+ (\fs)$ of \S \ref{sec:Set-up2}. After replacing the embedding $X \hookrightarrow \mathbb{P}_k ^N$ by a bigger one via a Veronese embedding if necessary, we have the following: for the given hyperplane $H \subset \mathbb{P}^N_k$ disjoint from $\Sigma$ and a general $L_0 \in \Gr ^{\tr} (H, N-r+1, \mathbb{P}_k ^N)(k)$, there exists a dense open subset $\mathcal{W} \subset \Gr (N-2, H)$ such that each $M \in \mathcal{W}(k)$ satisfies the following properties:
\begin{enumerate}
\item $M$ intersects $L_0$ transversely.
\item $M \cap L_0 \cap X = \emptyset$. 
\item For each $x \in \Sigma$, the cone $C_{x} (M)$ lies in $\mathcal{U}^{\Sigma \setminus \{ x \}} _{x} (k)$ for the open subset $\mathcal{U}^{\Sigma \setminus \{ x\}} _{x}\subset \Gr_{x} ^{\tr} (L_0, N-1, \mathbb{P}_k ^N)$ of Proposition \ref{prop:Ad-open-final-I}.
\end{enumerate}
\end{prop}

\begin{proof}
Note that if $\Sigma = \{x_1, \ldots ,x_n\}$, then the condition (3) consists of the conditions $(3)_i : \ C_{x_i} (M)\in\mathcal{U}^{ \Sigma \setminus \{ x_i\}} _{x_i} (k)$ for $1 \leq i \leq n$. Suppose we proved the existence of a dense open subset $\mathcal{W}_i \subset \Gr (N-2, H)$ for which each member $M \in \mathcal{W}_i (k)$ satisfies the conditions (1), (2), and $(3)_i$ for each $1 \leq i \leq n$. Then we can take $\mathcal{W}:= \bigcap_{i=1} ^n \mathcal{W}_i$, which is again a dense open subset of $\Gr (N-2, H)$. Hence, it is enough to prove the existence of those $\mathcal{W}_i$. Without loss of generality, we may assume $i=1$. For notational simplicity, we let $x:= x_1$ and $T:= \Sigma \setminus \{ x_1\}$. We note also that when $r=1$, we have $\Gr^{\tr} (H, N-r+1, \mathbb{P}_k ^N) = \Gr (N, \mathbb{P}_k ^N) = \{ \mathbb{P}_k ^N \}$ so that the choice of $L_0$ plays no role. We prove the proposition for the cases of  $r=1$ and $r \geq 2$ separately.

\textbf{Step 1.} Suppose $r=1$. Consider the affine morphism of schemes
\begin{equation}\label{eqn:vartheta0}
\vartheta_x: \Gr(x, N-2, \P_k^N) \to \Gr_x(N-1, \P_k^N), \ \  L \mapsto C_x (L).
\end{equation}

This is a smooth surjective morphism, and defines a vector bundle of rank $N-1$. For the closed irreducible subscheme $\Gr(N-2, H) \inj \Gr(x, N-2, \P_k^N)$, the restriction $\vartheta_{x, H}: \Gr(N-2, H) \to \Gr_x(N-1, \P_k^N)$ of $\vartheta_x$, is an isomorphism.  

Let $\mathcal{U}_x ^T \subset \Gr_x (N-1, \mathbb{P}_k ^N)$ be the dense open subset of Proposition \ref{prop:Ad-open-final-I}, applied to $x$, $T$ and $H_0 = H$ for $r=1$. Then $\vartheta^{-1}_{x, H}(\mathcal{U}^{T}_{x})$ is a dense open subset of $\Gr(N-2, H)$. Since $\Gr(X, N-2, H)$ is its dense open subset by Lemma \ref{lem:fs-open-1}, so is the intersection $\mathcal{W}_1:= \vartheta^{-1}_{x, H}(\mathcal{U}^{T}_{x}) \cap \Gr(X, N-2, H)$ in $\Gr (N-2, H)$. One checks that this satisfies the required conditions (1), (2), and $(3)_1$, proving the proposition for $r=1$.

\textbf{Step 2.} Suppose now that $r \geq 2$. 
As we did previously in Lemma \ref{lem:AD-open-refined-0}  with $H_0= H$ via a Bertini argument of \cite{KA}, we choose a reimbedding $\eta: X \inj \P_k^N$, a general $L_0 \in \Gr^{\rm tr}(H, N-r+1, \P_k^N) (k) $, a curve $C = L_0 \cap X$, and $Z|_{C \times \widehat{B}}$. Consider again the map in \eqref{eqn:vartheta0}. When $L_0$ contains $x$, this $\vartheta_x$ induces a smooth surjective map $\vartheta_x ^{L_0}: \Gr^{\rm tr}(L_0, x, N-2, \P_k^N) \to \Gr^{\rm tr}_x(L_0, N-1, \P_k^N)$, where we recall that $\Gr^{\rm tr}(L_0, x, n, \P_k^N): = \Gr^{\rm tr}(L_0, n, \P_k^N) \cap \Gr(x, n, \P_k^N)$. This restricts to give $\vartheta_{x, H}: \Gr^{\rm tr}(L_0, N-2, H) \to \Gr^{\rm tr}_x(L_0, N-1, \P_k^N).$ One checks that this map is an inclusion whose image is the dense open subset $\Gr^{\rm tr}_x(L_0 \cap H, N-1, \P_k^N)$. As $H \cap \{x\} = \emptyset$, we see that $\Gr^{\rm tr}_x(L_0 \cap H, N-1, \P_k^N)$ coincides with $\Gr^{\rm tr}_x(L_0, N-1, \P_k^N)$. This implies that $\vartheta_{x, H}$ is an isomorphism. 

Let $\mathcal{U}_x ^{T} \subset \Gr^{\rm tr}_x(L_0, N-1, \P_k^N)$ be the dense open subset of Proposition \ref{prop:Ad-open-final-I} applied to $x$, $T$, and $H_0 = H$ for $r \geq 2$. Since $\vartheta_{x, H}$ is an isomorphism, $\vartheta^{-1}_{x,H}(\mathcal{U}_x ^{T})$ is dense open in $\Gr^{\rm tr}(L_0, N-2, H)$, thus dense open in $\Gr(N-2, H)$. Combining this with Lemma \ref{lem:fs-open-1}, we conclude that $\mathcal{W}_1:=  \vartheta^{-1}_{x,H}(\mathcal{U}_x ^{T}) \cap \Gr(C, N-2, H)$ is dense open in $\Gr(N-2, H)$. One checks that each $M \in \mathcal{W}(k)$ satisfies the required conditions (1), (2), and $(3)_1$. This finishes the proof.
\end{proof}

\subsection{{\'E}taleness of linear projections at $L^+(\Sigma)$}\label{sec:Et-base}
Recall that we had obtained a linear projection $\phi_L: X \to \mathbb{P}_k ^r$ that is \'etale at each point of $\Sigma$ in the condition (1) of Lemma \ref{lem:fs-proj-1et}. Unfortunately, this is not quite enough for us. We need to have $L$ such that $\phi_L$ is \'etale at each point of $L^+ (\Sigma)$ as well. 
We show that we can achieve this as a geometric consequence of the condition (3) of Proposition \ref{prop:Semi-local-open-1}. 

Part of the requirement of 
Proposition \ref{prop:Semi-local-open-1} that $C_x (M)$ lies in $\mathcal{U}_x ^{\Sigma \setminus \{ x \}}(k)$ for the open set $ \mathcal{U}_x ^{\Sigma \setminus \{ x \}}$ is that $C_x (M)$ intersects $X_{\fs} \subset X_{\rm sm}$ transversely. This comes from the condition (2) of Lemma \ref{lem:AD-open-refined-0}. Here is its geometric meaning: % for a more general field:

\begin{lem}\label{lem:proj-etale}
Let $k$ be an algebraically closed field
and let $L \in \Gr (X, N-r-1, \mathbb{P}_k ^N) (k) $. Let $\mathbb{P}^r_k$ be a linear subspace of $\mathbb{P}_k ^N$ such that $L \cap \mathbb{P}_k ^r = \emptyset$. Let $y \in \P^r_k$ be a closed point such that $C_y (L) \cap X_{\rm sing} = \emptyset$. Then $C_y (L)$ intersects $X$ transversely if and only if the linear projection $\phi_L: X \to \P^r_k$ away from $L$ is finite and {\'e}tale over an affine neighborhood of $y$ in $\P^r_k$.
\end{lem}

\begin{proof}
$(\Rightarrow)$ Suppose that $C_y (L)$ intersects $X$ transversely and let $E:= C_y (L) \cap X$ be this scheme-theoretic intersection. Since $k$ is perfect while $C_y(L)$ and $X_{\rm sm}$ have the complementary dimensions $N-r$ and $r$ in $\mathbb{P}_k ^N$, respectively, the transverse intersection is equivalent to saying that $E$ is smooth, $|E|< \infty$, and each point of $E$ is a simple regular point of $X_{\rm sm}$. Because we are given that $X \cap L = \emptyset$ and $L \subset C_y (L)$, we see that $C_y (L) \cap X = (C_y (L) \setminus L) \cap X$, which is precisely the scheme-theoretic fiber $\phi_L ^{-1} (y)$ over $y \in \mathbb{P}_k ^r$. 

Since $C_y (L) \cap X_{\rm sing} = \emptyset$, we see that $\phi_L ^{-1} (y) \cap X_{\rm sing} = \emptyset$. Since $\phi_L$ is finite, $\phi_L(X_{\rm sing})$ is a closed subscheme of $\P^r_k$ not meeting $y$. Hence, there is an affine open $U \subset \P^r_k$ containing $y$ such that $\phi^{-1}_L(U)$ is regular. We therefore get a Cartesian square
\begin{equation}\label{eqn:proj-etale-0} 
\xymatrix@C1pc{
E \ar[r] \ar[d]_{\phi^y_L} & \phi^{-1}_L(U) \ar[d]^{\phi_L} \\
\Spec(k(y)) \ar[r] & U}
\end{equation} 
such that $\phi^y_L$ is smooth. Since $\phi_L$ is a finite map of regular affine schemes over $k$, it is flat by \cite[Exercise III-10.9, p.276]{Hartshorne} (or \cite[Proposition (6.1.5), p.136]{EGA4-2}). It follows therefore by \cite[Exercise III-10.2, p.275]{Hartshorne} (or \cite[Th\'eor\`eme (12.2.4)-(iii), p.183]{EGA4-3}) that there is an affine neighborhood of $y$ in $U$ over which the restriction of the map $\phi_L$ is smooth, thus finite and {\'e}tale.

$(\Leftarrow)$ If $\phi_L$ is \'etale over a neighborhood of $y$, then its base change to $\Spec (k(y))$, i.e., the map $\phi_L ^y : E = C_y (L) \cap X \to \Spec (k(y))$ from the scheme-theoretic intersection is \'etale. Since $k=k(y)$, this means $E$ is smooth over $k$ so that the intersection is transverse. 
\end{proof}

\begin{cor}\label{cor:residual etale}
Let $k$ be an algebraically closed field.
Let $L \in \Gr (X, N-r-1, \mathbb{P}_k ^N)(k)$ and realize the linear projection $\phi_L: X \to \mathbb{P}_k ^r$ for a linear subspace $\mathbb{P}_k ^r \subset \P_k ^N$ such that $L \cap \P_k ^r = \emptyset$. Suppose that for each $x \in \Sigma$, we have $C_x (L) \cap X_{\rm sing} = \emptyset$ and $C_x (L)$ intersects $X_{\rm sm}$ transversely. Then there is an affine open neighborhood $U \subset \mathbb{A}_k ^r$ of $\phi_L (\Sigma)$ such that $\phi_L: \phi_L ^{-1} (U) \to U$ is finite and \'etale. In particular, $\phi_L: X \to \mathbb{P}_k ^r$ is \'etale at every point of $\phi_L ^{-1} (\phi_L (\Sigma))$.
\end{cor}

\begin{proof}
Note that for each $x \in \Sigma$, we have $C_x (L) = C_{\phi_L (x)} (L)$ since $\phi_L$ is given with a chosen internal linear subspace $\mathbb{P}_k ^r \subset \mathbb{P}_k ^N$. Since $C_{\phi_L (x)} (L) \cap X_{\rm sing} = \emptyset$ and $C_{\phi_L (x)} (L)$ intersects $X$ transversely, Lemma \ref{lem:proj-etale} says that there is an affine open neighborhood $U_x \subset \mathbb{P}_k ^r$ of $\phi_L (x)$ such that $\phi_L: \phi_L ^{-1} (U_x) \to U_x$ is finite and \'etale. Hence, for $U:= \bigcup_{x \in \Sigma} U_x$, the restriction $\phi_L: \phi_L ^{-1} (U) \to U$ is finite and \'etale. By Lemma \ref{lem:FA}, we may shrink this $U$ into an affine open neighborhood of $\phi_L (\Sigma)$. This implies the corollary.
\end{proof}

\section{Regularity of residual cycles over finite closed points}\label{sec:Pres-lem}
Our goal in \S \ref{sec:Pres-lem} is to study the regularity of the residual cycles using the technique of vertical separation of residual fibers studied in \S \ref{sec:pre-Ad-set}. We continue to work with the Set-up $+ (\fs)$ of \S \ref{sec:Set-up2}. In particular, for each irreducible component $Z_i$ of $Z$, the projection $Z_i \to \widehat{B}$ is non-constant and the projection $Z_i \to X$ is fs over $X_{\rm fs}$.

\subsection{Admissible sets} 

The property $(I)$ in Proposition \ref{prop:Ad-open-final-I} encourages the following definition, that encodes a set of data needed to achieve the remaining properties of residual cycles.  

\begin{defn}\label{defn:Ad-set}
Let $k$ be an infinite perfect field. Let $x \in X_{\fs}$ be a closed point. A finite subset $D \subset X_{\fs} $ of distinct closed points is called \emph{$(Z, x)$-admissible} if $(1)$ $x \in D$, $(2)$ $Z$ is regular at all points lying over $D \setminus \{ x \}$, and $(3)$ $\widehat{g}(Z_{x_1}) \cap \widehat{g}(Z_{x_2}) = \emptyset$ for each distinct pair $x_1\not = x_2$ in $D$.
\end{defn}

The following application of Proposition \ref{prop:Semi-local-open-1} will be a basis for our proof of the regularity of the residual cycles along $\Sigma$. We study it for $k=\bar{k}$ case, but it will soon be generalized gradually.

\begin{prop}\label{prop:Main-I}
Let $k$ be an algebraically closed field. We are under the Set-up $+ (\fs)$ of \S \ref{sec:Set-up2}. Let $Y \subset X$ be a closed subset of dimension at most $r-1$. After replacing the embedding $\eta: X \hookrightarrow \mathbb{P}_k ^N$ by a bigger one via a Veronese embedding if necessary, we have the following: for the given hyperplane $H \subset \mathbb{P}^N_k$ disjoint from $\Sigma$, there is a dense open subset $\mathcal{U} \subset \Gr (N-r-1, H)$ such that for each $L \in \mathcal{U}(k)$, we have $L \cap X = \emptyset$ so that there is a finite and surjective linear projection $\phi_L: X \to \P^r_k$. Furthermore, it satisfies the following properties:
\begin{enumerate}
\item The map $\phi_L: \phi^{-1}_L (U) \to U$ is \'etale for some affine open $U \subset \P^r_k$ containing $\phi_L(\Sigma)$.
\item $\phi_L(x) \neq \phi_L(x')$ for each pair $x \not = x' \in \Sigma$.
\item $k(\phi_L(x)) \xrightarrow{\simeq} k(x)$ for each $x \in \Sigma$. 
%\item $L^{+}(\Sigma) \cap Y = \emptyset$.
\item $L^+ (x) \cap Y = \emptyset$ for each $x \in \Sigma$.
\item $\phi^{-1}_L\left(\phi_L(x)\right)$ is $(Z, x)$-admissible for each $x \in \Sigma$.
\end{enumerate}
\end{prop}

\begin{proof}
As in the proof of Proposition \ref{prop:Ad-open-final-I}, we can choose a reimbedding $\eta: X \inj \P^N_k$ and a dense open subset $\mathcal{U}_1 \subset \Gr_{\Sigma}(N-r+1, \P^N_k)$ such that each $L' \in \mathcal{U}_1(k)$ satisfies the condition that $L' \cap X$ is a reduced curve none of whose components is contained in $Y$, is regular away from $X_{\rm sing}$, and for each component of $Z|_{X \times \widehat{B}}$, the projection to $\widehat{B}$ is non-constant. Since $H \cap \Sigma = \emptyset$, we see that $\mathcal{U}_0:= \Gr_{\Sigma}(N-r+1, \P^N_k) \cap \Gr^{\rm tr}(H, N-r+1, \P^N_k) \neq \emptyset$. It follows that this intersection is dense open in $\Gr_{\Sigma}(N-r+1, \P^N_k)$. Letting $\mathcal{U}'_1 := \mathcal{U}_0 \cap \mathcal{U}_1$, we see that $\mathcal{U}'_1 $, is a dense open subset of $\Gr_{\Sigma}(N-r+1, \P^N_k)$ such that each $L' \in \mathcal{U}'_1(k)$ intersects $H$ transversely and $L' \cap X$ is a curve of the above type.

Choose $L_0 \in \mathcal{U}'_1(k)$. (N.B. Note that when $r=1$, there is a unique choice $L_0 = \mathbb{P}_k ^N$ automatically, and we have $C = X$.) We now apply Proposition \ref{prop:Semi-local-open-1}. It follows that there exists a dense open subset $\mathcal{W} \subset \Gr (N-2, H)$ such that each $M \in \mathcal{W}(k)$ satisfies the conditions (1) $\sim$ (3) of Proposition \ref{prop:Semi-local-open-1}.

On the other hand, the subset $\Gr^{\rm tr}( L_{0}, {\rm Sec}(\Sigma, Y \cap C), N-2, H)\subset \Gr(N-2, H)$ is a dense open subset by Lemmas \ref{lem:fs-open-1} and \ref{lem:fs-open-2}. Hence $\mathcal{V}':= \mathcal{W} \cap \Gr^{\rm tr}(L_{0}, {\rm Sec}(\Sigma, Y \cap C), N-2, H) \subset \Gr (N-2, H)$ is a dense open subset. Since $L_0$ intersects $H$ transversely, the map $\Gr (N-2, H) \to \Gr (N-r-1, H)$, given by $M \mapsto L_0 \cap M$, is smooth and surjective (note that $N \gg r$). In particular, the image $\mathcal{U}_2:= \{L_0 \cap M \in \Gr (N-r-1, H) \ | \ M \in \mathcal{V} \}$ of $\mathcal{V}$ is a dense open subset of $\Gr (N-r-1, H)$. Let $\mathcal{U}'_2 \subset \Gr(X, N-r-1, H)$ be the dense open set of Lemma \ref{lem:fs-proj-1et} so that $\mathcal{U}:= \mathcal{U}_2 \cap \mathcal{U}'_2 \subset \Gr (X, N-r-1, H)$ is a dense open subset. 

\textbf{Claim :} \emph{Each $L \in \mathcal{U}(k)$ satisfies the properties $(1)$ $\sim$ $(5)$ of the proposition.} 

We ignore $L$ from the notation of $\phi_L$ for simplicity. Before we prove the claim, we note that $\phi^{-1}(\phi(\Sigma)) \subset X_{\rm fs}$, as follows from the condition (3) of Proposition \ref{prop:Semi-local-open-1} which includes condition (1) of Lemma \ref{lem:AD-open-refined-0}.

Now, the condition (3) of Proposition \ref{prop:Semi-local-open-1} also implies that, by Corollary \ref{cor:residual etale}, there is an affine neighborhood $U$ of $\phi(\Sigma)$ such that $\phi^{-1}(U) \to U$ is finite \'etale. This proves (1). 

Since our open set $\mathcal{U}$ is contained in the open set of Lemma \ref{lem:fs-proj-1et}, we can use the properties there, too. The condition (2) of Lemma \ref{lem:fs-proj-1et} is that the map $\phi$ is injective on $\Sigma$, proving (2). The condition (3) is obvious because $k$ is assumed to be algebraically closed. The condition (4) follows from our choice of $M$ (thus of $L$) that it avoids the cone involving $Y$. 

We now prove (5). We need to verify the three conditions of the $(Z,x)$-admissibility of Definition \ref{defn:Ad-set} for each $x \in \Sigma$. The condition (1) of Definition \ref{defn:Ad-set} that $x \in \phi^{-1} (\phi (x))$ is obvious. 

We prove the condition (2) of Definition \ref{defn:Ad-set}. The condition (3) of Proposition \ref{prop:Semi-local-open-1} says that the condition (4) of Lemma \ref{lem:AD-open-refined-0} applied to $C_{x} (L) \cap X$ holds. Note that the cone $C_{x} (L)$ plays the role of the linear space in the statement of Lemma \ref{lem:AD-open-refined-0}. That is, each point of $Z$ lying over a point of $(C_{x} (L) \cap X) \setminus \{ x\}$ is a regular point of $Z$. This means that each point of $Z$ lying over a point of $\phi^{-1} (\phi (x)) \setminus \{ x \}$ is regular. This proves the condition (2) of Definition \ref{defn:Ad-set} for $\phi^{-1} (\phi(x))$.

The condition (3) of Definition \ref{defn:Ad-set} for the $(Z, x)$-admissibility of $\phi^{-1} (\phi(x))$ for $x \in \Sigma$ follows from the condition $(I)$ of Proposition \ref{prop:Ad-open-final-I}, which is part of the condition (3) of Proposition \ref{prop:Semi-local-open-1}. This proves (5). We have thus proven the Claim, and hence, the proposition.
\end{proof}

\subsection{Regularity of residual cycles: $k = \bar{k}$ case}\label{sec:conseq}
We now prove regularity of residual cycles at points lying over $\Sigma$ using Proposition \ref{prop:Main-I} when $k$ is algebraically closed. Recall (\S \ref{sec:Set-up}) that for a linear projection $\phi_L: X \to \P^r_k$, the residual scheme $L^+(Z)$ is the closure of $\widehat{\phi}^{-1}_L(\widehat{\phi}_L(Z)) \setminus Z$ in $X \times \widehat{B}$ with the reduced induced closed subscheme structure. 

We let $T: = \widehat{\phi}_L(Z) =\widehat{\phi}_L(L^+(Z)) \subset \P^r_k  \times \widehat{B}$ with the reduced subscheme structure and let $\widetilde{Z}: =T \times_{(\P^r_k  \times \widehat{B})} (X \times \widehat{B})= \widehat{\phi}_L ^{-1} (T) = \widehat{\phi}_L^{-1} (\widehat{\phi}_L (Z))$ as a scheme. We first have:

\begin{lem}\label{lem:cycle miss extra points}
We are under the Set-up $+ (\fs)$ of \S \ref{sec:Set-up2}. Let $x \in X_{\rm fs}$ be a closed point. Suppose in addition that $Z$ is irreducible. 

Let $ \phi_L:  X \to \mathbb{A}^r_k $ be a finite surjective morphism obtained by a linear projection as before such that $\phi^{-1}_L (\phi _L(x))$ satisfies the condition $(3)$ of Definition \ref{defn:Ad-set} of $(Z,x)$-admissibility. Let $\alpha \in Z$ be a point lying over a point of $\phi^{-1}_L (\phi _L(x))$. Let $S= \widehat{\phi}^{-1} _L\widehat{\phi}_L (\alpha)$. 

Then $Z \cap S = \{ \alpha \}$ and the natural map $\mathcal{O}_{Z, Z \cap S} \to \mathcal{O}_{Z, \alpha}$ is an isomorphism of local rings.
\end{lem}

\begin{proof}
Suppose $\alpha \in Z$ lies over $x_1 \in \phi^{-1}_L (\phi _L(x))$. Toward contradiction, suppose there is a point $\alpha' \in Z$ lying over some $x_2 \in \phi^{-1}_L (\phi _L (x)) \setminus \{ x_1 \}$. Since $\widehat{\phi}_L (\alpha) =\widehat{\phi}_L (\alpha')$, we have $\widehat{g} (\alpha) = \widehat{g}(\alpha')$ in $B$, where $\widehat{g}: X \times \widehat{B} \to \widehat{B}$ is the projection. Let $b_0$ be this common closed point. This $Z \to B$ is non-constant and we have $\alpha \in Z_{x_1}$ and $\alpha' \in Z_{x_2}$ so that $\widehat{g} (Z_{x_1}) \cap \widehat{g} (Z_{x_2}) \ni b_0$, contradicting the condition (3) of Definition \ref{defn:Ad-set} for the set $\phi^{-1}_L (\phi_L (x))$.
\end{proof}

\begin{lem}\label{lem:No-other-comp}
Let $k$ be algebraically closed. Let $L \in \mathcal{U}(k) \subset \Gr(N-r-1, H)(k)$ be as in Proposition \ref{prop:Main-I}. Suppose $Z$ is irreducible and let $\alpha = (a,b)\in Z$ be a closed point such that $a \in {\phi}^{-1}_L(\phi_L(\Sigma))$. Assume that $Z$ is irreducible and $\alpha \in Z$. Then $\mathcal{O}_{\widetilde{Z}, \alpha} \to \mathcal{O}_{Z,\alpha}$ is an isomorphism. In particular, $Z$ is the only irreducible component of $\widetilde{Z}$ which passes through $\alpha$, with the multiplicity $1$, and the cycle $[\widetilde{Z}] - [Z]$ has no component equal to $Z$.
\end{lem}

\begin{proof}
We shall write $\phi_L$ simply as $\phi$. Let $y = \phi(a)$ and $\beta = \widehat{\phi}(\alpha) = (\phi(a), b) =(y,b)$. We let $x \in \Sigma$ be such that $y = \phi(x)$ and let $S = \phi^{-1}(y) \times \{b\} = \widehat{\phi}^{-1}(\beta) \subset X \times \widehat{B}$. 

Let $U \subset \P^r_k$ be as in condition (1) of Proposition \ref{prop:Main-I}. Since $\widehat{\phi}$ is finite and \'etale over $U \times \widehat{B}$, it follows that the map $\widetilde{Z} \to T$ is finite and \'etale over $T \cap (U \times \widehat{B})$. In particular, the map of rings $\mathcal{O}_{T, \beta} \to \mathcal{O}_{\widetilde{Z}, S}$ is finite and \'etale. This in turn implies that the map $\mathcal{O}_{T, \beta} \to \mathcal{O}_{Z, Z \cap S}$ is finite and unramified. 

On the other hand, by the condition (5) of Proposition \ref{prop:Main-I} that $\phi_L ^{-1} (\phi_L (x))$ is $(Z, x)$-admissible, we deduce that for each $x \in \Sigma$, the map $\mathcal{O}_{Z, Z \cap S} \to \mathcal{O}_{Z, \alpha}$ is an isomorphism by Lemma \ref{lem:cycle miss extra points}. Hence, the map $\mathcal{O}_{T, \beta} \to \mathcal{O}_{Z, \alpha}$ is an injective (since $Z \surj T$), finite and unramified map of local rings which induces isomorphism between the residue fields (as $k$ is algebraically closed). Lemma \ref{lem:elem-com-alg} therefore says that the map $\mathcal{O}_{T, \beta} \to \mathcal{O}_{Z, \alpha}$ must be an isomorphism.

We next observe that as $\mathcal{O}_{T, \beta} \to \mathcal{O}_{\widetilde{Z}, S}$ is finite and {\'e}tale, the map $\mathcal{O}_{T, \beta} \to  \mathcal{O}_{\widetilde{Z}, \alpha}$ is {\'e}tale. In particular, the map $\widehat{\mathcal{O}}_{T, \beta} \to \widehat{\mathcal{O}}_{\widetilde{Z}, \alpha}$ of completions is finite and {\'e}tale. Since it induces an isomorphism between the residue fields, it follows again from Lemma \ref{lem:elem-com-alg} that $\widehat{\mathcal{O}}_{T, \beta} \to \widehat{\mathcal{O}}_{\widetilde{Z}, \alpha}$ is an isomorphism. Hence, there are local homomorphisms of complete local rings
\begin{equation}\label{eqn:sfs-local-v2-0}
\widehat{\mathcal{O}}_{T, \beta} \to \widehat{\mathcal{O}}_{\widetilde{Z}, \alpha} \surj 
\widehat{\mathcal{O}}_{{Z}, \alpha},
\end{equation}
where both the first map and the composite map are isomorphisms. Thus, the second map is an isomorphism too. The second map in \eqref{eqn:sfs-local-v2-0} being \emph{a priori} a surjection, the Krull intersection theorem (\cite[Theorem 8.10, p.60]{Matsumura}) says that this map is an isomorphism if and only if $\mathcal{O}_{\widetilde{Z}, \alpha} \surj \mathcal{O}_{{Z}, \alpha}$ (without completion) is an isomorphism. This in turn is equivalent to that ${Z}$ is the only irreducible component of $\widetilde{Z}$ passing through $\alpha$, and $Z$ has the multiplicity $1$ in $\widetilde{Z}$. We have thus proven the lemma.
\end{proof} 

\begin{lem}\label{lem:reg-irr}
Let $k$ be algebraically closed and 
$L \in \mathcal{U}(k)\subset \Gr(N-r-1, H)(k)$ as in Proposition \ref{prop:Main-I}. Suppose that $Z$ is irreducible. Then $L^+(Z)$ is regular at all points lying over $\Sigma$.
\end{lem}

\begin{proof}
We continue with the notations of the proof of Lemma \ref{lem:No-other-comp}. Let $\alpha = (x, b) \in X \times \widehat{B}$ with $x \in \Sigma$ be such that $\alpha \in L^+(Z)$. Let $\beta = \widehat{\phi}(\alpha) = (\phi(x), b) := (y, b)$. It follows from Lemma \ref{lem:No-other-comp} that $Z$ does not pass through $\alpha$. This implies that the canonical map $\mathcal{O}_{\widetilde{Z}, \alpha} \to \mathcal{O}_{L^+(Z), \alpha}$ is an isomorphism. Therefore, it suffices therefore to show that $\mathcal{O}_{\widetilde{Z}, \alpha}$ is regular. 

Since $\alpha \in L^+(Z)$, there must exist a closed point $\alpha' =(x', b) \in Z$ with $x' \in \phi^{-1}(y)$. As $\alpha \notin Z$, we must have $x' \neq x$. It follows again from Lemma \ref{lem:No-other-comp} that $\mathcal{O}_{\widetilde{Z}, \alpha'} \xrightarrow{\cong} \mathcal{O}_{{Z}, \alpha'}$. We have also shown in the middle of the proof of Lemma \ref{lem:No-other-comp} that the map $\widehat{\mathcal{O}}_{T, \beta} \to \widehat{\mathcal{O}}_{\widetilde{Z}, (a,b)}$ of completions in \eqref{eqn:sfs-local-v2-0} is an isomorphism for every $a \in \phi^{-1}(y)$. We thus get the commutative diagram of local rings
\begin{equation}\label{eqn:fugret_diag}
\xymatrix@C1pc{
{\mathcal{O}}_{\widetilde{Z}, \alpha} \ar[d] & {\mathcal{O}}_{T, \beta} \ar[l] \ar[r] \ar[d]  & {\mathcal{O}}_{\widetilde{Z}, \alpha'} \ar[d]  \ar[r]^{\cong} & {\mathcal{O}}_{Z, \alpha'} \\
\widehat{\mathcal{O}}_{\widetilde{Z}, \alpha}  & \widehat{\mathcal{O}}_{T, \beta} \ar[l]_-{\cong} \ar[r]^-{\cong}  & \widehat{\mathcal{O}}_{\widetilde{Z}, \alpha'}, &}
\end{equation}
where the vertical arrows are completion maps.

Since $x' \neq x$, it follows from the condition (2) in Definition \ref{defn:Ad-set} and the condition (5) in Proposition \ref{prop:Main-I} that $\mathcal{O}_{{Z}, \alpha'}$ is regular. It follows from  \eqref{eqn:fugret_diag} that all rings of the bottom of the diagram are regular, using a basic fact in commutative algebra that: $(\star)$ a noetherian local ring is regular if and only if its completion is a regular local ring (\emph{cf}. Proof of \cite[Theorem 19.5, p.~157]{Matsumura}). Equivalently, all rings of the top of the diagram are regular by $(\star)$ again. In particular, $\mathcal{O}_{\widetilde{Z}, \alpha}$ is regular. This finishes the proof. 
\end{proof}

To extend Lemma \ref{lem:reg-irr} to reducible subschemes $Z$ in Lemma \ref{lem:reg-gen}, we first consider the following:

\begin{lem}\label{lem:distinct-cycles}
Let $k$ be algebraically closed. We are under the Set-up $+ (\fs)$ of \S \ref{sec:Set-up2}. Here, $Z$ is not necessarily irreducible. Then after replacing the embedding $X \inj \P^N_k$ into a bigger space via a Veronese embedding if necessary, there is a dense open subset $\mathcal{U} \subset \Gr (X, N-r-1, H)$ such that for each $L \in \mathcal{U} (k)$, the induced map $\widehat{\phi}_L$ takes distinct components of $Z$ to distinct components of $\widehat{\phi}_L(Z)$.
\end{lem}

\begin{proof}
As we did previously in Lemma \ref{lem:fs_sansZ-*}, for each $1 \leq i \leq s$, choose a closed point $\alpha_i = (x_i, b_i) \in Z_i \setminus (\cup_{j \neq i} Z_j)$, so that $x_i := \widehat{f} (\alpha_i) \in X_{\rm sm}$ and $b_i:= \widehat{g} (\alpha_i) \in B$. We observe that if $j \neq i$, then $Z_j \not\subset X \times \{b_i\}$, because $Z_j \to \widehat{B}$ is non-constant.

Let $A_i = \bigcup_{j \not = i} \widehat{f}({Z}_j \cap (X \times \{b_i\}))$. This is a closed subset of dimension $\leq r-1$. In particular, $\dim({\rm Sec}(A_i, \{x_i\})) \le r$. Note that $x_i \notin A_i$. Let $\mathcal{U} :=  \Gr (X, N-r-1, H) \cap \bigcap_{i=1} ^s \Gr( {\rm Sec}(A_i, \{x_i\}), N-r-1, H)$. This is dense open in $\Gr(N-r-1, H)$ by Lemma \ref{lem:fs-open-1}.

Suppose now that $\phi_L: {X} \to \P^r_k$ is the projection obtained by any $L \in \mathcal{U}(k)$. We fix an integer $1 \le i \le s$ and let $\beta_i := \widehat{\phi}_L(\alpha_i)$. It is clear that $\beta_i \in \widehat{\phi}_L({Z}_i)$. We claim that $\beta_i \notin \widehat{\phi}_L({Z}_j)$ for $j \neq i$. To see this, note that $\beta_i \in \widehat{\phi}_L({Z}_j)$ if and only if ${Z}_j \cap (L^+(x_i) \times \{b_i\}) \neq \emptyset$. Equivalently, there exists a closed point $x'_j \neq x_i$ such that $\phi_L(x'_j) =\phi_L(x_i)$ and $x'_j \in A_i$. But this implies that $L \cap {\rm Sec}(A_i, \{x_i\}) \neq \emptyset$, which contradicts the choice of $L$. This proves the claim and hence the lemma.
\end{proof}

\begin{lem}\label{lem:reg-gen}
Let $k$ be algebraically closed. We are under the Set-up $+ (\fs)$ of \S \ref{sec:Set-up2}. Here, $Z$ is not necessarily irreducible. Let $\mathcal{U} \subset \Gr (N-r-1, H)$ be the intersection of the dense open subsets of Proposition \ref{prop:Main-I} and Lemma \ref{lem:distinct-cycles}. Then for each $L \in \mathcal{U}(k)$, the residual scheme $L^+(Z)$ is regular at all points lying over $\Sigma$.
\end{lem}

\begin{proof} 
For a choice of $L$, for simplicity write $\phi:= \phi_L$. For $1 \le i \le s$, let $T_i = \widehat{\phi}(Z_i)$ with the reduced closed subscheme structure and let $\widetilde{Z}_i = T_i \times_{(\P^r_k  \times \widehat{B})} (X \times \widehat{B}) = \widehat{\phi}^{-1} (\widehat{\phi} (Z_i))$ as a scheme.

The first claim is that $\widetilde{Z}_i$ and $\widetilde{Z}_j$ share no common component if $i \neq j$. Indeed, if they do share a common component, this would imply that $T_i = T_j$, which contradicts the choice of $L$ as in Lemma \ref{lem:distinct-cycles}. 

Our second claim is that $L^+(Z_i)$ and $L^+(Z_j)$ do not meet at points lying over $\Sigma $ if $i \neq j$. Suppose on the contrary that there is a closed point $\alpha = (x,b) \in L^+(Z_i) \cap L^+(Z_j)$ with $x \in \Sigma$. This implies that there are closed points $\alpha_i = (x_i, b) \in Z_i$ and $\alpha_j = (x_j, b) \in Z_j$ such that $x_i, x_j \in \phi^{-1}(y)$, where $y = \phi(x)$. It follows from Lemma \ref{lem:No-other-comp} that $x_i, x_j \in \phi^{-1}(y) \setminus \{x\}$.

If $x_i = x_j$, then two components $Z_i$ and $Z_j$ of $Z$ meet at $\alpha_i = \alpha_j$ that lies over $x_i = x_j$ in $\phi^{-1}(y) \setminus \{x\}$. In particular, $Z$ is singular at a point lying over $x_i = x_j$ in $\phi^{-1}(y) \setminus \{x\}$, which contradicts the condition (2) of Definition \ref{defn:Ad-set}, which is part of the condition (5) of Proposition \ref{prop:Main-I}. Hence we must have $x_i \neq x_j$. In this case, we get $b \in \widehat{g}(Z_{x_i}) \cap \widehat{g}(Z_{x_j}) \neq \emptyset$ for two distinct points $x_i, x_j \in \phi^{-1}(y)\setminus \{ x \}$. This time, it contradicts the condition (3) of Definition \ref{defn:Ad-set}, which is part of the condition (5) of Proposition \ref{prop:Main-I}. Hence, we proved the second claim. 

It follows from the two claims that $L^+(Z)$ is regular at all points lying over $\Sigma$ if and only if $L^+(Z_i)$ is so for every $1 \le i \le s$. Since we proved the latter holds in Lemma \ref{lem:reg-irr}, we finished the proof of the lemma.
\end{proof}

\subsection{Regularity of residual cycles: general case}\label{sec:conseq-gen}
We can now generalize Proposition \ref{prop:Main-I} to all infinite perfect field as follows. This includes the regularity of the residual cycle along $\Sigma$.

\begin{prop}\label{prop:Main-II}
Let $k$ be any infinite perfect field. We are under the Set-up $+ (\fs)$ of \S \ref{sec:Set-up2}. Let $Y \subset X$ be a closed subset of dimension at most $r-1$. Then after replacing the embedding $\eta: X \hookrightarrow \mathbb{P}_k ^N$ by a bigger one via a Veronese embedding if necessary, we have the following: for the given hyperplane $H \subset \mathbb{P}^N_k$ disjoint from $\Sigma$, there is a dense open subset $\mathcal{U} \subset \Gr (N-r-1, H)$ such that for each $L \in \mathcal{U}(k)$, we have $L \cap X = \emptyset$ so that there is a finite and surjective linear projection $\phi_L: X \to \P^r_k$. Moreover, it satisfies the following properties. 
\begin{enumerate}
\item $\widehat{\phi}_L(Z_i) \neq \widehat{\phi}_L(Z_j)$ if $i \ne j$. 
\item The map $\phi_L: \phi^{-1}_L (U) \to U$ is \'etale for some affine open $U \subset \P^r_k$ containing $\phi_L(\Sigma)$.
\item $\phi_L(x) \neq \phi_L(x')$ for each pair of distinct points $x \not = x' \in \Sigma$.
\item $k(\phi_L(x)) \xrightarrow{\simeq} k(x)$ for each $x \in \Sigma$. 
\item $L^+ (x) \cap  Y = \emptyset$ for each $x \in \Sigma$.
\item $L^+(Z)$ is regular at all points lying over $\Sigma$.
\item
The map $\widehat{\phi}_L: Z  \to \widehat{\phi}_L(Z)$ is birational. 
\end{enumerate}
\end{prop}

\begin{proof}
If $k$ is algebraically closed, the proposition follows from Proposition \ref{prop:Main-I} and Lemmas \ref{lem:distinct-cycles} and \ref{lem:reg-gen}. In general, let $\bar{k}$ be an algebraic closure of $k$ and let $\pi_X: X_{\bar{k}} \to X$ be the projection map from the base change to $\bar{k}$. We have $\Sigma_{\bar{k}}  = \bigcup _{x \in \Sigma} \pi_X ^{-1} (x)$. Choose a sufficiently large closed embedding $\eta: {X} \hookrightarrow \mathbb{P}_k ^N$ so that for the induced embedding $X_{\bar{k}} \hookrightarrow \mathbb{P}_{\bar{k}} ^N$, there exists a dense open subset $\widetilde{\mathcal{U}} \subset \Gr(N-r-1, H_{\bar{k}})$ for which all assertions of Proposition \ref{prop:Main-I} as well as Lemmas \ref{lem:distinct-cycles} and \ref{lem:reg-gen} applied to $X_{\bar{k}}, Z_{\bar{k}}$ and the set $\Sigma_{\bar{k}} \subset X_{\bar{k}}$ hold. (N.B. Under the base change to $\bar{k}$, the irreducible components $Z_i$ of $Z$ may decompose further into irreducible components $Z_{ij}$ of $Z_{i, \bar{k}}$. 
At least $Z_{\bar{k}}$ and $Z_{i, \bar{k}}$ for all $i$ stay reduced because the extension $\bar{k}$ over $k$ is separable.)

Then we can argue via a Galois descent as in the Step 2 of the proof of Lemma \ref{lem:fs-proj-1et} to find a dense open $\mathcal{U}_1 \subset \Gr(N-r-1, H)$ defined over $k$ such that $(\mathcal{U}_1)_{\bar{k}} \subset \widetilde{\mathcal{U}}$. We take $\mathcal{U}:= \mathcal{U}_1 \cap \mathcal{U}_2$, where $\mathcal{U}_2$ is the open set in Lemma \ref{lem:fs-proj-1et} so that we can also use the assertions of Lemma \ref{lem:fs-proj-1et} as well. 

Now, for each $L \in \mathcal{U} (k)$, we have ${X} \cap L = \emptyset$ by our choice of the open set. So, we get a finite linear projection map $\phi_L: X \to \P^r_k$ over $k$. We write this map as $\phi$. The condition (1) is clear now by construction together with Lemma \ref{lem:distinct-cycles}. The conditions (2), (3), (4) hold by the conditions (2), (1), (3) of Lemma \ref{lem:fs-proj-1et}, respectively. The condition (5) follows immediately from the condition (4) of Proposition \ref{prop:Main-I}. 

To prove (6), as we did at the beginning of \S \ref{sec:conseq}, let $T:= \widehat{\phi} (Z) = \widehat{\phi} (L^+ (Z)) \subset \mathbb{P}_k ^r \times \widehat{B}$ with the reduced subscheme structure, and let $\widetilde{Z}:= T \times_{(\mathbb{P}_k ^r \times \widehat{B})} (X \times \widehat{B}) = \widehat{\phi}^{-1} \widehat{\phi} (Z)$ as a scheme. Then we have the commutative diagram
\begin{equation}\label{eqn:Main-II-0}
\xymatrix@C.8pc{
Z_{\bar{k}} \ar@{^{(}->}[r] \ar[d] & \widetilde{Z}_{\bar{k}} 
\ar[r]^-{\widehat{\phi}_{\bar{k}}} \ar[d] &
T_{\bar{k}} \ar[d] \\
Z \ar@{^{(}->}[r] & \widetilde{Z} \ar[r]^-{\widehat{\phi}} & T,}
\end{equation}
where the vertical arrows are the base changes to $\bar{k}$. Note that the map $\widehat{\phi}: Z \to T$ is surjective by definition. As both squares are Cartesian and the vertical maps are smooth, it follows that $L^+(Z_{\bar{k}}) \xrightarrow{\cong} L^+(Z)_{\bar{k}}$. By the choice of our open set $\mathcal{U}$, Lemma \ref{lem:reg-irr} shows that $L^+(Z_{\bar{k}})$ is regular at all points lying over $\Sigma_{\bar{k}}$, i.e., $L^+(Z)_{\bar{k}}$ is regular at all points lying over $\Sigma_{\bar{k}}$. 

 We replace $Z$ by $Z|_V$ and consider the induced Cartesian squares
\begin{equation}\label{eqn:Main-II-1}
\xymatrix@C.8pc{
L^+(Z)_{\bar{k}} \ar[r] \ar[d] & V_{\bar{k}} \ar[r] \ar[d] & \Spec(\bar{k}) \ar[d] \\
L^+(Z) \ar[r] & V \ar[r] & \Spec(k),}
\end{equation}
where the vertical arrows are the base changes to $\bar{k}$. 
Since $L^+(Z)_{\bar{k}}$ is regular and $\bar{k}$ is perfect, the top horizontal composite map is smooth. Hence, by the faithfully flat descent (\cite[Corollaire (17.7.3)-(ii), p.72]{EGA4-2}), the bottom horizontal composite map is smooth. In particular, $L^+(Z)$ is regular. This proves (6). The property (7) is a direct consequence of Lemma~\ref{lem:fs_sansZ}.
\end{proof}

\begin{remk}
The condition (4) of Proposition \ref{prop:Main-I} or the condition (5) of Proposition \ref{prop:Main-II} that $L^+ (x) \cap Y= \emptyset$ for each $x \in \Sigma$ is no longer needed in this version of the article toward the proof of the main theorems. However, we decided to keep them in this article because the property that the residual points of a projection can be made to avoid the given proper closed subscheme $Y$ is nontrivial, and may be useful in an analysis of algebraic cycles in the future.
\end{remk}

\section{The main results}\label{sec:MR}
In this final section, we use various results of the previous sections to prove our main theorems: the presentation lemma and the \sfs-moving lemma. The set-up for the main results is as in \S \ref{sec:Set-up3}. This differs a bit from the Set-up of \S \ref{sec:Set-up} and the Set-up $+ (\fs)$ of \S \ref{sec:Set-up2}.

\subsection{The Set-up ($\star$)}\label{sec:Set-up3}
Let $k$ be an infinite perfect field and $n \ge 1$ an integer. We work under the following setting:

(1) \emph{The box coordinates:} For $0 \le i \le n-1$, let $\widehat{A}_i$ be a smooth projective geometrically integral $k$-scheme of positive dimension and let $A_i \subset \widehat{A}_i$ be a nonempty affine open subset. Let $C_0 = \Spec(k) = \widehat{C}_0$.  For $1 \le j \le n$, we write $C_j =  \prod_{i=0} ^{j-1} A_i$ and $\widehat{C}_j = \prod_{i=0} ^{j-1}  \widehat{A}_i$. Let $\pi_j: \widehat{C}_{n} \to \widehat{C}_j$ be the projection map. We write $B = C_{n}$ and $\widehat{B} = \widehat{C}_{n}$. Let $F := \widehat{B} \setminus B$.

(2) \emph{The base scheme and the cycles:}
Let $X \subset \A^m_k$ be an integral smooth affine closed subscheme of dimension $r \ge 1$ and let $\overline{X} \inj \P^m_k$ be its closure with the reduced subscheme structure. 
Let $\Sigma \subset X$ be a finite set of closed points.

Let $Z \subset X \times B$ be a reduced closed subscheme of pure dimension $r$, and let $\{Z_1, \ldots , Z_s\}$ be all of its irreducible components. Suppose $Z \to X$ is an \fs-morphism, i.e., finite and surjective because $X$ is integral. Let $E \subset \widehat{B}$ be a closed subset containing $F$ such that no irreducible component of $Z$ is contained $X \times E$.

Let $\widehat{Z} \subset \overline{X} \times \widehat{B}$ denote the closure of $Z$ in $\overline{X} \times \widehat{B}$ with the reduced structure. Similarly, $\widehat{Z}_i$ denotes the closure of $Z_i$ in $\overline{X} \times \widehat{B}$. We let $\widehat{f}: \widehat{Z} \to \overline{X}$ and $\widehat{g}: \overline{Z} \to \widehat{B}$ denote the projection maps.

For each $0 \leq j \leq n$, we define $Z^{(j)} = \pi_j(Z):= ({\id}_X \times \pi_j)(Z)$. Because $Z \to X$ is fs, this definition makes sense. Similarly we define $\widehat{Z}^{(j)}$ for $0 \leq j \leq n$.

(3) \emph{The linear projections:} Suppose we are given a Veronese embedding $\P^m_k \inj \P^N_k$ with $N \gg m$. For $L \in \Gr(\overline{X}, N-r-1, H)(k)$, where $H = \P^N_k \setminus \A^N_k$ as in Lemma \ref{lem:Ver-affine-map}, let $\phi_L: \overline{X} \to \P^r_k$ be the linear projection away from $L$ which restricts to a finite map $\phi_L: X \to \A^r_k$. 
If $L$ is fixed in a given context, we often drop it from $\phi_L$ and write $\phi$.

For $0 \le j \le n$, let $\phi_j = \phi \times {\rm id}_{C_j}: X \times C_j \to \A^r_k \times C_j,$ $\widetilde{\phi}_j = \phi \times {\rm id}_{\widehat{C}_j}: X \times \widehat{C}_j \to \A^r_k \times \widehat{C}_j$ and $\widehat{\phi}_j = \phi \times {\rm id}_{\widehat{C}_j}: \overline{X} \times \widehat{C}_j \to\P^r_k \times \widehat{C}_j$ be the induced maps. We let $L^+(Z)$ denote the closure of ${\phi}^{-1}_n({\phi}_n(Z)) \setminus Z$ in $X \times B$ with the reduced structure. We define $L^+(\widehat{Z})$ similarly.

\subsection{The residual cycle}\label{sec:Res-**}

For $L \in \Gr (\overline{X}, N-r-1, H)(k)$ as in the Set-up ($\star$) of \S \ref{sec:Set-up3}, the morphism $\phi=\phi_L: X \to \mathbb{A}_k ^r$ is a finite surjective morphism of affine $k$-schemes so that it is automatically flat by \cite[Exercise III-10.9, p.276]{Hartshorne} (or \cite[Proposition (6.1.5), p.136]{EGA4-2}). Hence, for algebraic cycles on $X \times C_j $, we have the proper push-forward $\phi_{j*}$ and the flat pull-back $\phi_j ^*$ operations. (See \cite[\S 1.4, 1.7]{Fulton}.) For $X \times \widehat{C}_j$, we have similar operations $\widetilde{\phi}_{j*}$ and $\widetilde{\phi}_j ^*$. 

\begin{defn}
If $Z \subset X \times C_j$ is an integral closed subscheme, the \emph{residual cycle} by $\phi=\phi_L$ is defined to be
$$L^* ([Z]) :=\phi_{j}^*\phi_{j*} ([Z] ) - [Z].$$
We extend it $\mathbb{Z}$-linearly to all cycles on $X \times C_j$. Similarly, for cycles on $X \times \widehat{C}_j$, we define the \emph{residual cycle} by
$$L^* ([Z]) :=\widetilde{ \phi}_{j}^* \widetilde{\phi}_{j*} ([Z] ) - [Z].$$

Note that by definition, $L^+ ([Z]) = | L^* ([Z])|$. 
\end{defn}

\begin{lem}\label{lem:fs-res}
We are under the Set-up $(\star)$ of \S \ref{sec:Set-up3}. In particular, $Z \to X$ is an \fs-morphism. Suppose that $Z$ is integral. Then for each $L$ in the Set-up $(\star)$, the morphism $L^+(Z) \to X$ is also fs.
\end{lem}

\begin{proof}
Let $T = \phi_n(Z) \subset \A^r_k \times C_n$. Let $\widetilde{Z} := T \times_{(\A^r_k \times {C}_n)} (X \times {C}_n) = \phi_n ^{-1} (T)$ as a scheme. It suffices to show that the map $\widetilde{Z} \to X$ is fs. Consider the commutative diagram
\begin{equation}\label{eqn:fs-res-0}
\xymatrix@C1pc{
Z \ar@{^{(}->}[r]^{\iota} \ar[dr] & \widetilde{Z} \ar[r]^-{\phi_n} \ar[d]^{\widehat{f}} & T 
\ar[d]^{\widehat{f}'} \\
& X \ar[r]^-{\phi} & \A^r_k,}
\end{equation}
where the vertical arrows are the projection maps and the right square is Cartesian, and $\iota$ is the closed immersion.

Since $Z \to X$ is an \fs-morphism and $\phi$ is an \fs-morphism, the composite $(\phi \circ \widehat{f})|_Z = \phi \circ \widehat{f} \circ \iota$ is an \fs-morphism. By the commutativity, this means $\widehat{f}' \circ \phi_n \circ \iota$ is an \fs-morphism. But since $Z \to T$ is surjective (as $T$ being the image of $Z$ under $\phi_n$ by definition), it follows that $\widehat{f}'$ is finite (e.g., see \cite[Proposition 3.16-(f), p.104]{Liu}). Hence $\widehat{f}'$ is an \fs-morphism. Now, $\phi$ is flat, so by Lemma \ref{lem:fs bc}, the morphism $\widehat{f}$ is an \fs-morphism. 
\end{proof}

\begin{lem}\label{lem:Res-sch-cycle}
We are under the Set-up $(\star)$ of \S \ref{sec:Set-up3}. In particular, $Z \to X$ is an \fs-morphism. Suppose that $Z$ is integral. Suppose that there is an integer $1 \leq j \leq n$ such that the projection map $Z^{(j)} \to C_j$ is non-constant. 

Then for each $L \in \Gr(N-r-1, H)(k)$ satisfying Lemma \ref{lem:fs_sansZ} and Proposition \ref{prop:Main-II} for all $Z^{(i)}$ over $j \le i \le n$, we have the equalities $[L^+(Z^{(j)})] = L^*([Z^{(j)}])$ and $\pi_{j *}(L^*([Z])) = m_jL^*([Z^{(j)}])$, where $m_j = [k(Z): k(Z^{(j)})]$.
\end{lem}

\begin{proof}
First of all, note that by Lemma \ref{lem:fs-res}, every component of $L^+ (Z)$ is fs over $X$. In particular, by the finiteness criterion Lemma \ref{lem:finiteness}, each irreducible component of $L^+ (Z)$ is closed in $X \times \widehat{C}_n$. The push-forward $\pi_{j *}([L^*(Z)])$ is given by the projective map $\pi_j: X \times \widehat{C}_n \to X \times \widehat{C}_j$ is projective.

To prove the first equality, replacing $Z$ by $Z^{(j)}$, we may assume $n = j$ and $Z^{(j)} = Z$. Let $T = \widetilde{\phi}_n(Z) \subset \A^r_k \times \widehat{C}_n$. 

Note that the map $Z \to T$ is birational by Lemma \ref{lem:fs_sansZ}. Hence, by the definition of the proper push-forward and flat pull-back of cycles, the first equality is equivalent to showing that $\widetilde{Z}:=T \times_{(\A^r_k \times \widehat{C}_n)} (X \times \widehat{C}_n) = \widetilde{\phi}_n ^{-1} (T)$ is a reduced scheme.

To show that $\widetilde{Z}$ is reduced, let $U \subset \A^r_k$ be an affine open neighborhood of $\Sigma$ as in the condition (3) of Proposition \ref{prop:Main-II}. Since $Z \to X$ is finite and surjective, the open subset $T \cap (U \times \widehat{C}_n)$ is dense in $T$. 
The map $\widetilde{\phi}_n$ is \'etale over this dense open subset of $T$. 
Hence, $\widetilde{Z}  = \widetilde{\phi}_n ^{-1} (T)$ is reduced over this dense open subset of $T$. However, $\widetilde{\phi}_n ^{-1} (T) \to T$ is finite and flat everywhere, it means $\widetilde{Z} = \widetilde{\phi}_n ^{-1} (T)$ is reduced. This proves the first equality.

For the second equality, consider the commutative diagram
\begin{equation}\label{eqn:Res-sch-cycle-0}
\xymatrix@C.8pc{
X \times \widehat{C}_n \ar[r]^-{\widetilde{\phi}_n} \ar[d]_{\pi_j} & 
\A^r_k \times \widehat{C}_n \ar[d]^{\pi_j} \\
X \times \widehat{C}_j \ar[r]^-{\widetilde{\phi}_j} & 
\A^r_k \times \widehat{C}_j.}
\end{equation}

This is a Cartesian square in which the vertical arrows are projective and the horizontal arrows are finite and flat. Hence, by \cite[Proposition 1.7]{Fulton}, we have
\[
\begin{array}{lllll}
\pi_{j *}(L^*([Z])) & = & \pi_{j *}(\widetilde{\phi}^*_n \circ \widetilde{\phi}_{n *}([Z]) - [Z])  =  \pi_{j *} \circ \widetilde{\phi}^*_n \circ \widetilde{\phi}_{n *}([Z]) -  \pi_{j *}([Z]) \\
& = & \widetilde{\phi}^*_j \circ \pi_{j *} \circ \widetilde{\phi}_{n *}([Z]) - \pi_{j *}([Z])  =  \widetilde{\phi}^*_j \circ \widetilde{\phi}_{j *} \circ \pi_{j *}([Z]) - \pi_{j *}([Z]) \\
& {=} & \widetilde{\phi}^*_j \circ \widetilde{\phi}_{j *}(m_j[Z^{(j)}]) - m_j[Z^{(j)}] =  m_j(\widetilde{\phi}^*_j \circ \widetilde{\phi}_{j *}([Z^{(j)}]) - [Z^{(j)}])  =  m_j L^*([Z^{(j)}]),
\end{array}
\]
which proves the second equality.
\end{proof} 

The following complements Lemma \ref{lem:Res-sch-cycle}:

\begin{lem}\label{lem:Res-sch-cycle-1}
We are under the Set-up $(\star)$ of \S \ref{sec:Set-up3}. In particular, $Z \to X$ is an \fs-morphism. Suppose that $Z$ is integral such that $Z \to C_n$ is non-constant. 
Suppose that for an integer $0 \leq j \leq n-1$, the projection $Z^{(j)} \to C_j$ is constant. Then for each $L \in \Gr(N-r-1, H)(k)$ satisfying the conditions of Proposition \ref{prop:Main-II} for $Z$, we have the equality $\pi_j (Z') = \pi_j (Z)$ for each irreducible component $Z'$ of $L^+ (Z)$. 
\end{lem}

\begin{proof}
Toward contradiction, suppose that there is an irreducible component $Z'$ of $L^+(Z)$ such that $\pi_j(Z') \neq \pi_j(Z)$. In particular, this implies that $L^+(Z^{(j)}) \neq \emptyset$. On the other hand, we are given that $Z^{(j)} = \pi_j(Z) = X \times \{c_j\}$ for some closed point $c_j \in \widehat{C}_j$. In this case, $\widetilde{\phi}_{j} (Z^{(j)} ) = \mathbb{A}_k ^r \times \{ c_j \}$ so that $\widetilde{\phi}_j ^{-1} \widetilde{\phi}_j (Z^{(j)}) = X \times \{ c_j \} = Z^{(j)}$. Hence, $L^+ (Z^{(j)}) = \emptyset$. This is a contradiction. 
\end{proof}

\subsection{The presentation lemma}\label{sec:P-lemma*}
We now prove the presentation lemma for residual cycles under linear projections. We are under the Set-up ($\star$) in \S \ref{sec:Set-up3}.

\begin{thm}\label{thm:Main-2*}
Let $k$ be an infinite perfect field. Let $Z \subset X \times C_n$ be an integral closed subscheme such that $Z \to X$ is finite surjective, and the projection $Z \to  C_n$ is non-constant. 

Then there exist an embedding $\eta: \overline{X} \inj \P^N_k$ and a dense open subset $\mathcal{U} \subset \Gr(N-r-1, H)$, where $H :=
 \P^N_k \setminus \A^N_k$, such that for each $L \in \mathcal{U}(k)$, the linear projection $\phi_L: \P^N_k \setminus L \to \P^r_k$ away from $L$ defines a finite surjective morphism $\phi: \overline{X} \to \P^r_k$ satisfying the following properties:
\begin{enumerate}
\item
There exists a Cartesian square
\[
\xymatrix@C.8pc{
X \ar@{^{(}->}[r] \ar[d]_{\phi} & \overline{X} \ar[d]^{\phi} \\
\A^r_k \ar@{^{(}->}[r] & \P^r_k.}
\]
\item $\phi$ is {\'e}tale over an affine open neighborhood of $\phi(\Sigma)$.
\item $\phi(x) \neq \phi(x')$ for every pair $x \neq x'$ in $\Sigma$.
\item The map $k(\phi(x)) \to k(x)$ is an isomorphism for each $x \in \Sigma$.
\item The induced map $Z \to {\phi}_n(Z)$ is birational.
\item The map $L^+(Z) \to X$ is finite surjective.
\item For each $0 \leq j \leq n$, the scheme $\pi_j(L^+(Z))$ is regular at all points lying over $\Sigma$.
\end{enumerate}
\end{thm}

\begin{proof}
Since $Z \to X$ is finite surjective, for each $0 \leq j \leq n$ the morphism $Z^{(j)} \to X$ is also finite surjective. Let $i_0 \in \{ 0, \ldots, n\}$ be the largest integer $i$ such that $Z^{(i)} \subset X \times \{ b \}$ for some closed point $b \in {C}_i$. Note that $Z^{(0)} = X = X \times_k {C}_0$, so such $i_0$ exists. Note also that $i_0 \le n-1$ by our assumption.

Choose a large enough Veronese embedding $\P^m_k \inj \P^N_k$ such that for the composite embedding $\eta: \overline{X} \inj \P^N_k$, and the hyperplane $H = H_{N,0}$ as in Lemma \ref{lem:Ver-affine-map}, there are open dense subsets $\mathcal{U}_j \subset \Gr(N-r-1, H)$ such that each $L \in \mathcal{U}_j(k)$ satisfies Lemma \ref{lem:fs_sansZ} and the conditions (1) $\sim$ (6) of Proposition \ref{prop:Main-II} for $Z^{(j)}$ over all $i_0 +1 \le j \le n$. We let $\mathcal{U} =  \bigcap_{j= i_0+1} ^n \mathcal{U}_j$.

The condition (1) of the theorem automatically follows from our choice of $H$ and Lemma \ref{lem:Ver-affine-map}. The conditions (2) $\sim$ (4) follow directly from conditions (2) $\sim$ (4) of Proposition \ref{prop:Main-II}. The condition (5) follows from Lemma \ref{lem:fs_sansZ}. The condition (6) follows from Lemma \ref{lem:fs-res}.

We prove (7). We have to show that every irreducible component of $\pi_j(L^+(Z))$ is regular at all points lying over $\Sigma$ and no two components of $\pi_j(L^+(Z))$ meet at points lying over $\Sigma$. We first assume that $i_0+ 1 \le j \le n$.

Let $Z'$ be an irreducible component of $L^+(Z)$. Since $j > i_0$, Lemma \ref{lem:Res-sch-cycle} says that $\pi_j(Z')$ is a component of the effective cycle $\pi_{j *}(L^*([Z])) = m_jL^*([Z^{(j)}]) = m_j[L^+(Z^{(j)})]$ with $m_j \ge 1$. Since $Z'$ was arbitrary, it follows that the irreducible components of $\pi_j(L^+(Z))$ are
the same as those of $L^+(Z^{(j)})$. On the other hand, the condition (6) of Proposition \ref{prop:Main-II} (with our choice of $L$) says that $L^+(Z^{(j)})$ is regular at all points lying over $\Sigma$. It follows that each irreducible component of $\pi_j(L^+(Z))$ is regular at all points lying over $\Sigma$, and in particular no two components meet at points lying over $\Sigma$.

If $0 \le j \le i_0$, then Lemma \ref{lem:Res-sch-cycle-1} says that $\pi_j(L^+(Z))$ coincides with $\pi_j(Z)$, which in turn is of the form $X \times \{b\}$ for some closed point $b \in C_j$. In particular, $\pi_j(L^+(Z))$ is irreducible. As $X$ is regular everywhere, in particular at all points lying over $\Sigma$, it follows that $\pi_j(L^+(Z))$ is regular at all points lying over $\Sigma$. This completes the proof of the theorem.
\end{proof}

\subsection{The \sfs-moving lemma}\label{sec:sfs-final}
We now prove the \sfs-moving lemma for additive higher Chow groups of relative $0$-cycles over semi-local $k$-schemes. A similar argument also proves the \sfs-moving lemma for Bloch's higher Chow groups of relative $0$-cycles over semi-local $k$-schemes.

Let $k$ be an infinite perfect field. We apply Theorem \ref{thm:Main-2*} with $\widehat{A}_i := \mathbb{P}_k ^1$ for $0 \leq i \leq n-1$, while $A_0:= \mathbb{A}_k ^1$ and $A_1 = \cdots = A_{n-1} = \square_k ^1$ so that $C_j = B_j = \mathbb{A}_k ^1 \times \square_k ^{j-1}$ for $j \ge 1$. The \sfs-moving lemma for additive higher Chow groups of relative $0$-cycles is the following:

\begin{thm}\label{thm:Main-1*}
Let $R$ be a regular semi-local $k$-scheme essentially of finite type of dimension $r \geq 0$ over an infinite perfect field $k$. Let $V= \Spec (R)$ and let 
$m, n \geq 1$ be integers. Then the canonical map $\TH^n_{\sfs}(V, n; m) \to \TH^n(V, n; m)$ is an isomorphism.
\end{thm}

This theorem is proven in steps. Since $R$ is regular, it is a product of regular semi-local $k$-domains, and each $k$-domain corresponds to a connected component of $\Spec (R)$. Thus we may reduce to the case when $R$ is integral. We also remark that by Proposition \ref{prop:no_closed}, we may assume that $R$ is obtained by localizing an integral smooth affine $k$-scheme at a finite set of closed points. Note that Theorem \ref{thm:Main-1*} is obvious for $r=0$, so we may assume $r \geq 1$. 
We have injective maps (using Lemma \ref{lem:TH_M}), 
$$\TH^n _{\sfs} (V, n;m) \to \TH^n _{ \fs} (V, n; m) \to {\TH}^n _{\Sigma} (V, n; m) \to \TH^n (V, n;m).$$ 
The last arrow is an isomorphism by \cite[Theorem 4.10]{KP3}. We show that the middle arrow is an isomorphism, which we call the \emph{$\fs$-moving lemma}:

\begin{lem}\label{lem:fs-move-*}
The map $\TH^n _{ \fs} (V, n;m) \to {\TH}^n _{\Sigma} (V, n; m)$ is an isomorphism.
\end{lem}

\begin{proof}
By the above discussion, we assume $V$ is integral. Since this map is injective, we only have to show that it is surjective. Let $\gamma \in \TZ^n_{\Sigma}(V, n; m)$ be a cycle with $\partial(\gamma) = 0$. 

First suppose that there is an atlas $(\mathbb{A}^r_k, \Sigma)$ so that $\gamma$ lifts to a cycle $\overline{\gamma} \in \TZ^n_{\Sigma}(\A^r_k, n;m)$. In this case, we can apply Theorem \ref{thm:Ar-spread-case} and write $\gamma = \gamma_1 + \partial(\gamma_2)$, where $\gamma_1 \in  \TZ^n_{\sfs}(V, n; m) \subset  
\TZ^n_{\fs}(V, n; m)$ and $\gamma_2 \in  \TZ^n(V, n+1; m)$. One immediately has $\partial(\gamma_1) = 0$, proving the desired surjectivity in this case.

In general, we write $\gamma = \alpha + \beta $, where no component of $\alpha$ is an fs-cycle and $\beta$ is an fs-cycle. Lemma \ref{lem:Adm-spread} says that there is a connected smooth affine atlas $(X,\Sigma)$ for $V$, and cycles $\overline{\alpha}, \overline{\beta}, \overline{\gamma} \in \TZ^n_{\Sigma}(X,n;m)$ such that $\overline{\alpha}_V = \alpha$, $\overline{\beta}_V = \beta$, $\overline{\gamma}_V= \gamma$, $\overline{\gamma} = \overline{\alpha} +  \overline{\beta}$ and $\partial(\overline{\gamma}) = 0$.

Since no component of $\alpha$ is fs over $V$, it follows that the projection of every component of $\overline{\alpha}$ to $B_n$ must be non-constant. We can therefore apply Theorem \ref{thm:fs-present-intro} to obtain a finite flat map $\phi: X \to \mathbb{A}^r_k$ such that $\alpha$ satisfies all the properties there. Let $\Sigma' = \phi(\Sigma)$, which consists of finitely many closed points of $\mathbb{A}^r_k$. Let $V' = \Spec (\mathcal{O}_{\mathbb{A}^r_k, \Sigma'})$ and $W:= X \times_{\mathbb{A}^r} V'$. We have inclusions $\Sigma \subset V \subset W \subset X$, and a finite flat morphism $\phi: W \to V'$. 

Write $\overline{\alpha}= \overline{\alpha}_1  + \overline{\alpha}_2$, where each component of $\overline{\alpha}_1$ is dominant over $X$ and no component of $\overline{\alpha}_2$ is dominant over $X$. As ${\beta}$ is an fs-cycle over $V$, after shrinking $X$ if needed, $\overline{\beta}$ is an fs-cycle over $X$ along $\Sigma$ by Corollary \ref{cor:fs-fine}.

We now have  
\[
\overline{\gamma}= \overline{\alpha}_1 + \overline{\alpha}_2 + \overline{\beta}=  (\overline{\alpha}_1 - \phi^*_n \phi_{n *}(\overline{\alpha}_1)) +(\overline{\alpha}_2 - \phi^*_n \phi_{n *}(\overline{\alpha}_2)) +  (\overline{\beta} - \phi^*_n \phi_{n *}(\overline{\beta})) + \phi^*_n \phi_{n *}(\overline{\gamma}).
\]

Let $\overline{\alpha}_i ':= \overline{\alpha}_i - \phi^*_n \phi_{n *} (\overline{\alpha}_i)$ for $i=1,2$, and $\overline{\beta}' := \overline{\beta} - \phi^*_n \phi_{n *} (\overline{\beta})$. Since $\overline{\beta}$ is an \fs-cycle on $X$ along $\Sigma$ and $\phi$ is finite, $\phi_{n *} (\overline{\beta})$ is an \fs-cycle over $\A^r_k$ by Lemma \ref{lem:fs fpf}. Since $X \to \mathbb{A}_k ^r$ is flat, by Lemma \ref{lem:fs bc} $\phi^*_n \phi_{n *} (\overline{\beta})$ is an fs-cycle over $X$ along $\Sigma$. In particular, $\overline{\beta}' \in \TZ^n_{\Sigma, \fs}(X,n;m)$. On the other hand, by Theorem \ref{thm:fs-present-intro}, we have $(\overline{\alpha}_2' )_V=0$ and $(\overline{\alpha}_1')_V \in \TZ_{\Sigma, \fs} ^n (V, n;m)$. 

Since $\overline{\gamma} \in \TZ^n _{\Sigma}(X, n;m)$ with $\partial (\overline{\gamma}) = 0$, it follows that $\phi_* (\overline{\gamma}) \in \TZ^n_{\Sigma'} (\mathbb{A}^r_k, n;m)$ with $\partial (\phi_{n *} (\overline{\gamma})) = 0$. By the previous case, there are cycles $\eta_1 \in \TZ^n_{\fs} (V', n;m)$, and $\eta_2 \in \TZ^n(V', n+1;m)$ such that $ j^* ( \phi_{n *} (\overline{\gamma}) ) = \eta _1 + \partial \eta_2$. Equivalently, $\phi_{n *} (\overline{\gamma} _W) = \eta_1 + \partial \eta_2$. Hence, $\phi^*_n \phi_{n *} (\overline{\gamma}_W) = \phi^* (\eta_1) + \phi^*_n (\partial \eta_2) = \phi^*_n (\eta_1) + \partial (\phi^*_n (\eta_2))$. Moreover, $\phi^*_n(\eta_1)$ is an fs-cycle by Lemma \ref{lem:fs bc}. Combining these, we have 
\[
\gamma = (\overline{\gamma})_V = (\overline{\alpha}_1 ')_V + \overline{\beta}'_V + (\phi^*_n (\eta_1))_{V} + \partial  (( \phi^*_n (\eta_2))_{V}) = \gamma_1 + \partial  (( \phi^*_n (\eta_2))_{V}),
\]
where $\gamma_1:= (\overline{\alpha}_1')_V + \overline{\beta}_V' + (\phi^*_n (\eta_1))_{V} \in \TZ^n _{\fs} (V, n;m)$. Since $\partial \gamma = 0$, we also deduce that $\partial \gamma_1 = 0$. This completes the proof of the lemma.
\end{proof}

{\bf Proof of Theorem \ref{thm:Main-1*}.}
We may assume that $V$ is integral. Using Lemma \ref{lem:fs-move-*}, it suffices to show that the map $\TH^n _{\sfs} (V, n;m) \to \TH^n _{\fs} (V, n; m)$ is surjective. Let $\alpha \in \TZ^n_{\fs}(V, n; m)$ be an $\fs$-cycle, which always has $\partial(\alpha) = 0$ by Lemma \ref{lem:No-intersection}. Write $\alpha = \alpha_1 + \alpha_2$, where $\alpha_2 \in \TZ^n _{\sfs} (V, n;m)$, while $\alpha_1 \in \TZ^n _{\fs} (V, n;m)$ but no component of $\alpha_1$ lies in $\TZ^n_{\sfs} (V, n;m)$. Note that $\partial(\alpha_i) = 0$ for $i = 1,2$ by Lemma \ref{lem:No-intersection} again. It is enough to prove that $\alpha_1$ is equivalent to a cycle in $\TZ^n _{\sfs} (V, n;m)$. Replacing $\alpha$ by $\alpha_1$, we may therefore assume that no component of $\alpha$ lies in $\TZ^n _{\sfs} (V, n;m)$. 

Apply Lemma \ref{lem:Adm-spread} to choose a connected smooth affine atlas $(X,\Sigma)$ for $V$ and a cycle $\overline{\alpha} \in \TZ^n_{\Sigma}(X, n;m)$ such that $\partial(\overline{\alpha}) = 0$. If $X\simeq \A^r_k$, we can apply Theorem \ref{thm:Ar-spread-case} to write $\alpha = \beta + \partial(\gamma)$, where $\beta \in  \TZ^n_{\sfs}(V, n; m) \subset \TZ^n_{\fs}(V, n; m)$ and $\gamma \in \TZ^n(V, n+1; m)$. This solves the problem in this case. 

Suppose that $X$ is not an affine space. If $Z$ is a component of $\alpha$ whose projection to $B_n$ is constant, then $Z$ is already 
an \sfs-cycle. But, we supposed no component of $\alpha$ is an \sfs-cycle. Hence, $Z \to B_n$ is non-constant for each irreducible component $Z$. It follows that Lemma \ref{lem:Res-sch-cycle} and Theorem \ref{thm:Main-2*} apply to every component of $\alpha$. Let $\phi: X \to \A^r_k$ be the finite and flat map as in Theorem \ref{thm:Main-2*} and let $\Sigma' = \phi(\Sigma)$. By shrinking $\mathcal{U} \subset \Gr(N-r-1, H)$ if necessary, we can assume that conditions (1) $\sim$ (7) of Theorem \ref{thm:Main-2*} hold for each $L \in \mathcal{U}(k)$ and for each component of $\alpha$. 

Let $V'= \Spec (\mathcal{O}_{\mathbb{A}^r_k, \Sigma'})$ and let $W = X \times_{\mathbb{A}^r_k} V'$. We have inclusions $\Sigma \subset V \subset W \subset X$ and a finite and flat morphism $\phi_\Sigma: W \to V'$ of smooth semi-local $k$-schemes. Let $j: V \to W$ be the localization map.

We can write $\overline{\alpha}_W=  (\overline{\alpha}_W - \phi^*_n \phi_{n *}(\overline{\alpha}_W)) + \phi^*_n \phi_{n *}(\overline{\alpha}_W)$. We have $\partial(\phi_{n *}(\overline{\alpha}_W)) =\phi_{n *}(\partial(\overline{\alpha}_W)) = 0$. By the previous case of affine space atlas, we can write $\phi_{n *}(\overline{\alpha}_W) = \eta_1 + \partial(\eta_2)$, where $\eta_1 \in \TZ^n_{\sfs}(V', n; m)$ and $\eta_2 \in  \TZ^n(V', n+1; m)$. This yields $\phi^*_n \phi_{n *}(\overline{\alpha}_W) = \phi^*_n(\eta_1) + \partial(\phi^*_n(\eta_2))$. Since $\phi: W \to V'$ is finite and {\'e}tale, it follows by Lemmas \ref{lem:fs bc} and \ref{lem:pull-back-sfs} that $\phi^*_n(\eta_1) \in \TZ^n_{\sfs}(W, n;m)$.

It follows from Lemma \ref{lem:Res-sch-cycle} and Theorem \ref{thm:Main-2*} that $j^*(\overline{\alpha}_W - \phi^*_n \phi_{n *}(\overline{\alpha}_W)) \in \TZ^n_{\sfs}(V, n;m)$. Let $\beta = j^*(\overline{\alpha}_W - \phi^*_n \phi_{n *}(\overline{\alpha}_W)) + j^*(\phi^*_n(\eta_1)) \in 
\TZ^n_{\sfs}(V, n;m)$ and $\gamma = j^*(\phi^*_n(\eta_2))  \in \TZ^n(V, n+1; m)$. Then, we get 
\[
\begin{array}{lll}
\alpha & = & j^*(\overline{\alpha}_W) = j^*(\overline{\alpha}_W - \phi^*_n \phi_{n *}(\overline{\alpha}_W)) + j^*\phi^*_n(\eta_1) + j^*(\partial(\phi^*_n(\eta_2))) \\
& = &  j^*(\overline{\alpha}_W - \phi^*_n \phi_{n *}(\overline{\alpha}_W)) + j^*\phi^*_n(\eta_1) + \partial(j^*\phi^*_n(\eta_2)) = \beta + \partial(\gamma).
\end{array}
\]

Since $\partial(\alpha) = 0$, we must have $\partial(\beta) = 0$ as well. This proves the theorem. $\hfill\square$

\bigskip

{\bf Proof of Theorem \ref{thm:Main-3}.}
We take $n \ge 2, \ A_0 = \widehat{A}_0 = \Spec(k)$, $A_i = \square_k$ and $\widehat{A}_i := \mathbb{P}_k ^1$ for $1 \le i \le n-1$ in Theorem \ref{thm:Main-2*}. We now repeat the proof of Theorem \ref{thm:Main-1*} verbatim using Remark \ref{remk:sfs-affine-Chow}. $\hfill\square$

\bigskip

\noindent\emph{Acknowledgments.} 
We thank Spencer Bloch, H\'el\`ene Esnault, Najmuddin Fakhruddin, Bruno Kahn, Marc Levine, and Kay R\"ulling for some conversations related to the question studied here throughout the time we were working on it. We thank the handling 
editors and, the anonymous referee of Algebra \& Number Theory, who read the article with thoroughness and patience, and supplied numerous suggestions, 
which helped the authors in improving the article.

We acknowledge that part of this work was done during JP's visits to TIFR and AK's visits to KAIST, and we thank both the institutions. JP would like to thank Juya and Damy for his peace of mind at home through the time this paper was written. JP was partially supported by the National Research Foundation of Korea (NRF) grant (2015R1A2A2A01004120 and 2018R1A2B6002287) funded by the Korean government (MSIP), and TJ Park Junior Faculty Fellowship funded by POSCO TJ Park Foundation.

\end{document}